\documentclass[11pt,reqno]{amsart}

 \usepackage{amsmath,amsthm,amsfonts}
\usepackage{latexsym,mathabx,txfonts}
\usepackage{amssymb}
\usepackage{bbm}
\usepackage{cite}
\usepackage{graphicx}
\usepackage{hyperref}

\newcommand{\R}{\mathbb{R}}
\newcommand{\N}{\mathbb{N}}

\newcommand*\jb[1]{\left\langle #1 \right\rangle}
\newcommand*\abs[1]{\left| #1 \right|}
\newcommand{\tphi}{\tilde{\phi}}
\newcommand{\tpsi}{\tilde{\psi}}
\newcommand{\tH}{\widetilde{H}}

\newcommand{\Ga}{\Gamma}
\newcommand*\pr{\partial}
\newcommand*\eps{\varepsilon}

\numberwithin{equation}{section}

\newtheorem{thm}{Theorem}[section]

\newtheorem{lem}[thm]{Lemma}
\newtheorem{prop}[thm]{Proposition}

\theoremstyle{remark}
\newtheorem{rem}{Remark}[section]

\begin{document}

\title[Codimension one stability of the catenoid]{Codimension one stability of the catenoid under the
vanishing mean curvature flow in Minkowski space}
\author[R. Donninger]{Roland Donninger}
\address{\'{E}cole Polytechnique F\'{e}d\'{e}rale de Lausanne,
Switzerland}
\email{roland.donninger@epfl.ch}
\author[J. Krieger]{Joachim Krieger}
\address{\'{E}cole Polytechnique F\'{e}d\'{e}rale de Lausanne,
Switzerland}
\email{joachim.krieger@epfl.ch}
\author[J. Szeftel]{J\'er\'emie Szeftel}
\address{Laboratoire Jacques-Louis Lions, Universit\'{e} Pierre et Marie
Curie, Paris, France}
\email{jeremie.szeftel@upmc.fr}
\thanks{J. Szeftel is supported by advanced ERC grant BLOWDISOL}
\author[W. Wong]{Willie Wong}
\address{\'{E}cole Polytechnique F\'{e}d\'{e}rale de Lausanne,
Switzerland}
\email{wongwwy@member.ams.org}
\thanks{Support from the Swiss National Science Foundation for J.
Krieger and W. Wong is gratefully acknowledged.}
\thanks{The authors would like to thank Igor Rodnianski, Sergiu
Klainerman, Demetrios Christodoulou, Wilhelm Schlag, and the anonymous
referees for helpful comments on the manuscript.}


\begin{abstract}
We study time-like hypersurfaces with vanishing mean curvature in the $(3+1)$ dimensional Minkowski
space, which are the hyperbolic counterparts to minimal embeddings of
Riemannian manifolds. The catenoid is a stationary solution of the associated Cauchy problem. 
This solution is linearly unstable, and we show that this instability is the only obstruction 
to the global nonlinear stability of the catenoid. More precisely, we
prove in a certain symmetry class 
the existence, in the neighborhood of the catenoid initial data, of a co-dimension 1 Lipschitz
manifold transverse to the unstable mode consisting of initial data
whose solutions exist globally in time and converge asymptotically to the catenoid. 
\end{abstract}

\maketitle

\tableofcontents


\section{Introduction}
\label{sec:introhm}

We study here \emph{extremal hypersurfaces} embedded
in the $(1+3)$-dimensional Minkowski space $\mathbb{R}^{1,3}$. More precisely, 
we consider for a three-dimensional smooth manifold $M$ the embeddings 
$\Phi:M \to \mathbb{R}^{1,3}$ such that $\Phi(M)$ has vanishing mean curvature, 
and such that the pull-back metric has Lorentzian signature. 
We will consider the associated \emph{Cauchy problem}.  
Given a two-dimensional smooth manifold $\Sigma$ and two maps  
$\Phi_0:\Sigma\to\mathbb{R}^3$ and $\Phi_1:\Sigma \to
\mathbb{R}^3$, we can ask for the existence and uniqueness of an
interval $I = (T_0,T_1)\ni 0$ and a map $\Phi:I\times \Sigma \to
\mathbb{R}^{1,3}$ such that $\Phi(I\times\Sigma)$ has vanishing mean
curvature, $\Phi: \{t\}\times\Sigma \to
\{t\}\times \mathbb{R}^3$, and the initial conditions 
$\Phi|_{\{0\}\times\Sigma} = (0,\Phi_0)$ and
$\partial_t\Phi|_{\{0\}\times\Sigma} = (1,\Phi_1)$ are satisfied.
Observe that with the knowledge of $\Phi_0,\Phi_1$ it is possible to
compute the pullback metric of $I\times \Sigma$ along
$\{0\}\times\Sigma$. As it turns out, as long as the pullback metric
is Lorentzian, the \emph{quasilinear} system of equations for the
extremal hypersurface is second order regularly hyperbolic
\cite{Chr00, Won11}, and local well-posedness for smooth initial data 
holds (see \cite{ACh79, KrLi}). It is then natural to consider the large time 
behavior of the flow. 

Note that not all solutions are global as there are known finite time 
blow up dynamics. Let us for example exhibit 
a large but compactly supported perturbation of the catenoid (see
following paragraphs; also Section \ref{sect:formulation}) initial data 
which becomes singular in finite time. If we start with the
initial data given by the standard infinite cylinder, that is
\[\Phi_0(\omega,y) = ( \cos\omega, \sin\omega, y),\,\,\Phi_1(\omega,y) = 0,\]
we see that the equations of motion
reduce to an ordinary differential equation in time for the radius 
$R$ of the cylinder
\[ R R'' = -1 + (R')^2 \]
which leads to the explicit solution
$$\Phi(t,\omega,y) = (t, \cos t\cos\omega, \cos t\sin\omega, y).$$
This solution is singular at time $t = \pi/2$, as $\mathrm{d}\Phi |_{t
= \pi/2}$ is no longer injective: the cylinder has collapsed to the
$x_3$ axis. (Similar collapse results under symmetry
assumptions has been studied by Aurilia-Christodoulou \cite{ACh79a}
and can also be derived more generally from the result of Nguyen-Tian
\cite{NgTi13}.) It is clear that such blow-ups do not depend on the
asymptotics, as $y\to \pm\infty$, of the initial data, due to the finite 
speed of propagation property of quasilinear wave equations. 
We can spatially localize the blow-up by taking initial data which 
smoothly glues a compact portion of the cylinder of length (measured 
in the $y$ direction) at least 2 to the catenoid (with a corresponding
compact region excised),  
analogously to the construction of compactly supported 
blow-up solutions to the focusing semilinear wave equation 
$\Box u = - |u|^{p-1} u$ from the ODE blow-up mechanism.

On the other hand, a particular class of initial data which admits \emph{global}
solutions are those for which $\Phi_0:\Sigma \to\mathbb{R}^3$ is the
embedding for a \emph{minimal surface}, and $\Phi_1 = 0$. It is easily checked that the map $\Phi(t,p) = (t,\Phi_0(p))$ embeds
$\mathbb{R}\times\Sigma$ into $\mathbb{R}^{1,3}$ with zero mean curvature, and 
 $\partial_t\Phi = (1,0,0,0)$ implies that the pullback metric is
Lorentzian. We consider in this paper the problem of \emph{stability} of these stationary solutions.  
The first explicit consideration of a problem of this sort is due to 
Brendle (in higher dimensions) \cite{Br} and Lindblad
\cite{Lin04}, building upon earlier works \cite{Kl, Chr86} on global
existence of general quasilinear wave equations. They consider small 
perturbations of the
stationary solution given by a flat hyperplane. One can then write the
solution as a graph over the stationary background, and reduce the
problem to the small data problem for a \emph{scalar} quasilinear 
wave equation satisfying both the quadratic \cite{Kl, Chr86} and cubic
\cite{Ali01, Ali01a} null
conditions.

In this paper we will consider the problem of stability for a non trivial
stationary background. Our work is in the spirit of recent studies of asymptotic stability of solitary waves for \emph{semilinear} 
wave equations (see for example \cite{KrS,BKrT,BecP12,NaSc11,NaSc12}; 
see also \cite{RoSt, RaRo, MeZa, HiRa12} for finite time blow up regimes which correspond to asymptotic 
stability in suitable rescaled variables), but in a \emph{quasilinear} setting. The
background solution we choose is the \emph{catenoid}, which is an
embedded minimal surface in $\mathbb{R}^3$, and is a surface of
revolution with topology $\mathbb{S}^1\times \mathbb{R}$. The induced
Riemannian metric on $\Sigma$ at a fixed time for this stationary
solution is asymptotically flat (with two ends). This fact is
important in our analysis. Indeed, as it is clear from the study by Brendle
and Lindblad, to prove any sort of global existence statement we need
to exploit the \emph{pointwise radiative decay} of solutions to the
\emph{linearized equation} on our background manifold. In \cite{Lin04}
the linearized equation is exactly the linear wave equation on
$\mathbb{R}^{1,2}$, and the pointwise decay utilized is the classical
one. In our case, the linearized equation is a geometric wave equation
on the curved background $\Sigma$ \emph{with a potential term}. 
The asymptotic flatness of $\Sigma$ thus plays an important 
role in establishing a decay mechanism. 

As mentioned above, a significant difference with the small data cases 
considered by Lindblad and Brendle is that
the linearized equation is no longer the linear wave
equation on the background manifold $\mathbb{R}\times\Sigma$; it 
also contains a potential term. In addition to introducing complications when
applying the vector-field method to obtain decay, the potential term
turns out to have the ``wrong sign''. That is to say, the linearized
equation admits an exponentially growing mode. As observed
by Krieger-Lindblad \cite{KrLi}, if one isolates the perturbation away
from the ``collar region'' (see Figure \ref{fig1}), one can verify that the solution exists ``up
to the time when the collar begins to move'' (due to finite speed of
propagation). One should interpret this restriction as when the
exponentially growing mode (which is very small initially) \emph{overtakes}
the radiating parts of the perturbation in size. In view of this exponentially growing mode, we cannot obtain stability
for arbitrary perturbations. Similar to the
analysis of Krieger-Schlag \cite{KrS} for the semilinear wave
equation, we will prove (for a more precise statement, see Section
\ref{sec:statement})
\begin{thm}[Codimension one stability of the catenoid, version 1]
Let $g_d$ be the unstable mode for the linear evolution. 
For any sufficiently small initial perturbation $(\Phi_0, \Phi_1)$ 
of catenoid initial
data, we can find a small number $\alpha$ such that solution generated
by the initial data $(\Phi_0 + \alpha g_d, \Phi_1)$ exists for all
future time and decays asymptotically to the catenoid. 
\end{thm}

\begin{rem}
While here we explicitly consider the
case of embedding a hypersurface in $\mathbb{R}^{1,3}$, the method
should easily carry over to the case where the ambient Minkowski
space has higher dimensions, as pointwise decay estimates (Section
\ref{sect:linear}) improve in higher dimensions, making the
nonlinear analysis (Section \ref{subsec:nonlin}) simpler. Furthermore
the spectral properties of the linearized operators (Section
\ref{subsec:lin}) are qualitatively the same independently of the dimension.
\end{rem}

This paper is organized as follows. In Section \ref{sect:mainres} we
introduce the equation which we will study, discuss some of its main
features, describe the linear theory, and state our main theorem. In Section \ref{sec:boot} we
describe the bootstrap argument which will be used to prove our main
theorem. In Sections \ref{sec:energybounds} through
\ref{sec:dispersive}, we improve on our bootstrap assumptions under the 
assumption that the projection of our solution on the unstable mode is
under control. In Section \ref{sec:existencea} we improve our control
of the unstable mode. Finally, we prove our main theorem in Section \ref{sec:proofmaintheorem}. 

\begin{figure}[t]
\centering
\includegraphics[width=\textwidth]{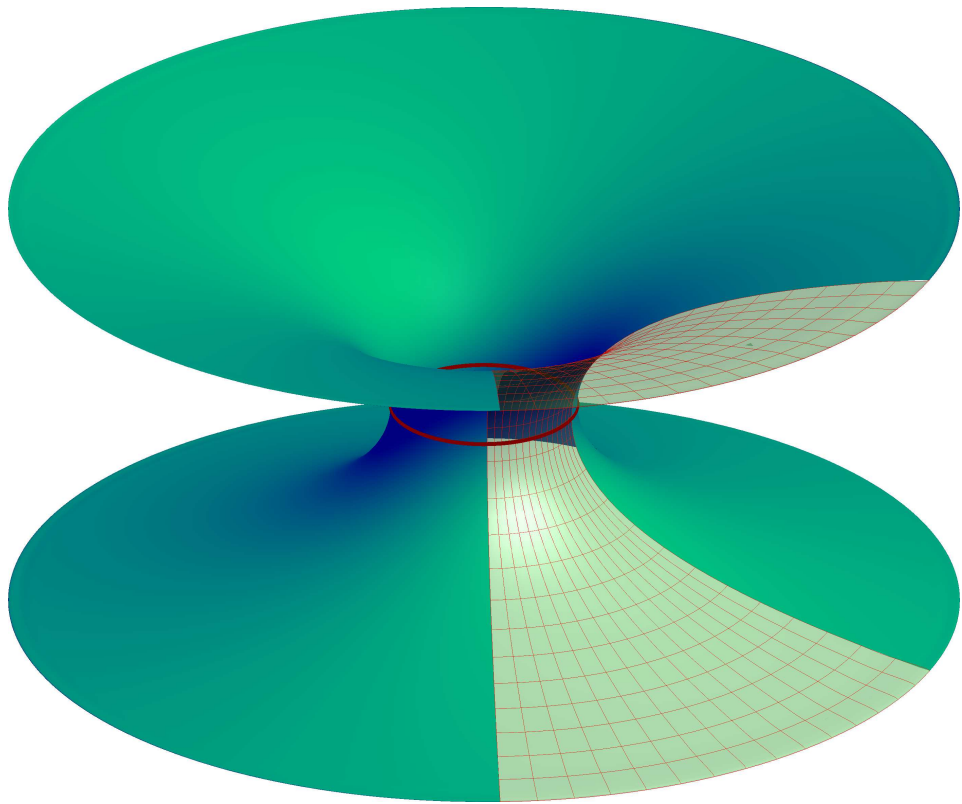}
\caption{The catenoid surface with the ``collar'' (thinnest part of
the surface of revolution) marked out. In the transparent portion we
can see the level sets of the angular coordinate $\omega$ as well as the
``radial'' coordinate $y$.}
\label{fig1}
\end{figure}


\section{Main Results}\label{sect:mainres}


\subsection{Formulation of the problem}\label{sect:formulation}


As mentioned above, we consider perturbations of the stationary
catenoid solution to the extremal surface equation. 
The catenoid as a surface of revolution can be parametrized by (see
also Figure \ref{fig1})
\begin{equation}\label{eq:stdcat}
\mathbb{R}\times\mathbb{S}^1 \ni (y,\omega) \mapsto \left( r = \sqrt{1
+ y^2}, z = \sinh^{-1} y, \theta = \omega\right) \in \mathbb{R}_+ \times
\mathbb{R}\times \mathbb{S}^1, 
\end{equation}
where we use the standard cylindrical coordinates system on $\mathbb{R}^3$. 
Throughout we use the notation $\jb{y} = \sqrt{1+y^2}$. 
The parametrization here exposes the catenoid, a surface of
revolution, as a warped product manifold with base $\mathbb{R}$ and fibre 
$\mathbb{S}$; the coordinate $y$ is chosen to be orthogonal to the fibers 
and to have unit length (note that the parametrization is ``by 
arc length'' if we ``mod'' out the rotational degree of freedom). In 
 this coordinate system we see that the
induced Riemannian metric on the catenoid has the line element
\[ \mathrm{d}y^2 + \jb{y}^2 \mathrm{d}\omega^2~,\]
and that $\jb{y}/ \abs{y} \to 1$ as $y\to \pm\infty$ captures the
asymptotic flatness of this manifold. 

In addition to the rotational symmetry, the catenoid
also has a reflection symmetry about the plane $z = 0$; in terms of
the intrinsic coordinates, this is the mapping $y\mapsto -y$. \emph{For
simplicity, we will consider only perturbations that preserve both 
symmetries.} More precisely, we will consider the case where the
perturbed solution is still, at any instance of time, a surface of
revolution that is symmetric about the plane $z = 0$. Note that 
since the induced Riemannian metric on $\Sigma$ is asymptotically flat with two ends, 
 the Hamiltonian flow on $\mathbb{R}\times \Sigma$ using 
the pullback metric exhibits trapping, which is manifest in the closed geodesic at the
``collar'' of $\Sigma$ (see Figure \ref{fig1}). The rotational symmetry reduces our scenario to
the ``zero angular momentum case'', and hence issues associated with
the trapping of the geodesic flow do not appear in our analysis. A
treatment of the full problem, without rotational symmetry, will require 
a modification of some parts of our proof which rely explicitly on the 1+1 reduction of the 
problem in rotational symmetry, as well as a detailed study of the trapping phenomenon, 
which usually induces a loss of derivatives. 
On the other hand, the reflection symmetry is only used to simplify 
the analysis by effectively fixing the centre of mass; we do not 
expect there to be obstructions in removing this assumption given
finite speed of propagation for nonlinear wave equations. 

\begin{figure}[t]
\centering
\includegraphics[width=0.9\textwidth]{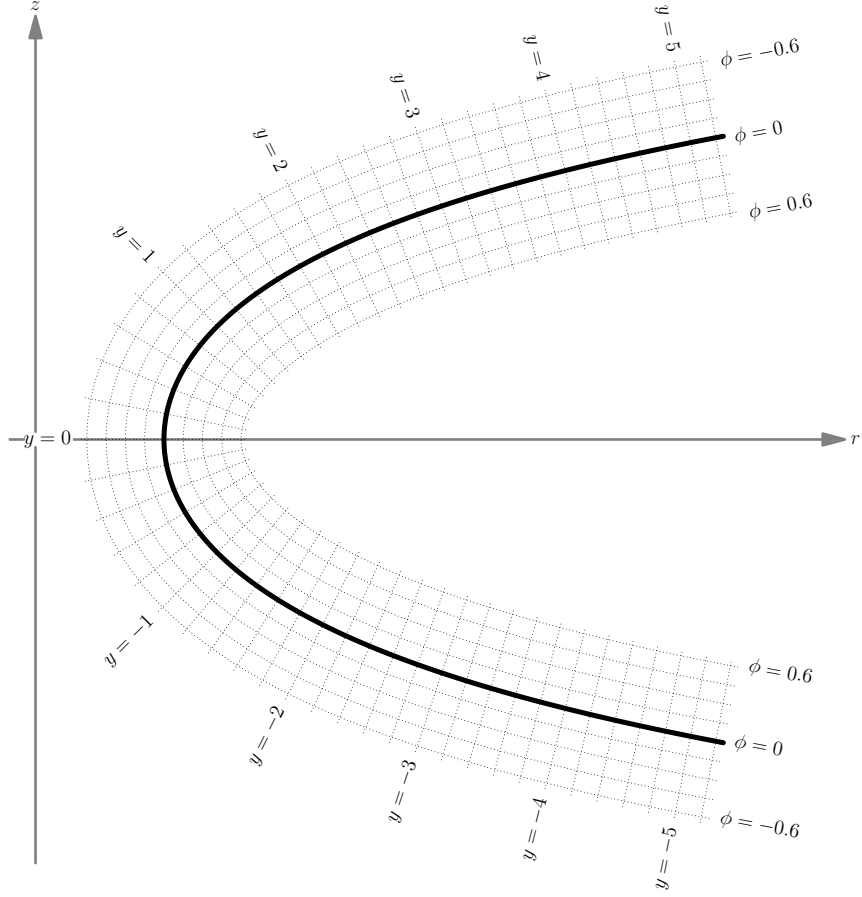}
\caption{The local coordinate system induced by the normal bundle of
the catenoid. We show here the cross section for a fixed $\omega$.}
\label{fig2}
\end{figure}

Given the geometric nature of our problem, there are many different
ways of \emph{parametrizing} our solution manifold $M$ (or
equivalently, fixing the time parameter $t$, parametrizing the time
slices). To cast the problem as a concrete system of partial
differential equations requires choosing a \emph{gauge} (in other
words, fixing a preferred parametrization; this problem is typical 
for geometric equations such as the Ricci flow
or Einstein equations). Given the assumed
symmetries one may be tempted into a \emph{geometric} gauge choice via
intrinsic quantities: for example, the rotational symmetry means that
naturally $\omega$ is a good candidate coordinate, and we may want to
choose the other coordinate $y$ of $\Sigma$ to be orthogonal to $\omega$
and of unit length, similar to our parametrization of the catenoid.
This choice turns out to be not suitable for studying the stability
problem as the equation for the \emph{difference} between our
perturbed solution and the stationary catenoid becomes a
complicated equation for a vector-valued function with a compatibility
constraint (coming from the ``unit-length'' requirement). By using the
compatibility constraint one can convert this to a scalar
\emph{non-local} integro-differential equation. 

Since we are interested in the stability problem in the rotationally
symmetric case, instead we will
consider our perturbed solution as a \emph{graph} over the catenoid.
More precisely, using that the outward-pointing unit normal vector
field to the
catenoid is $\jb{y}^{-1}(\partial_r - y \partial_z)$, there is a
natural smooth \emph{surjection} from the 
normal bundle of the catenoid to $\mathbb{R}^3$, given by (see
Figure \ref{fig2})
\begin{equation}\label{eq:coordchoice}
(y,\omega,\phi)\mapsto \left( r = \jb{y} + \frac{\phi}{\jb{y}}, z =
\sinh^{-1} y - \frac{y\phi}{\jb{y}}, \theta = \omega\right) ~.
\end{equation}
By considering the radius of curvature for the constant $\omega$ level
curves, we see that restricted to $\abs\phi < \jb{y}^2$ this mapping is regular
and injective. Since we 
are interested in perturbations of the $\phi =
0$ level set, we make the assumption that our perturbed solution can
be written as a graph over $\{\phi = 0\}$ in this coordinate system.
That is to say, we will study the \emph{small data problem} for 
$\phi = \phi(t,y)$. Note that our assumption of reflection symmetry
implies that $\phi$ will be an even function in $y$, and the lack of
$\omega$ dependence indicates that the graph is a surface of
revolution. 

Under this parametrization, we can derive the equation of motion for
the extremal surface by formally writing down the Euler-Lagrange
equations for the Lagrangian given by the induced volume form on the
graph associated to $\phi(t,y)$; this computation is carried out in
Appendix \ref{app:reduction}. We find that the equation of motion can
be written as a quasilinear wave equation \emph{with potential} for
$\phi$ in the coordinates $t,y$:
\begin{equation}\label{eq:Main}
-\partial_{tt}^2\phi + \partial_{yy}^2\phi +
\frac{y}{\jb{y}^2}\partial_y\phi + \frac{2}{\jb{y}^4}\phi 
= Q_2 + Q_3 + Q_4 + S_2 + S_3 + S_4, 
\end{equation}
where the quasilinear terms $Q_{*}$ and semilinear terms $S_{*}$ are
split into those quadratic, cubic, and quartic-or-more in $\phi$ and
its derivatives:
\begin{subequations}\label{eqs:mainQuasi}
\begin{align}
Q_2 &= - \frac{2\phi}{\jb{y}^2} \partial^2_{tt} \phi,\\
Q_3 &= \frac{\phi^2}{\jb{y}^4} \partial^2_{yy}\phi +
(\partial_t\phi)^2 \partial^2_{yy}\phi - 2 \partial_t \phi \partial_y
\phi \partial^2_{ty}\phi + (\partial_y\phi)^2 \partial^2_{tt}\phi,\\
Q_{4} &= \frac{\phi^2}{\jb{y}^4} \left[\left( \frac{2\phi}{\jb{y}^2} -
\frac{\phi^2}{\jb{y}^4} - (\partial_y\phi)^2\right)
\partial^2_{tt}\phi + 2 \partial_y\phi \partial_t\phi
\partial^2_{ty}\phi - (\partial_t\phi)^2 \partial^2_{yy}\phi \right],
\end{align}
\end{subequations}
and
\begin{subequations}\label{eqs:mainSemi}
\begin{align}
S_2 &= \frac{4\phi^2}{\jb{y}^6} + \frac{4y\phi \partial_y
\phi}{\jb{y}^4} - \frac{(\partial_y\phi)^2}{\jb{y}^2}, \\
S_3 &= \frac{y\phi^2}{\jb{y}^6} \partial_y \phi - \frac{2\phi^3}{\jb{y}^8} - \left(
\frac{3\phi}{\jb{y}^4} + \frac{y
\partial_y\phi}{\jb{y}^2}\right)(\partial_y\phi)^2  + \left(
\frac{2\phi}{\jb{y}^4} + \frac{y
\partial_y\phi}{\jb{y}^2}\right)(\partial_t\phi)^2,
\\
S_{4} &= - \left(\frac{4y\phi}{\jb{y}^4} + \frac{y\phi^2}{\jb{y}^6}\right) \partial_y \phi (\partial_t\phi)^2
-\left(\frac{4\phi^2}{\jb{y}^6} -
\frac{2\phi^3}{\jb{y}^8}\right) (\partial_t\phi)^2.
\end{align}
\end{subequations}
We denote by $F$ this nonlinearity 
\begin{equation}
F(y,\phi,\nabla\phi,\nabla^2\phi) = Q_2 + Q_3 + Q_4 + S_2 + S_3 +
S_4~.
\end{equation}


\subsection{A first look at the structure of the equation}
\label{sec:firstlook}

Let us point out some of the main features of the  
equations \eqref{eq:Main}, \eqref{eqs:mainQuasi}, and 
\eqref{eqs:mainSemi}. That our argument can control the nonlinear
terms using pointwise decay estimates is largely due to two special 
structures: the terms are either
\emph{localized} or they exhibit a \emph{null condition}. We comment on 
these structures in Section \ref{subsec:nonlin}. The linear evolution 
introduces additional difficulties as there is an exponentially
growing mode: the pointwise decay estimates we need can only be expected away
from the growing mode. This is discussed in Section \ref{subsec:lin}. 

\subsubsection{Nonlinearities}\label{subsec:nonlin}

The reason that we
separated the quadratic, cubic, and quartic-and-higher nonlinearities
is that we intend to make use of the \emph{radiative decay} effects of the
wave equation on a (2+1) dimensional, asymptotically flat space-time
to gain decay in the ``wave zone'', the region where $y$ and $t$ are
comparable. The experience with small-data, quasilinear wave equations
on $\mathbb{R}^{1,2}$, see \cite{Lin04,Ali01,Ali01a}, indicates that the most dangerous terms are
those which are quadratic and cubic in the nonlinearities, due
to the expected linear decay rate of $1/\sqrt{t}$ 
for wave equations in 2 spatial dimensions (see also Section
\ref{sect:linear}).  

On the other hand, in \eqref{eqs:mainQuasi} and
\eqref{eqs:mainSemi}, almost all the nonlinear terms, in particular
all the \emph{quadratic ones}, gain an
additional boost in decay from the
coefficients of the form $\jb{y}^{-k}$ --- in the wave zone this term
contributes a decay rate of $t^{-k}$ which vastly improves the
situation. The term $Q_2$, for example, has the form
$O(t^{-5/2})\cdot \partial^2_{tt}\phi$ with a coefficient which is much better
than the integrability threshold of $O(t^{-1})$, if we assume an
expected linear decay rate. As we shall see in the analysis, this
\emph{localization} of some of the most dangerous nonlinearities plays
a crucial role in allowing us to close our decay estimates. 

The only exception to this boost in decay occurs in
the term $Q_3$: there we have a non-linearity of the form 
\begin{equation}\label{eq:potdanger} 
(\partial_t\phi)^2 \partial^2_{yy}\phi - 2 \partial_t\phi
\partial_y\phi \partial^2_{ty}\phi + (\partial_y\phi)^2
\partial^2_{tt}\phi~
\end{equation}
which is unweighted. However, as was observed in \cite{Lin04} for the
perturbation of the trivial solution, this term carries a \emph{null
structure}. One can see this purely at an algebraic level: in terms of
the \emph{background null coordinates} $u = t + y$ and $v = t-y$,
the nonlinearity takes the form 
\[ 4(\partial_v\phi)^2 \partial^2_{uu}\phi + 4(\partial_u\phi)^2
\partial^2_{vv} \phi - 8 \partial_u\phi\partial_v \phi \partial^2_{uv}
\phi \]
and hence verifies the \emph{cubic, quasilinear null condition} \cite{Ali01}. 
The null condition exhibits in particular 
a hidden divergence/gradient structure: in the context of 
elliptic theory it appears in the proof of Wente's inequality 
 \cite{Wen69}; and in the context of wave equations it drives
the null form estimates of Klainerman and Machedon \cite{KlMa94}. For
our explicit nonlinearity above, one can check easily 
that the following identity holds 
\begin{multline*}
(\partial_t\phi)^2 \partial^2_{yy}\phi  - 2\partial_t\phi \partial_y
\phi \partial^2_{ty}\phi  + (\partial_y\phi)^2 \partial^2_{tt}\phi\\
= \partial_t\big[(\partial_t\phi)^3\big] -
2\partial_y\big[\partial_y\phi(\partial_t\phi)^2\big] +
\partial_t\big[(\partial_y\phi)^2\partial_t\phi\big] +
3(\partial_t\phi)^2(\partial^2_{yy}\phi - \partial^2_{tt}\phi).
\end{multline*}
The first three terms of the right-hand side exhibit the hidden divergence
structure, while for the last term, we may replace $-\partial^2_{tt}\phi+\partial^2_{yy}\phi$ 
using our original equation \eqref{eq:Main} and hence obtain terms which are cubic \emph{with
sufficient weights} together with quartic and higher terms which have better decay properties. 

\subsubsection{Linear spectral analysis}\label{subsec:lin}

Having described the difficulties that arise from the ``right hand
side'' of \eqref{eq:Main}, we turn our attention to the ``left hand
side''. The linear operator 
\begin{equation} \label{eq:linearwaveop}
- \partial^2_{tt} \phi + \partial^2_{yy}\phi + \frac{y}{\jb{y}^2}
\partial_y \phi \end{equation}
is in fact the coordinate-invariant wave operator $\Box_M \phi$ on the
background $\mathbb{R}\times \Sigma$. Indeed, the induced Lorentzian
metric on the stationary catenoid solution, as an embedded
hypersurface of $\mathbb{R}^{1,3}$, is
\[ - \mathrm{d}t^2 + \mathrm{d}y^2 + \jb{y}^2 \mathrm{d}\omega^2~, \]
and its corresponding Laplace-Beltrami operator can be computed to be
exactly \eqref{eq:linearwaveop}. However, since we are considering the
perturbation of a non trivial solution, there is also a lower
order correction term generated by the linearization, namely 
the potential term $2\jb{y}^{-4} \phi$ on the left hand side of \eqref{eq:Main}. 
Note that the coefficient $2 \jb{y}^{-4}$ has a
positive sign, which indicates that it is an \emph{attractive}
potential, and opens up the possibility of the existence of a negative
energy ground state. This is related to the variational instability of
the catenoid as a minimal surface \cite{FiSc80}: indeed, geometrically
we can write the linear operator on the left hand side of
\eqref{eq:Main} as 
\[ \Box_g - \mathrm{Scalar} \]
where $\Box_g$ is the Laplace-Beltrami operator for the induced
Lorentzian metric on the static catenoid solution, and
$\mathrm{Scalar}$ is the induced scalar curvature. This operator is
formally the second variation of the extremal surface Lagrangian, and
in the Riemannian case is precisely the \emph{stability operator} for
minimal hypersurfaces embedded in Euclidean spaces. The positivity of
the potential term and the linear instability is then seen as a
consequence of minimal surfaces in $\mathbb{R}^3$ necessarily having
\emph{negative} Gaussian curvature. Any corresponding 
eigenfunction of the linearized operator
will generate either non-decaying or exponentially growing modes; 
clearly this will complicate our estimates based on expectation of linear 
pointwise decay. 

Now, the natural space on which to study our linear operator is the
$L^2$ space adapted to the geometry; that is to say, we should be
looking at $L^2(\Sigma)$ where $\Sigma$ is the catenoid. In the
intrinsic coordinates $(y,\omega)$ this is $L^2( \jb{y}
\mathrm{d}y\,\mathrm{d}\omega)$. Since we are working with rotationally
symmetric functions, we find it convenient to absorb the weight
$\jb{y}$ onto the function $\phi$ instead, and work with
$L^2(\mathrm{d}y)$. In other words we introduce the notation
\[
\tilde{\phi}: = \jb{y}^{\frac{1}{2}}\phi
\]
and we obtain in place of \eqref{eq:Main} the following equation:
\begin{equation}\label{eq:Maintilde}
-\partial_t^2\tilde{\phi} + \partial_y^2\tilde{\phi} +
\frac{6 + y^2}{4\jb{y}^4}\tilde{\phi} = \jb{y}^{\frac{1}{2}}F(\phi, \nabla\phi, \nabla^2\phi).
\end{equation}
Thus, we are now working with the standard
$L^2(\mathrm{d}y)$ space and on this space the relevant linear
operator
\begin{equation}\label{eq:LinOp}
\mathcal{L}: = -\partial_y^2 - \frac{6+y^2}{4\jb{y}^4}
\end{equation}
is a short-range perturbation of the Laplacian. Since the potential
term is a bounded multiplier which decays to $0$ as $\abs{y}\to
\infty$, the operator $\mathcal{L}$ is
self-adjoint on $L^2(\mathrm{d}y)$ with domain $\{\partial^2_{yy} f
\in L^2(\mathrm{d}y)\}$, and its essential spectrum is exactly
$[0,\infty)$ (this result is classical, see e.g.\ \cite[Sections 13.1
and 14.3]{HiSi96}). Due to the $O(\jb{y}^{-2})$ decay of the potential
term, the solutions to the ordinary differential equation $(\mathcal{L} -
\lambda)\eta_\lambda = 0$ for $\lambda > 0$ are 
given by the Jost solutions \cite[Theorem XI.57]{ReSiIII}, and hence
there are no $L^2$ eigenfunctions with positive eigenvalue. 

In the case $\lambda = 0$, the equation $\mathcal{L} \eta_0 = 0$ can
be solved explicitly: this is simply due to the fact that
$\mathcal{L}$ is the natural \emph{linearized} operator for the
minimal surface embedding problem, and that after fixing rotational
symmetry, the catenoid solutions form a two parameter family due to
the freedoms for scaling and translating (along the axis).  To be more
precise, the standard catenoid we choose in \eqref{eq:stdcat} is the element of the family 
\begin{equation}\label{eq:catfam}
(y,\omega) \mapsto \left(r = a \jb{a^{-1}y}, z = b + a \sinh^{-1}(a^{-1}y),
\theta = \omega\right)~,
\end{equation}
parametrized by $(a,b)\in \mathbb{R}_+\times \mathbb{R}$, with $a = 1$
and $b = 0$. The two linearly independent solutions to
$\mathcal{L}\eta_0 = 0$ correspond to infinitesimal motions in $a$ and
$b$ of the above. From this consideration it is clear that the
movement in $b$ corresponds to an \emph{odd} solution (and so ruled
out by our symmetry assumptions) with a unique
root at $y = 0$, while movement in $a$ corresponds to an \emph{even}
solution with two roots. We can easily obtain the explicit form of
these two solutions by formally taking derivatives relative to $a,b$
after expressing \eqref{eq:catfam} in the coordinates
\eqref{eq:coordchoice}. This yields
\begin{equation}\label{eq:zeroeigen} 
\eta_0 = \jb{y}^{\frac12} \cdot \begin{cases}
\frac{y}{\jb{y}} \sinh^{-1} y  - 1 & \text{(scaling symmetry in } a),\\
\frac{y}{\jb{y}} & \text{(translation symmetry in } b).
\end{cases} \end{equation}
One sees easily from the asymptotic behavior that neither of these
functions belong to $L^2(\mathrm{d}y)$. 

\begin{rem}
The fact that the solutions $\eta_0$ do not belong to $L^2(\mathrm{d}y)$ implies
that we do not have to modulate. In other words, the individual
elements of our two parameter family \eqref{eq:catfam} are
``infinitely far'' from one another (this can be seen from their
asymptotic behavior) and we do not need to track the
``motion along the soliton manifold'' for our analysis. 
\end{rem}

We lastly consider the
possible discrete spectrum below 0. By testing with bump functions we
easily see that there must be a negative eigenvalue. By the
Sturm-Picone comparison theorem \cite[Section 10.6]{BiRo89} and the
explicit solutions \eqref{eq:zeroeigen} above, we see
that the eigenvalue is unique, and its eigenfunction is nowhere
vanishing (it is the ground state). We call this eigenfunction
$g_d(y)$ and its associated eigenvalue $-k_d^2$. (Numerically 
$\sim -0.5857$.)
Note that $g_d$ is smooth, and decays exponentially as
$\abs{y}\to\infty$. 

In the sequel we let $P_d$ denote the projection onto the ground
state $g_d$, and $P_c$ the projection onto the continuous spectrum.
Noting that $g_d$ contributes an exponentially growing mode to the
linear evolution, we cannot expect to have stability for any 
perturbation. Instead, we will show that given a sufficiently small
initial perturbation $\tilde{\phi}$, we can adjust its projection to
the ground state $P_d\tilde{\phi}$ while keeping $P_c\tilde{\phi}$
unchanged so as to guarantee global existence and asymptotic 
vanishing of the solution. In the analysis we will treat the continuous part and the discrete part of
the spectrum separately. We will describe the linear decay 
estimates for the continuous part of the solution in Section
\ref{sect:linear}. This will be combined with the analysis of the nonlinear
terms (in the spirit of Section \ref{subsec:nonlin}) to derive \emph{a
priori estimates} assuming that the discrete part of the solution is
well behaved. Finally we will close the argument in Section
\ref{sec:existencea} by showing that such a good choice of initial
$P_d\tilde{\phi}$ is possible. 


\subsection{Energy bounds and pointwise decay estimates for
$\mathcal{L}$}\label{sect:linear} 


For the sequel, we shall use the following key energy and decay
estimates associated with the evolution of the operator $\mathcal{L}$,
which are proved in \cite{DoKr}. Recall that we take $P_c =
1 - P_d$ to be the projection to the continuous part of the spectrum
of $\mathcal{L}$. In the sequel, we shall frequently use the notations (as well as variations thereof)
\[
\|\langle\nabla_{t,y}\rangle^{\alpha}\psi\|_{S},\,\,\|\langle\nabla_{t,y}\rangle^{\alpha}\langle \Gamma\rangle^{\kappa}\psi\|_{S}
\]
for various norms $\|\cdot\|_S$. By these expressions we shall understand the quantities 
\[
\sum_{0\leq|\beta|\leq \alpha}\|\nabla_{t,y}^{\beta}\psi\|_{S},\,\,\,\sum_{0\leq|\beta|\leq \alpha}\sum_{0\leq\tilde{\kappa}\leq\kappa}\|\nabla_{t,y}^{\beta}\Gamma^{\tilde{\kappa}}\psi\|_{S}.
\]
Here $\Gamma$ stands for either one of the vector fields $\Gamma_1 = t\partial_y + y\partial_t$, $\Gamma_2 = t\partial_t + y\partial_y$. 
\begin{prop}\label{eq:keyest1}For any multi-index $\alpha = (\alpha_1, \alpha_2)\in \N_{\geq 0}^2$, we have 
\begin{equation}\label{eq:plainenergy1}
\|\nabla_{t,y}^{\alpha}P_c e^{it\sqrt{\mathcal{L}}}f\|_{L^2_{dy}}\lesssim \|\langle\partial_y\rangle^{|\alpha|}f\|_{L^2_{dy}}
\end{equation}
with constant depending on $|\alpha| = \alpha_1 + \alpha_2$. Moreover, denoting the scaling vector field 
\[
\Gamma_2: = t\partial_t + y\partial_y, 
\]
we have for any $\alpha\in \N_{\geq 0}$, $\beta\in \N_{\geq 0}^2$ the weighted energy bounds 
\begin{equation}\label{eq:weightedenergy1}
\|\nabla_{t,y}^{\beta}\Gamma_2^{\alpha}P_c e^{it\sqrt{\mathcal{L}}}f\|_{L^2_{dy}}\lesssim \|\langle y\rangle^{\alpha}\langle\partial_y\rangle^{|\beta|+\alpha}f\|_{L^2_{dy}}.
\end{equation}
For the sine evolution, we have the following bounds for $|\alpha|\geq
1$: for any $\eps_* \geq 0$ 
\begin{equation}\label{eq:plainenergy2}
\|\nabla_{t,y}^{\alpha}P_c \frac{\sin(t\sqrt{\mathcal{L}})}{\sqrt{\mathcal{L}}}f\|_{L^2_{dy}}\lesssim \|\langle\partial_y\rangle^{|\alpha|-1}f\|_{L^2_{dy}} + \|f\|_{L^1_{\langle y\rangle^{\eps_*}dy}}
\end{equation}
as well as 
\begin{equation}\label{eq:weightedenergy2}
\|\nabla_{t,y}^{\alpha}\Gamma_2^{\kappa}P_c \frac{\sin(t\sqrt{\mathcal{L}})}{\sqrt{\mathcal{L}}}f\|_{L^2_{dy}}\lesssim \|\langle\partial_y\rangle^{|\alpha|-1}\langle \Gamma_2\rangle^{\kappa}f\|_{L^2_{dy}} + \|\langle \Gamma_2\rangle^{\kappa}f\|_{L^1_{\langle y\rangle^{\eps_*}dy}}.
\end{equation}
As for radiative decay, we have the following: for every $\sigma \in
[1/2, 1]$, it holds that
\[
\|\langle y\rangle^{-\sigma} P_c
e^{it\sqrt{\mathcal{L}}}f\|_{L^\infty_{dy}}\lesssim \langle
t\rangle^{-\sigma}\big[\|\langle y\rangle^{\sigma}f\|_{L^1_{dy}} +
\|\langle y\rangle^{\sigma}f'\|_{L^1_{dy}}\big],
\]
and 
\[
\|\langle y\rangle^{-\sigma} P_c
\frac{\sin(t\sqrt{\mathcal{L}})}{\sqrt{\mathcal{L}}}g\|_{L^\infty_{dy}}\lesssim
\langle t\rangle^{-\sigma}\|\langle y\rangle^{\sigma}g\|_{L^1_{dy}}.
\]
\end{prop}

\begin{rem}
In \cite{DoKr}, the linear estimates \eqref{eq:plainenergy2} and
\eqref{eq:weightedenergy2} are proven with $\eps_* = 0$, which imply
the stated versions above. We introduce the $\jb{y}^{\eps_*}$ weight
(which should be regarded as a constraint on the decay rate of initial
data) in order to prevent logarithmic divergences when controlling the 
nonlinear contributions. As such, $\eps_*$ should be considered a small
positive number that is fixed once and for all. 
\end{rem}

\begin{rem}
For the radiative decay, the value $\sigma = \frac12$ corresponds to
the natural, unweighted decay for wave equations on $1+2$ dimensional
space-times. The values $\sigma > \frac12$ are weighted estimates
factoring in spatial localizations. 
\end{rem}

The preceding bounds are still too crude to handle the unweighted cubic
interaction terms that shows up in $Q_3$ of \eqref{eqs:mainQuasi}, and 
so we complement them with the following. 
\begin{prop}\label{eq:keyest2}For any multi-index $\alpha = (\alpha_1, \alpha_2)\in \N_{\geq 0}^2$, $|\alpha|\geq 1$, we have 
\begin{equation}\label{eq:plainenergy3}
\|\nabla_{t,y}^{\alpha}P_c \frac{\sin(t\sqrt{\mathcal{L}})}{\sqrt{\mathcal{L}}}(\partial_y f)\|_{L^2_{dy}}\lesssim \|\langle\partial_y\rangle^{|\alpha|}f\|_{L^2_{dy}}, 
\end{equation}
\begin{equation}\label{eq:weightedenergy3}
\|\nabla_{t,y}^{\alpha}\Gamma_2^{\kappa}P_c \frac{\sin(t\sqrt{\mathcal{L}})}{\sqrt{\mathcal{L}}}(\partial_y f)\|_{L^2_{dy}}\lesssim \|\langle\partial_y\rangle^{|\alpha|}\langle \Gamma_2\rangle^{\kappa}f\|_{L^2_{dy}},
\end{equation}
as well as for the inhomogeneous evolution 
\begin{equation}\label{eq:plainenergy4}
\|\nabla_{t,y}^{\alpha}P_c
\int_0^t\frac{\sin\left((t-s)\sqrt{\mathcal{L}}\right)}{\sqrt{\mathcal{L}}}(\partial_s F)\|_{L^2_{dy}}\lesssim \|\langle\nabla_{s,y}\rangle^{|\alpha|}F\|_{L_s^1L^2_{dy}}, 
\end{equation}
\begin{equation}\label{eq:weightedenergy4}
\|\nabla_{t,y}^{\alpha}\Gamma_2^{\kappa}P_c
\int_0^t\frac{\sin\left((t-s)\sqrt{\mathcal{L}}\right)}{\sqrt{\mathcal{L}}}(\partial_s F)\|_{L^2_{dy}}\lesssim \|\langle\nabla_{s,y}\rangle^{|\alpha|}\langle \Gamma_2\rangle^{\kappa}F\|_{L_s^1L^2_{dy}}. 
\end{equation}
\end{prop}

In order to handle the local terms in \eqref{eq:Main}, we need a local energy decay result. This is given by the following 
\begin{prop}\label{eq:keyest3} We have the space-time bounds 
\begin{align*}
&\|\langle y\rangle^{-1}\nabla_{t,y}^{\alpha}\Gamma_2^{\kappa}\cos(t\sqrt{\mathcal{L}})P_c f\|_{L^2_{t,y}} + \|\langle y\log y\rangle^{-1}\nabla_{t,y}^{\alpha}\Gamma_2^{\kappa}\frac{\sin(t\sqrt{\mathcal{L}})}{\sqrt{\mathcal{L}}}P_cg\|_{L_{t,y}^2}
\\&\lesssim \|\langle\partial_y\rangle^{|\alpha|}\langle\Gamma\rangle^{\kappa}f\|_{L^2_{dy}} + \|\langle\partial_y\rangle^{|\alpha|-1}\langle\Gamma\rangle^{\kappa}g\|_{L^2_{\langle y\rangle^{1+}dy}}.
\end{align*}
The inhomogeneous version with source terms of gradient structure is as follows: 
\begin{align*}
\|\langle y \log y\rangle^{-1} \nabla_{t,y}^{\alpha}\Gamma_2^{\kappa}P_c \int_0^t\frac{\sin(t-s)\sqrt{\mathcal{L}})}{\sqrt{\mathcal{L}}}(\partial_{s,y} F)\|_{L^2_{t,y}}\lesssim \|\langle\nabla_{s,y}\rangle^{|\alpha|}\langle \Gamma_2\rangle^{\kappa}F\|_{L_s^1L^2_{dy}}. 
\end{align*}
\end{prop}

These linear estimates were proven in \cite{DoKr} for a large class of
$1+1$ dimensional wave equations with potentials. The main technique
used there is that of the \emph{distorted Fourier transform}, which is
essentially the representation of solutions as superpositions of
generalized eigenfunctions of the linear operator $\mathcal{L}$
(analogous to how the Fourier transform represents functions as
superpositions of generalized eigenfunctions of the Laplacian). When
restricted to the continuous part of the spectrum, this representation
allows the use of oscillatory integral techniques to obtain dispersive
estimates. Part of the difficulty in implementing this process is in
analyzing the spectral measure and obtain suitable descriptions of the
generalized eigenfunctions. These are done in detail in \cite{DoKr}. 

To control the decay rate of the nonlinear evolution, in addition to
the standard dispersive bounds we also need a variation of the
\emph{vector field method} to take advantage of the algebraic
structure of the nonlinearities. 
Recall that the vectorfields associated to $-\pr^2_t+\pr^2_y$ 
are the Lorentz boost generator $\Ga_1=t\pr_y+y\pr_t$ and the generator of scaling symmetry $\Ga_2=t\pr_t+y\pr_y$. 
In our problem, we cannot obtain estimates on $\Gamma \phi$ using the
estimates available for solutions $\phi$, since $\Gamma_{1,2}$ do not
commute with the linearized operator $\mathcal{L}$. To overcome this
in \cite{DoKr} a method is introduced where bounds on $\Gamma_2$ are
obtained by studying its analogue under the distorted Fourier
transform. From these linear
estimates on $\Gamma_2$ derivatives we obtain control on $\Gamma_1$
derivatives, by using the structure of the equation and
the behavior of the solution in the space-time regions $y\ll t$ and 
$y\gtrsim t$; see Lemma \ref{lem:3} for details.


\subsection{Main Theorem}
\label{sec:statement}

The unstable mode associated with $\mathcal{L}$ should lead in general
to exponentially growing solutions for \eqref{eq:Maintilde}, even for
arbitrarily small initial data. Nonetheless, it is natural to expect
the existence of a suitable co-dimension one set of small initial data
corresponding  to solutions which exist globally in forward time and
decay toward zero, i.\ e.\ the evolved surface converges to the static catenoid. This is proved in the following theorem which is our main result. 

\begin{thm}[Codimension one stability of the catenoid]\label{thm:Main}
Let us be given a pair of even functions $(\tilde{\phi}_1, \tilde{\phi}_2)\in W^{N_0,1}(\R)\cap W^{N_0, 2}(\R)$ satisfying the smallness condition 
\[
\Vert \tilde{\phi}\Vert_{X_0}: = \sum_{j=1,2}\|\langle y\rangle^{N_0-j+1}\langle\partial_y\rangle^{N_0-j+1}\tilde{\phi}_j\|_{L^1_{dy}\cap L^2_{dy}}\leq \delta_0
\]
for $\delta_0>0$ sufficiently small, and $N_0$ sufficiently large. Then there exists a parameter $a\in \R$ which depends Lipschitz continuously on $\tphi_{1,2}$ with respect to $X_0$ such that the solution $\tilde{\phi}$ of \eqref{eq:Maintilde} corresponding to the initial data 
\[
\left(\tilde{\phi}(0,\cdot), \pr_t\tilde{\phi}(0,\cdot)\right) = (\tilde{\phi}_1 + a g_d, \tilde{\phi}_2)
\]
exists globally in forward time $t>0$. Moreover, $\phi=\langle y\rangle^{-1/2}\tilde{\phi}$ decays toward zero: 
$$|\phi(t,\cdot)|\lesssim \langle t\rangle^{-\frac{1}{2}}.$$ 
\end{thm}

\vspace{0.2cm}

An interesting open problem is the description of the flow in the neighborhood 
of the codimension 1 manifold of Theorem \ref{thm:Main}, and in particular whether this 
manifold is a threshold between two different types of stable regimes. 
An analogous problem has been studied in \cite{MeRaSz} in the case of 
the $L^2$ critical nonlinear Schr\"odinger equation. The initial data corresponding to 
Bourgain-Wang solutions (which are expected to form a co-dimension one
manifold \cite{KrS1}), are 
shown to lie at the boundary between solutions blowing up in finite time in the log-log regime  
and solutions scattering to 0 (note that both are known to be stable regimes for that equation). 
 Numerical simulations for the extremal surface equation suggest that a similar behavior might 
 take place here. Indeed, the codimension 1 manifold of Theorem \ref{thm:Main} seems to be the threshold 
 between two types of regimes: one leading to a collapse of the 
collar ($\phi \to - \jb{y}^2$ for some 
$\abs{y} \ll 1$ and the solution ceases to be an immersed submanifold), 
and another
leading to the accelerated widening of the collar region. 


\section{Setting up the analysis}\label{sec:boot}


The aim of this section is to set up the bootstrap argument. 


\subsection{Spectral decomposition of the solution}


We decompose our solution $\tphi$ as 
$$\tphi=h(t)g_d+\tpsi$$
so that $\tpsi$ satisfies
$$\langle \tpsi, g_d\rangle=0.$$
Thus, we have
$$P_d\tphi=h(t)g_d,\,\, P_c\tphi=\tpsi.$$
In particular, $\tpsi$ satisfies in view of \eqref{eq:Maintilde}
\begin{equation}\label{eq:tpsi}
\begin{cases}
-\partial_t^2\tilde{\psi} + \partial_y^2\tilde{\psi} +
\frac{1}{2}\frac{3+\frac{y^2}{2}}{(1+y^2)^2}\tilde{\psi} =
P_c((1+y^2)^{\frac{1}{4}}F(y,\phi, \nabla\phi, \nabla^2\phi)),\\
\tpsi(0,.)=P_c\tphi_1,\,\, \pr_t\tpsi(0,.)=P_c\tphi_2.
\end{cases}
\end{equation}
We derive a formula for $h(t)$ in the following lemma.
\begin{lem}\label{lemma:formluah}
$h(t)$ is given by
\begin{eqnarray*}
&&h(t)\\
&=& \frac{1}{2}\left(a+\langle \tphi_1, g_d\rangle+\frac{\langle \tphi_2, g_d\rangle}{k_d}-\frac{1}{k_d}\int_0^t\langle (1+y^2)^{\frac{1}{4}}F(\phi, \nabla\phi, \nabla^2\phi)(s), g_d\rangle e^{-k_ds}ds\right)e^{k_dt}\nonumber\\
&+&\frac{1}{2}\left(a+\langle \tphi_1, g_d\rangle-\frac{\langle \tphi_2, g_d\rangle}{k_d}+\frac{1}{k_d}\int_0^t\langle (1+y^2)^{\frac{1}{4}}F(\phi, \nabla\phi, \nabla^2\phi)(s), g_d\rangle 
e^{k_ds}ds\right)e^{-k_dt}.\nonumber
\end{eqnarray*}
\end{lem}

\begin{proof}
$h(t)$ satisfies in view of \eqref{eq:Maintilde} and the fact that
$g_d$ is en eigenvector of $\mathcal{L}$ with eigenvalue $-k_d^2$:
$$-h''(t) + k_d^2h(t) = \langle (1+y^2)^{\frac{1}{4}}F(\phi, \nabla\phi, \nabla^2\phi), g_d\rangle.$$
Using the variation of constant methods, we deduce
\begin{eqnarray*}
h(t)&=& \left(A_1-\frac{1}{2k_d}\int_0^t\langle (1+y^2)^{\frac{1}{4}}F(\phi, \nabla\phi, \nabla^2\phi)(s), g_d\rangle e^{-k_ds}ds\right)e^{k_dt}\nonumber\\
&&+\left(A_2+\frac{1}{2k_d}\int_0^t\langle (1+y^2)^{\frac{1}{4}}F(\phi, \nabla\phi, \nabla^2\phi)(s), g_d\rangle e^{k_ds}ds\right)e^{-k_dt}.\nonumber
\end{eqnarray*}
Since we have
$$h(0)=a+\langle \tphi_1, g_d\rangle,\,\, h'(0)=\langle \tphi_2, g_d\rangle,$$
we deduce
$$A_1=\frac{1}{2}\left(a+\langle \tphi_1, g_d\rangle+\frac{\langle \tphi_2, g_d\rangle}{k_d}\right)\textrm{ and }A_2=\frac{1}{2}\left(a+\langle \tphi_1, g_d\rangle-\frac{\langle \tphi_2, g_d\rangle}{k_d}\right).$$
This concludes the proof of the lemma.
\end{proof}


\subsection{Setting up the bootstrap}


Consider a time $T>0$ such that the following bootstrap assumptions hold on $[0,T)$:
\begin{equation}\label{eq:ener1boot}
\|\nabla_{t,y}\nabla_{t,y}^{\alpha}\tilde{\phi}\|_{L^2_{dy}}\leq \eps\langle t\rangle^{\nu},\,0\leq |\alpha|\leq N_1,
\end{equation}
\begin{equation}\label{eq:disp1boot}
\|\nabla_{t,y}^{\beta}\phi\|_{L^\infty_{dy}}\leq \eps\langle t\rangle^{-\frac{1}{2}},\,0\leq |\beta|\leq \frac{N_1}{2} + C,
\end{equation}
\begin{equation}\label{eq:disp2boot}
\|\langle y\rangle^{-\frac{1}{2}}\nabla_{t,y}^{\beta}\phi\|_{L^\infty_{dy}}\leq \eps\langle t\rangle^{-\frac{1}{2}-\delta_1},\,0\leq |\beta|\leq \frac{N_1}{2} + C,
\end{equation}
\begin{equation}\label{eq:ener2boot}
\|\nabla_{t,y}\nabla_{t,y}^{\beta}\Gamma_2^{\gamma}\tilde{\phi}\|_{L^2_{dy}}\leq \eps\langle t\rangle^{([\frac{2|\beta|}{N_1}]+1)10^\gamma\nu},\,0\leq |\beta|\leq N_1-\gamma,\,0\leq \gamma\leq 2,
\end{equation}
\begin{equation}\label{eq:localenerboot}\begin{split}
&\|\langle y\log y\rangle^{-1}\big(\nabla_{t,y}^{\beta}\Gamma_2^{\gamma}\tilde{\phi})\|_{L^2_{t,y}([0,T])}\leq \eps \langle T\rangle^{(\chi_{\gamma>0}[\frac{2|\beta|}{N_1}]+1)10^\gamma\nu},\\&\hspace{6cm} 0\leq |\beta|\leq 1+N_1-\gamma,\,0\leq \gamma\leq 2,
\end{split}\end{equation}
\begin{equation}\label{eq:unstableboundboot}\begin{split}
&\sum_{\beta\leq N_1+1}|\partial_t^{\beta}h(t)|\leq \eps \langle t\rangle^{-1-2\delta_1},\\
&\sum_{\beta+\kappa\leq N_1+1}|\partial_t^{\beta}(t\partial_t)^{\kappa}h(t)|\leq \eps \langle t\rangle^{(1+[\frac{2|\beta|}{N_1}])10^{\kappa}\nu},\,\kappa\in \{1,2\},\\
&\sum_{\beta+\kappa\leq N_1+1}\big\|\partial_t^{\beta}(t\partial_t)^{\kappa}h\big\|_{L^2_{[0,T]}}\leq \eps\langle T\rangle^{(1+[\frac{2|\beta|}{N_1}])10^{\kappa}\nu},\,\kappa\in \{1,2\}.\\
\end{split}\end{equation}
Our claim is that the above regime is trapped.

\begin{prop}[Improvement of the bootstrap assumptions]\label{prop:Core} 
 There exists an $N_1$ sufficiently large, such that the following holds: there is $N_0$ sufficiently large, such that if $N_1\gg C\geq 10$ and given $\eps>0$, $1\gg \delta_1\gg\nu\gg\eps$, there is  $\delta_0 = \delta_0(\eps, N_0)>0$ sufficiently small (as in Theorem~\ref{thm:Main}) and 
$$a\in [-\eps^{\frac{3}{2}},\, \eps^{\frac{3}{2}}]$$
 such that $\tilde{\phi}$ satisfies the following bounds 
\begin{equation}\label{eq:ener1}
\|\nabla_{t,y}\nabla_{t,y}^{\alpha}\tilde{\phi}\|_{L^2_{dy}}\lesssim (\delta_0+\eps^{\frac{3}{2}})\langle t\rangle^{\nu},\,0\leq |\alpha|\leq N_1,
\end{equation}
\begin{equation}\label{eq:disp1}
\|\nabla_{t,y}^{\beta}\phi\|_{L^\infty_{dy}}\lesssim (\delta_0+\eps^{\frac{3}{2}})\langle t\rangle^{-\frac{1}{2}},\,0\leq |\beta|\leq \frac{N_1}{2} + C,
\end{equation}
\begin{equation}\label{eq:disp2}
\|\langle y\rangle^{-\frac{1}{2}}\nabla_{t,y}^{\beta}\phi\|_{L^\infty_{dy}}\lesssim (\delta_0+\eps^{\frac{3}{2}})\langle t\rangle^{-\frac{1}{2}-\delta_1},\,0\leq |\beta|\leq \frac{N_1}{2} + C,
\end{equation}
\begin{equation}\label{eq:ener2}
\|\nabla_{t,y}\nabla_{t,y}^{\beta}\Gamma_2^{\gamma}\tilde{\phi}\|_{L^2_{dy}}\lesssim (\delta_0+\eps^{\frac{3}{2}})\langle t\rangle^{([\frac{2|\beta|}{N_1}]+1)10^\gamma\nu},\,0\leq |\beta|\leq N_1-\gamma,\,0\leq \gamma\leq 2,
\end{equation}
\begin{equation}\label{eq:localener}\begin{split}
&\|\langle y\log y\rangle^{-1}\big(\nabla_{t,y}^{\beta}\Gamma_2^{\gamma}\tilde{\phi})\|_{L^2_{t,y}([0,T])}\lesssim (\delta_0+\eps^{\frac{3}{2}}) \langle T\rangle^{(\chi_{\gamma>0}[\frac{2|\beta|}{N_1}]+1)10^\gamma\nu},\\&\hspace{6cm} 0\leq |\beta|\leq 1+N_1-\gamma,\,0\leq \gamma\leq 2,
\end{split}\end{equation}
\begin{equation}\label{eq:unstablebound}\begin{split}
&\sum_{\beta\leq N_1+1}|\partial_t^{\beta}h(t)|\lesssim (\delta_0+\eps^{\frac{3}{2}}) \langle t\rangle^{-1-2\delta_1},\\
&\sum_{\beta+\kappa\leq N_1+1}|\partial_t^{\beta}(t\partial_t)^{\kappa}h(t)|\lesssim (\delta_0+\eps^{\frac{3}{2}}) \langle t\rangle^{(1+[\frac{2|\beta|}{N_1}])10^{\kappa}\nu},\,\kappa\in \{1,2\},\\
&\sum_{\beta+\kappa\leq N_1+1}\big\|\partial_t^{\beta}(t\partial_t)^{\kappa}h\big\|_{L^2_{[0,T]}}\lesssim (\delta_0+\eps^{\frac{3}{2}})\langle T\rangle^{(1+[\frac{2|\beta|}{N_1}])10^{\kappa}\nu},\,\kappa\in \{1,2\}.\\
\end{split}\end{equation}
\end{prop}

\vspace{0.3cm}

The rest of the paper is as follows. In section
\ref{sec:energybounds}, we prove the energy bounds  \eqref{eq:ener1}
and \eqref{eq:ener2}. In section \ref{sec:localenergydecay}, we prove
the local energy decay \eqref{eq:localener}. In section
\ref{sec:dispersive}, we prove the decay estimates \eqref{eq:disp1} and \eqref{eq:disp2}. In section \ref{sec:existencea}, we prove the existence of $a$ such that \eqref{eq:unstablebound} holds which concludes the proof of Proposition~\ref{prop:Core}. Finally, we prove Theorem \ref{thm:Main} in section \ref{sec:proofmaintheorem}.


\section{Energy bounds}\label{sec:energybounds}


The goal of this section is to prove the estimates \eqref{eq:ener1} and \eqref{eq:ener2}.


\subsection{The proof of the estimate (\protect\ref*{eq:ener1})}


In view of \eqref{eq:tpsi}, we have
\begin{equation}\label{eq:psiparam}
\tilde{\psi} = \cos(t\sqrt{\mathcal{L}})P_c\tilde{\phi}_1 + \frac{\sin(t\sqrt{\mathcal{L}})}{\sqrt{\mathcal{L}}}P_c\tilde{\phi}_2 + \int_0^t \frac{\sin([t-s]\sqrt{\mathcal{L}})}{\sqrt{\mathcal{L}}}P_c\big(G(s, \cdot)\big)\,ds
\end{equation}
where 
\[
G(s, y) = (1+y^2)^{\frac{1}{4}}F(y,\phi , \nabla\phi ,\nabla^2\phi).
\]
In order to derive the desired energy bounds, we can use Proposition~\ref{eq:keyest1} for the weighted terms without maximum order derivatives, and Proposition~\ref{eq:keyest2} for the pure cubic terms, as we shall see. In order to deal with the maximum order derivative terms, we have to use a direct integration by parts argument. 
To begin with, we reveal the gradient structure in the top
order cubic terms. One can check easily 
that the following identity holds 
\begin{multline}\label{eq:keygradientstructure2}
\partial_t\big[\phi_t^2\psi_{t}\big] -
2\partial_y\big[\phi_{y}\phi_t\psi_t\big] +
\partial_t\big[\phi_y^2\psi_t\big] + \phi_t^2(\psi_{yy} - \psi_{tt})+
2(\phi_{yy}-\phi_{tt})\phi_t\psi_t\\
= \phi_t^2 \psi_{yy} - 2\phi_y\phi_t\psi_{ty} + \phi_y^2\psi_{tt}
\end{multline}
Denote 
\[
X_{t,y}(\nabla\phi,\nabla\psi) = X_{t,y}(\nabla \phi,\partial_t \psi) : = \partial_t\big[\phi_t^2\psi_t\big] -
2\partial_y\big[\phi_{y}\phi_t\psi_t\big] +
\partial_t\big[\phi_y^2\psi_t\big]. 
\]
Note that $X_{t,y}$ is linear in its second argument. 

In order to recover the bounds for $\tilde{\psi}$, we then distinguish between the following three cases: 

\subsubsection{First order derivatives} 
Examining the wave equation \eqref{eq:tpsi}, we can split the
right-hand side into two parts writing
\[ \jb{y}^{\frac12} F(y,\phi,\nabla\phi,\nabla^2\phi) =
\jb{y}^{\frac12} F_1(y,\phi,\nabla\phi,\nabla^2\phi) +
X_{t,y}(\nabla\phi,\nabla\tilde{\phi}).\]
Then applying Propositions \ref{eq:keyest1} and \ref{eq:keyest2} to
the wave equation \eqref{eq:tpsi} in view of the splitting above, we obtain 
\begin{multline}\label{eq:lowerorderder1}
\sup_{t\in [0,T]}\|\nabla_{t,y}\tilde{\psi}\|_{L^2_{dy}}\lesssim
\|\langle\nabla_y\rangle\tilde{\phi}_1\|_{L^2_{dy}} +
\|\tilde{\phi}_2\|_{L^2_{dy} \cap L^1_{\langle y\rangle^{\eps_*}dy}}
\\
+ \big\| F_1\big\|_{L_t^1 L^2_{\jb{y}dy}\cap L_t^1
L^1_{\jb{y}^{\smash{\frac12+\eps_*}}dy}[0,T]}
+\sum_{k=1}^3\big\|A_k\big\|_{L_t^1 L^2_{ dy}[0,T]}
\end{multline}
where 
\[
A_1 = ( \pr_t\phi  )^2\pr_t\tphi,\quad A_2 =  \pr_y\phi   \pr_t\phi
\pr_t\tilde{\phi},\quad A_3 = ( \pr_y\phi  )^2\pr_t\tilde{\phi}
\]
come from $X_{t,y}(\nabla\phi,\nabla\tilde{\phi})$. 

From our assumptions on $\tilde{\phi}_{1,2}$, we have
\begin{equation}\label{eq:lowerorderder2}
\|\langle\nabla_y\rangle\tilde{\phi}_1\|_{L^2_{dy}} +    \|\tilde{\phi}_2\|_{L^2_{dy} \cap L^1_{\langle y\rangle^{\eps_*}dy}}\leq\delta_0.
\end{equation}
The contributions from the terms $A_k$ are also straightforward to control. Using the bootstrap assumptions \eqref{eq:ener1boot} and \eqref{eq:disp1boot}, we have
\begin{multline}\label{eq:lowerorderder3}
\big\|(\nabla_{t,y}\phi  )^2\nabla_{t,y}\tilde{\phi}\big\|_{L_t^1 L^2_{dy}[0,T]}
\lesssim \big\|\langle t\rangle^{\nu}|\nabla_{t,y}\phi  |^2\big\|_{L_{t}^1 L^\infty_{dy}[0,T]}\big\|\langle t\rangle^{-\nu}\nabla_{t,y}\tphi\big\|_{L_t^\infty L^2_{dy}}\\
\lesssim \nu^{-1}\eps^3\langle T\rangle^{\nu}\lesssim \eps^2\langle T\rangle^{\nu}.
\end{multline}
It remains to deal with the more complicated source term $F_1$. We
observe that $F_1$ can be decomposed (see \eqref{eq:Main},
\eqref{eqs:mainQuasi}, and \eqref{eqs:mainSemi}) as 
\begin{multline*} 
F_1 = \underbrace{Q_2 + Q_4 + S_2 + S_3 + S_4 +
\frac{\phi^2}{\jb{y}^4}\phi_{yy}}_{\text{with good } y \text{
weights}} \\
+ (\phi_t)^2 \phi_{yy} - 2 \phi_t \phi_y \phi_{ty} + (\phi_y)^2
\phi_{tt} - \jb{y}^{-\frac12} X_{t,y}(\nabla \phi, \jb{y}^{\frac12}
\partial_t \phi).
\end{multline*}

We easily see that using the bootstrap assumption
\eqref{eq:disp1boot} and \eqref{eq:disp2boot}
 the terms with the good $y$ weights are bounded
pointwise
\begin{multline*} 
\jb{y}^{\frac12} \left| Q_2 + Q_4 + S_2 + S_3 + S_4 +
\frac{\phi^2}{\jb{y}^4}\phi_{yy} \right| \\
\lesssim \jb{y}^{-\frac32}
\left[ |\jb{\nabla_{t,y}}\phi|^2 + |\jb{\nabla_{t,y}}^2\phi|^2\right]
\lesssim \jb{y}^{-1} \eps^2 \jb{t}^{-1-\delta_1}
\end{multline*}
Integrating in $L^1_t (L^2_{dy} \cap L^1_{\jb{y}^{\eps_*} dy})$ its
contribution to $\|F_1\|_{L^1_t (L^2_{\jb{y}dy}\cap
L^1_{\jb{y}^{\smash{\frac12 + \eps_*}} dy})}$ can be bounded by
$\eps^{\frac32}$. 

Using \eqref{eq:keygradientstructure2}, we can rewrite
\begin{multline*}
\jb{y}^{\frac12} \left[ (\phi_t)^2 \phi_{yy} - 2 \phi_t \phi_y
\phi_{ty} + (\phi_y)^2
\phi_{tt}\right] - X_{t,y}(\nabla \phi, \nabla \tilde{\phi}) \\
= (\phi_t)^2 \left[ \jb{y}^{\frac12} \phi_{yy} -
\tilde{\phi}_{yy}\right] - 2 \phi_t \phi_y \left[ \jb{y}^{\frac12}
\phi_{ty} - \tilde{\phi}_{ty}\right] \\
+ \phi_t^2( \tilde{\phi}_{yy} - \tilde{\phi}_{tt}) + 2 \phi_t
\tilde{\phi}_t (\phi_{yy} - \phi_{tt}).
\end{multline*}
The first two terms exhibit an important cancellation:
\[ (\phi_t)^2 \left[ \jb{y}^{\frac12} \phi_{yy} -
\tilde{\phi}_{yy}\right] - 2 \phi_t \phi_y \left[ \jb{y}^{\frac12}
\phi_{ty} - \tilde{\phi}_{ty}\right] = 
- (\phi_t)^2\phi \partial^2_{yy}(\jb{y}^{\frac12}) 
\]
which has a good $y$ weight of order $\jb{y}^{-\frac32}$, 
allowing us to estimate it exactly as above. For the remaining two
terms we apply the equation. The wave equation \eqref{eq:Maintilde}
gives the crude pointwise bound
\[ |\tilde{\phi}_{yy} - \tilde{\phi}_{tt}| \lesssim \jb{y}^{-2} \left(
|\tilde{\phi}| + |\jb{\nabla_{t,y}}^2\phi|\cdot|\jb{\nabla_{t,y}}^2
\tilde{\phi}|\right) +
|\nabla_{t,y}\phi|^2|\nabla^2_{t,y}\tilde{\phi}|.\]
The first term gives rise to a term of good $y$ weight which can be
estimated like above, and the second term can be controlled by 
\[
\jb{y}^{-\frac12}|\partial_t\tilde{\phi}||\nabla_{t,y}\phi|^3|\nabla_{t,y}^2\tilde{\phi}|
\]
whose norm in $L^1_t (L^2_{dy} \cap L^1_{\jb{y}^{\eps_*} dy})$ is
easily controlled by $\eps^4$ using the bootstrap assumptions
\eqref{eq:ener1boot} and \eqref{eq:disp1boot}. The term with
$|\phi_{yy} - \phi_{tt}|$ can be treated similarly using
\eqref{eq:Main}. Summarizing, we have
shown that  
\begin{equation}\label{eq:lowerorderder4}
 \big\| F_1\big\|_{L_t^1 L^2_{\jb{y}dy}\cap L_t^1
L^1_{\jb{y}^{\smash{\frac{1}{2}+\eps_*}}dy}[0,T]}\lesssim \eps^{\frac{3}{2}}.
\end{equation}

Combining the estimates
\eqref{eq:lowerorderder1}-\eqref{eq:lowerorderder4} we obtain,
finally, 
\begin{equation}\label{eq:lowerorderder5}
\sup_{t\in [0,T]}\|\nabla_{t,y}\tilde{\psi}\|_{L^2_{dy}}\lesssim \delta_0+\eps^{\frac{3}{2}}+\eps^2\langle T\rangle^{\nu}.
\end{equation}

\subsubsection{Higher order derivatives of degree strictly less than
$N_1$}\label{sect:highermidnottop} Here we use induction on the degree of the derivatives, assuming the bound \eqref{eq:lowerorderder5}. Write the equation for $\tilde{\psi}$ schematically in the form 
\[
-\partial_t^{2}\tilde{\psi} + \partial_y^2\tilde{\psi}+ \frac{1}{2}\frac{3+\frac{y^2}{2}}{(1+y^2)^2}\tilde{\psi} = P_cG.
\]
Applying $\partial_t^{\beta}$ with $1\leq \beta\leq N_1-1$, and integrating against $\partial_t^{\beta+1}\tilde{\psi}$, we easily infer 
\begin{equation}\begin{split}\label{eq:intpar}
&\big(\int_{\R}\frac{1}{2}\big[|\partial_t^{\beta+1}\tilde{\psi}|^2 + |\partial_t^{\beta}\partial_y\tilde{\psi}|^2 - \frac{1}{2}\frac{3+\frac{y^2}{2}}{(1+y^2)^2}|\partial_t^{\beta}\tilde{\psi}|^2\big]\,dy\big)\big|_{0}^T\\
&= - \int_0^T\int_{\R} (P_c\partial_t^\beta G) \partial_t^{\beta+1}\tilde{\psi}\,dt dy.
\end{split}\end{equation}
Recall that we have 
\[
G = (1+y^2)^{\frac{1}{4}}F(\phi , \nabla\phi ,\nabla^2\phi).
\]
Note that we have the crude bound
\[
\big|\partial_t^\beta G\big|\lesssim \sum_{\beta_1 + \beta_2 = \beta}\frac{|\langle\nabla_{t,y}\rangle^2\partial_t^{\beta_1}\phi||\langle\nabla_{t,y}\rangle^2\partial_t^{\beta_2}\tphi|}{1+y^2} + \sum_{\beta_1+\beta_2+\beta_3 = \beta}\prod_{j=1}^2|\nabla_{t,y}\langle \nabla_{t,y}\rangle\partial_t^{\beta_j}\phi||\nabla_{t,y}\langle \nabla_{t,y}\rangle\partial_t^{\beta_3}\tphi|
\]
where we may assume $\beta_3\geq\beta_2\geq\beta_1$. We use the energy
bound \eqref{eq:ener1boot}, the local energy decay
\eqref{eq:localenerboot} (with $\gamma = 0$), as well as the decay
estimates \eqref{eq:disp1boot}, \eqref{eq:disp2boot}, the latter in order to deal with the logarithmic degeneracy in \eqref{eq:localenerboot}. It then follows that (using $\beta+2\leq N_1+1$)
\begin{align*}
&\langle T\rangle^{-2\nu}\big|\int_0^T\int_{\R} (P_c\partial_t^\beta G) \partial_t^{\beta+1}\tilde{\psi}\,dt dy\big|\\&\lesssim \big\|\langle t\rangle^{\frac{1}{2}+}\langle y\rangle^{-\frac{1}{2}}\langle \nabla_{t,y}\rangle^{\frac{N_1+3}{2}}\phi\big\|_{L_{t,y}^\infty}\langle T\rangle^{-\nu}\big\|\frac{\langle\nabla_{t,y}\rangle^{N_1+1}\tilde{\phi}}{1+y^\frac{3}{2}}\big\|_{L_{t,y}^2([0,T])}\big\|\langle t\rangle^{-\nu}\partial_t^{\beta+1}\tilde{\psi}\big\|_{L_t^\infty L_y^2([0,T])}\\
& + \nu^{-1}\big\|\langle t\rangle^{\frac{1}{2}}\langle \nabla_{t,y}\rangle^{\frac{N_1+3}{2}}\phi\big\|_{L_{t,y}^\infty}^2\big\|\langle t\rangle^{-\nu}\langle\nabla_{t,y}\rangle^{N_1+1}\tilde{\phi}\big\|_{L_t^\infty L_y^2([0,T])}\big\|\langle t\rangle^{-\nu}\partial_t^{\beta+1}\tilde{\psi}\big\|_{L_t^\infty L_y^2([0,T])}\\
&\lesssim \eps^3+\frac{\eps^4}{\nu}\lesssim \eps^3.
\end{align*}
This recovers the desired bound \eqref{eq:ener1} for $\partial_t^{\beta+1}\tilde{\psi}$. To get control over $\|\nabla_{t,y}^{\beta}\phi\|_{L_y^2}$, $1\leq |\beta|\leq N_1$, one uses the pure $t$-derivative bounds, the equation, and induction on the number of $y$-derivatives. 

\subsubsection{Top order derivatives}\label{sect:ener1toporder} 
Here we need to perform integration by parts in the top order derivative contributions. Again it suffices to bound the expression $\partial_t^{N_1+1}\tilde{\psi}$, as the remaining derivatives are controlled directly from the equation. Using \eqref{eq:intpar} with $\beta = N_1$, write schematically
\begin{align*}
(P_c\partial_t^{N_1} G) \partial_t^{N_1+1}\tilde{\psi}= \langle y\rangle^{-2}\phi \partial_t^{N_1}\nabla_{t,y}^2\tphi\partial_t^{N_1+1}\tilde{\psi}+(\nabla_{t,y}\phi )^2\partial_t^{N_1}\nabla_{t,y}^2\tphi\partial_t^{N_1+1}\tilde{\psi} +l.o.t.
\end{align*}
where the contribution of the lower order terms is treated as in
Section \ref{sect:highermidnottop} above. We conclude that 
\begin{align*}
&\int_0^T\int_{\R} (P_c\partial_t^{N_1} G) \partial_t^{N_1+1}\tilde{\psi}\,dt dy\\& = \big(\pm\frac{1}{2}\int_{\R} \langle y\rangle^{-2}\phi |\partial_t^{N_1}\nabla_{t,y}\tilde{\psi}|^2\,dy\big)\big|_{t = 0}^{t = T}\\
&+\big(\pm\frac{1}{2}\int_{\R}(\nabla_{t,y}\phi )^2|\partial_t^{N_1}\nabla_{t,y}\tilde{\psi}|^2\,dy\big)\big|_{t = 0}^{t = T}+l.o.t.\\
&+\int_0^T\int_{\R}\big[\langle y\rangle^{-2}\phi  + (\nabla_{t,y}\phi )^2\big]\big(\partial_t^{N_1}\nabla_{t,y}^2 P_d\tilde{\phi} \big)\partial_t^{N_1+1}\tilde{\psi}\,dydt
\end{align*}
where the terms ``l.o.t.'' can be bounded like in
\ref{sect:highermidnottop}. As the first two integral expressions on the right can be bounded by 
\begin{align*}
&\big|\big(\pm\frac{1}{2}\int_{\R} \langle y\rangle^{-2}\phi |\partial_t^{N_1}\nabla_{t,y}\tilde{\psi}|^2\,dy\big)\big|_{t = 0}^{t = T}+\big(\pm\frac{1}{2}\int_{\R}(\nabla_{t,y}\phi )^2|\partial_t^{N_1}\nabla_{t,y}\tilde{\psi}|^2\,dy\big)\big|_{t = 0}^{t = T}\big|\\
&\lesssim \eps \big\|\partial_t^{N_1}\nabla_{t,y}\tilde{\psi}\big\|_{L_t^\infty L_y^2([0,T])}^2.
\end{align*}
It remains to bound the last integral expression, for which we need to
control $\partial_t^{N_1}\nabla_{t,y}^2 P_d\tilde{\phi} $. We recall that 
\[
P_d\tilde{\phi} = h(t)g_d(y)\textrm{ with }h(t)=\langle \tphi, g_d\rangle
\]
so that
\[
\partial_t^{N_1} \nabla_{t,y} \partial_y P_d \tilde{\phi} =
\nabla_{t,y} ( \partial_t^{N_1} h \cdot \partial_y g_d )
\]
which we can control using the decay and smoothness for $g_d$ together with the
bootstrap assumption \eqref{eq:unstableboundboot}. It remains to
control the term $\partial_t^{N_1} P_d \partial_t^2 \tilde{\phi}$.
Here we have to use the hyperbolic structure of the equation for
$\tilde{\phi}$. After dividing by the coefficient in front of
$\tilde{\phi}_{tt}$ the equation of motion can be re-written as
\begin{equation}\label{eq:tradettderaway}
\tilde{\phi}_{tt} = - \mathcal{L} \tilde{\phi} + H(y, \phi,
\nabla\phi, \nabla \partial_y \tilde\phi) 
\end{equation}
where the quasilinearity $H$ does not depend on $\tilde{\phi}_{tt}$.
Therefore we can write
\[ \partial_t^{N_1} P_d \partial_t^2 \tilde{\phi} = - \partial_t^{N_1}
\langle \tilde{\phi}, \mathcal{L} g_d\rangle g_d + \partial_t^{N_1}
\langle H, g_d\rangle g_d .\]
The first term we use that $g_d$ is an eigenstate of $\mathcal{L}$ and can
control using the bootstrap assumption \eqref{eq:unstableboundboot}
and decay and smoothness for $g_d$. For the second term, for the
non-top-order terms where not all $\partial_t^{N_1}$ hit on $\nabla
\partial_y \tilde{\phi}$ we can apply the bootstrap assumptions
\eqref{eq:ener1boot} and \eqref{eq:disp2boot} using again that $g_d$
decays spatially. The 
top-order term has the form (by quasilinearity) 
\[ \int_y H'(y,\phi,\nabla \phi) ~(\partial_{y} \partial^{N_1}_t \nabla_{t,y}
\tilde{\phi})~ g_d(y) dy \]
which we can integrate by parts to bound by 
\[ \int_y (\partial^{N_1}_t \nabla_{t,y} \tilde{\phi})~ \partial_y\left[
H'(y,\phi,\nabla\phi) g_d\right] dy \]
and which we can control, using the spatial decay of $g_d$, by the
bootstrap assumptions \eqref{eq:ener1boot} and \eqref{eq:disp2boot}.
Putting it all together we have that
\[
\| \partial_t^{N_1} \nabla^2_{t,y} P_d \tilde{\phi} \|_{L^2_y}
\lesssim
\eps \jb{t}^{-1 - 2 \delta_1} + \eps^2 \jb{t}^{-\frac12 -
\delta_1 + \nu}.\]
Taking $\delta_1 > \nu$ and using the additional time decay from
$\jb{y}^{-2}\phi + (\nabla_{t,y}\phi)^2$ we get 
\begin{align*}
&\big|\int_0^T\int_{\R}\big[\langle y\rangle^{-2}\phi  + (\nabla_{t,y}\phi )^2\big]\big(\partial_t^{N_1}\nabla_{t,y}^2 P_d\tilde{\phi} \big)\partial_t^{N_1+1}\tilde{\psi}\,dydt\big|\\
&\lesssim \eps^2 + \eps\big\|\partial_t^{N_1+1}\tilde{\psi}\big\|_{L_t^\infty L_y^2([0,T])}^2.
\end{align*}
Combining the preceding bounds, one easily infers the improved estimate 
\[
\big\|\nabla_{t,y}\partial_t^{N_1}\tilde{\psi}\big\|_{L_t^\infty
L_y^2([0,T])}\lesssim \delta_0+\eps^2.
\]
The remaining (mixed) derivative terms $\nabla_{t,y}^{\beta}\tilde{\psi}$, $|\beta| = N_1+1$, are bounded by induction on the number of $y$-derivatives, using the equation for $\tilde{\psi}$. This completes the proof of \eqref{eq:ener1}.


\subsection{The proof of the estimate (\protect\ref*{eq:ener2})}


We next turn to the weighted energy estimates, of the form \eqref{eq:ener2}. Here we use the weighted bounds in Propositions~\ref{eq:keyest1}, ~\ref{eq:keyest2}. The key to control the quadratic nonlinear terms shall be the local energy bounds \eqref{eq:localenerboot}. To deal with the cubic terms, we start with the following lemma, which will also be useful later on. It ensures that we get control over the Lorentz boost generator $\Gamma_1 = t\partial_y+y\partial_t$.

\begin{lem}\label{lem:3} Let $\Gamma_1: = t\partial_y + y\partial_t$. Then, we can infer the bounds
\[
\big\|\nabla_{t,y}\nabla_{t,y}^{\beta}\Gamma_1^{\kappa}\tilde{\psi}\big\|_{L^2_{dy}}\lesssim \eps \langle t\rangle^{|\kappa|-\frac{1}{2}-\delta_1},\,\kappa+|\beta|\leq N_1,\,\kappa\in \{1,2\}.
\]
\[
\big\|\nabla_{t,y}\nabla_{t,y}^{\beta}\Gamma_1\Gamma_2\tilde{\psi}\big\|_{L^2_{dy}}\lesssim \eps \langle t\rangle^{1+(1+[\frac{2|\beta|}{N_1}])10\nu},\,2+|\beta|\leq N_1.
\]
\end{lem}

\begin{proof}
We start with the first bound of the lemma with $\kappa = 1$. \\

{\it{(1): Proof of the first inequality with $\kappa = 1$.}} Observe that 
\begin{equation}\label{eq:estimga1ga2}
(\Gamma_1\tilde{\psi})_{t,y} - (\Gamma_2\tilde{\psi})_{t,y} = O(\big|(t-y)\nabla_{t,y}^2\tilde{\psi}\big|) + O(|\nabla_{t,y}\tilde{\psi}|).
\end{equation}
Further, note
\begin{equation}\label{eq:technical1}
(\Gamma_2\tilde{\psi})_{t} = t\tilde{\psi}_{tt} + y\tilde{\psi}_{ty} + \tilde{\psi}_t,\,(\Gamma_2\tilde{\psi})_{y} = t\tilde{\psi}_{ty} + y\tilde{\psi}_{yy} + \tilde{\psi}_y.
\end{equation}
We can replace $y\tilde{\psi}_{yy}$ by $y\tilde{\psi}_{tt}$ by using the equation 
\[
y\tilde{\psi}_{yy} = y\tilde{\psi}_{tt} - \frac{y}{2}\frac{3+\frac{y^2}{2}}{(1+y^2)^2}\tilde{\psi} + y P_cG.
\]
We infer
\begin{equation}\label{eq:replacepsiyy1}
(t-y)\tpsi_{tt} = \frac{t(\Ga_2\tpsi)_t-y(\Ga_2\tpsi)_y+t\psi_t-y\tpsi_y-y\big(- \frac{y}{2}\frac{3+\frac{y^2}{2}}{(1+y^2)^2}\tilde{\psi} + y P_cG\big)}{t+y},
\end{equation}
\begin{equation}\label{eq:replacepsiyy2}
(t-y)\tpsi_{ty} = \frac{y(\Ga_2\tpsi)_t-t(\Ga_2\tpsi)_y+y\psi_t-t\tpsi_y-t\big(- \frac{y}{2}\frac{3+\frac{y^2}{2}}{(1+y^2)^2}\tilde{\psi} + y P_cG\big)}{t+y}.
\end{equation}
Using the bootstrap assumption \eqref{eq:ener1boot}, we have for $|\beta|+1\leq N_1$ 
\begin{align*}
\big\|\nabla_{t,y}^{\beta}\big[y P_cG(t, \cdot)\big]\big\|_{L^2_{dy}} +\big\|\nabla_{t,y}^{\beta}\big[\frac{y}{2}\frac{3+\frac{y^2}{2}}{(1+y^2)^2}\tilde{\psi}(t, \cdot)\big]\big\|_{L^2_{dy}}\lesssim \eps\langle t\rangle^{3\nu}.
\end{align*}
Together with \eqref{eq:replacepsiyy1}, \eqref{eq:replacepsiyy2} and the bootstrap assumptions \eqref{eq:ener1boot} and \eqref{eq:ener2boot}, we obtain 
\begin{equation}\label{eq:technical2}
\big\|(t-y)\nabla_{t,y}^{\beta}\tilde{\psi}_{tt}\big\|_{L^2_{dy}} + \big\|(t-y)\nabla_{t,y}^{\beta}\tilde{\psi}_{ty}\big\|_{L^2_{dy}}\lesssim \eps\langle t\rangle^{(1+[\frac{2|\beta|}{N_1}])10\nu},\,|\beta|+1\leq N_1. 
\end{equation}
It remains to bound 
\[
\big\|(t-y)\nabla_{t,y}^{\beta}\tilde{\psi}_{yy}\big\|_{L^2_{dy}},\,|\beta|+1\leq N_1.
\]
Here we directly use the equation satisfied by $\tilde{\psi}$. Let $\tilde{\phi}_{1,2}$ be a fundamental system associated with $\mathcal{L}$, with $\tilde{\phi}_1$ given by 
\[
\frac{-\sqrt{1+y^2} + y\sinh^{-1}(y)}{(1+y^2)^{\frac{1}{4}}}.
\]
Note in particular that $|\tilde{\phi}_{1,2}(y)|\lesssim y^{\frac{1}{2}}\log y$ as $y\rightarrow\infty$.  Then we have the formula
\begin{equation}\label{eq:technical3}\begin{split}
\tilde{\psi}(t, y)& = \tilde{\phi}_2(y)\int_0^y\tilde{\phi}_1(\tilde{y})\big[\tilde{\psi}_{tt}(t, \tilde{y}) + P_cG(t, \tilde{y})\big]\,d\tilde{y}\\
&-\tilde{\phi}_1(y)\int_0^y\tilde{\phi}_2(\tilde{y})\big[\tilde{\psi}_{tt}(t, \tilde{y}) + P_cG(t, \tilde{y})\big]\,d\tilde{y}\\
&+a(t)\tilde{\phi}_1(y),
\end{split}\end{equation}
and the improved local decay \eqref{eq:disp2boot} implies 
\[
\big|\partial_t^{\beta}a(t)\big|\lesssim \eps\langle t\rangle^{-\frac{1}{2}-\delta_1},\,\beta\leq \frac{N_1}{2}+C.
\]
But then \eqref{eq:technical2} as well as the precise form of $G$ imply that {\it{restricting to $y\leq t$}}, we have 
\[
\big\|(t-y)\nabla_{t,y}^{\beta}\tilde{\psi}_{yy}(t, \cdot)\big\|_{L^2_{dy}(y\leq t)}\lesssim \eps \langle t\rangle^{\frac{1}{2}-\delta_1},\,|\beta|\leq \frac{N_1}{2}+C,
\]
while the bound
\begin{equation}\label{eq:estimatetmyextregion}
\big\|(t-y)\nabla_{t,y}^{\beta}\tilde{\psi}_{yy}(t, \cdot)\big\|_{L^2_{dy}(y> t)}\lesssim\eps \langle t\rangle^{(1+[\frac{2|\beta|}{N_1}])10\nu}
\end{equation}
follows directly from the equation satisfied by $\tilde{\psi}$. In fact, replacing $\tilde{\psi}_{yy}$ by $\tilde{\psi}_{tt}$ the bound follows from \eqref{eq:replacepsiyy1}, and we can absorb the factor $(t-y)$ in the potential for the linear term (in the region $y\gtrsim t$), while this factor is easily absorbed by the nonlinearity as in the inequality after \eqref{eq:replacepsiyy2}. 
The missing bounds with $|\beta|>\frac{N_1}{2}+C$ are easily obtained directly from the equation (inductively). Together with \eqref{eq:estimga1ga2}, \eqref{eq:technical2} and the bootstrap assumption \eqref{eq:ener2boot}, we deduce
\[
\big\|\nabla_{t,y}\nabla_{t,y}^{\beta}\Gamma_1\tilde{\psi}\big\|_{L^2_{dy}}\lesssim \eps \langle t\rangle^{\frac{1}{2}-\delta_1},\,1+|\beta|\leq N_1.
\]

\vspace{0.3cm}

{\it{(2): Proof of the first inequality of the lemma with $\kappa = 2$.}} Observe that
\[
(\Gamma_1^2 - \Gamma_2^2)\tilde{\psi} = (t^2-y^2)(\tilde{\psi}_{yy} - \tilde{\psi}_{tt}), 
\]
{\it{(a): inner region, $y\leq t$}}. We get 
\[
\big\|\nabla_{t,y}^{\beta}\big[(t^2-y^2) \tilde{\psi}_{tt}(t, \cdot)\big]\big\|_{L^2_{dy}(y\leq t)}\lesssim \eps\langle t\rangle^{1+(1+[\frac{2|\beta|}{N_1}])10\nu},\,|\beta|+1\leq N_1,
\]
on account of \eqref{eq:technical2}. Using \eqref{eq:technical3}, we have
\[
\big\|\nabla_{t,y}^{\beta}\big[(t^2-y^2)\tilde{\psi}_{yy}\big]\big\|_{L^2_{dy}(y\leq t)}\lesssim \eps\langle t\rangle^{\frac{3}{2}-\delta_1},\,|\beta|+1\leq N_1,
\]
provided $\nabla_{t,y}^{\beta} = \partial_t^{\beta_1}\partial_y^{\beta_2}$ with $\beta_1\leq \frac{N_1}{2}+C$, and the remaining cases are obtained  using induction and the equation for $\tilde{\psi}$. It then follows that we have the bounds 
\begin{align*}
\big\|\nabla_{t,y}\nabla_{t,y}^{\beta}(\Gamma_1^2 - \Gamma_2^2)\tilde{\psi}\big\|_{L^2_{dy}(y\leq t)}\lesssim \eps \langle t\rangle^{\frac{3}{2}-\delta_1},\,|\beta|+2\leq N_1.
\end{align*}
{\it(b) For the outer cone region} $y>t$, we use 
\[
(\Gamma_1^2 - \Gamma_2^2)\tilde{\psi} = (t^2-y^2)\big(P_c G - \frac{1}{2}\frac{3+\frac{y^2}{2}}{(1+y^2)^2}\tilde{\psi}\big).
\]
Then use the bound (for suitable $\delta>0$)
\[
\big|P_c G\big|\lesssim \langle y\rangle^{-\frac{3}{2}}\big|\langle\nabla_{t,y}\rangle^2\phi\big|^2 + \langle y\rangle^{\frac{1}{2}}|\nabla_{t,y}\phi|^2\nabla_{t,y}^2\phi + \eps e^{-\delta |y|}.
\]
It remains to verify that the weight $t^2-y^2$ may be absorbed in the cubic terms. Note that for $|\beta|\leq N_1-1$, we have by the Sobolev embedding $H^1(\R)\hookrightarrow L^\infty(\R)$ and bootstrap assumption \eqref{eq:ener1boot}
\[
\|\langle y\rangle^{\frac{1}{2}}\nabla_{t,y}^{\beta}\nabla_{t,y}\phi\|_{L^\infty}\lesssim \eps\langle t\rangle^{\nu},
\]
while from \eqref{eq:estimatetmyextregion}, we know that 
\[
\big\|(t-y)\nabla_{t,y}^{\beta}\nabla_{t,y}^2\tilde{\psi}\big\|_{L^2_{dy}(y>t)}\lesssim \eps \langle t\rangle^{10\nu},\,|\beta|<\frac{N_1}{2}.
\]
It then follows that for any $|\beta|\leq N_1-1$, we have
\begin{align*}
\big\|\nabla^\beta_{t,y}\big((t^2-y^2)\langle y\rangle^{\frac{1}{2}}|\nabla_{t,y}\phi|^2\nabla_{t,y}^2\phi\big)\big\|_{L^2_{dy}(y>t)}\lesssim \eps^3\langle t\rangle^{12\nu}.
\end{align*}
Finally, bootstrap assumption \eqref{eq:ener1boot} yields
\[
\big\|\nabla_{t,y}\nabla_{t,y}^{\beta}\big((t^2-y^2)\frac{1}{2}\frac{3+\frac{y^2}{2}}{(1+y^2)^2}\tilde{\psi}\big)\big\|_{L^2_{dy}(y>t)}\lesssim \eps\langle t\rangle^{\nu},\,|\beta|\leq N_1 - 2.
\]
It now follows that for $|\beta|+2\leq N_1$, we have 
\begin{align*}
\big\|\nabla_{t,y}\nabla_{t,y}^{\beta}\big((\Gamma_1^2 - \Gamma_2^2)\tilde{\psi}\big)\big\|_{L^2_{dy}(y>t)}\lesssim \eps\langle t\rangle^{12\nu}\lesssim \eps\langle t\rangle^{\frac{3}{2}-\delta_1}.
\end{align*}
The estimates in {\it(a), (b)} complete the proof of the first estimate of the lemma for $\kappa = 2$. 
\\

{\it{(3): Proof of the second inequality of the lemma.}} We have the following identity 
\begin{align*}
\Gamma_1\Gamma_2 - \Gamma_2^2 = (t-y)^2\partial^2_{ty} - (t-y)^2\partial_t^2 + (ty - y^2)(\partial_y^2 - \partial_t^2) + \Gamma_1 - \Gamma_2.
\end{align*}
Note that 
\begin{align*}
\tilde{\psi}_{ty} - \tilde{\psi}_{tt} = \frac{(\Gamma_2\tilde{\psi})_y  - (\Gamma_2\tilde{\psi})_t - \tilde{\psi}_y + \tilde{\psi}_t - y(\tilde{\psi}_{yy} - \tilde{\psi}_{tt})}{t-y}
\end{align*}
Then, using simple variations of the estimates above, in particular the structure of $\tilde{\psi}_{yy} - \tilde{\psi}_{tt}$,  one concludes that
\begin{align*}
\big\|\nabla_{t,y}^{\beta}\big((\Gamma_1\Gamma_2 - \Gamma_2^2)\tilde{\psi}\big)\big\|_{L^2_{dy}(y\lesssim t)}\lesssim \eps\langle t\rangle^{1+(1+[\frac{2|\beta|}{N_1}])10\nu},\,2\leq |\beta|+1\leq N_1,
\end{align*}
which in light of the a priori bound on $\Gamma_2^2\tilde{\psi}$ implies the second estimate of the lemma in the region $y\lesssim t$. In the region $y\gg t$, one uses Lemma~\ref{lemma:appB} to estimate $\|(t-y)^2\nabla_{t,y}^3\tilde{\psi}\|_{L^2(y\gg t)}$ directly. Note that the proof of the latter actually allows us estimate $\|(t-y)^2\nabla_{t,y}^3\tilde{\psi}\|_{L^2(y\gg t)}$ also in the region $y\sim t$. 
\end{proof}

\begin{rem}\label{rem:1} The preceding proof reveals that for $\Gamma^\kappa$ any product of at most two of the vector fields $\Gamma_1, \Gamma_2$, we have 
\[
\big\|\nabla_{t,y}\nabla_{t,y}^{\beta}\Gamma^\kappa\tilde{\psi}\big\|_{L_y^2(y\gtrsim t)}\lesssim \eps\langle t\rangle^{(1+[\frac{2|\beta|}{N_1}])100\nu},\,|\beta|+|\kappa|\leq N_1,\, \kappa\in\{1,2\}.
\]
\end{rem}

\begin{lem}\label{lem:subtle}
We can split $\Gamma_1\tilde{\psi} = (\Gamma_1\tilde{\psi})_1 + (\Gamma_1\tilde{\psi})_2$, where we have
\[
\big\|\langle \log y\rangle^{-1}\nabla_{t,y}^{\beta}(\Gamma_1\tilde{\psi})_1\big\|_{L_y^2(y\ll t)}\lesssim \eps\langle t\rangle^{\frac{1}{2}-\delta_1}
\]
provided have $0\leq |\beta|\leq \frac{N_1}{2}+C$, while we have 
\[
\big\|\nabla_{t,y}\nabla_{t,y}^{\beta}(\Gamma_1\tilde{\psi})_2\big\|_{L_y^2(y\ll t)}\lesssim \eps \langle t\rangle^{(1+[\frac{2|\beta|}{N_1}])100\nu},\,|\beta|+2\leq N_1.
\]
Moreover, there is a splitting $\Gamma_1^2\tilde{\psi} = (\Gamma_1^2\tilde{\psi})_1 + (\Gamma_1^2\tilde{\psi})_2$, with 
\[
\big\|\langle \log y\rangle^{-1}\nabla_{t,y}^{\beta}(\Gamma_1^2\tilde{\psi})_1\big\|_{L_y^2(y\ll t)}\lesssim \eps\langle t\rangle^{\frac{3}{2}-\delta_1},\,0\leq |\beta|\leq \frac{N_1}{2}+C,
\]
as well as 
\[
\big\|\nabla_{t,y}\nabla_{t,y}^{\beta}(\Gamma_1^2\tilde{\psi})_2\big\|_{L_y^2(y\ll t)}\lesssim \eps \langle t\rangle^{(1+[\frac{2|\beta|}{N_1}])100\nu},\,|\beta|+2\leq N_1.
\]
Finally, there is a splitting $\Gamma_1\Gamma_2\tilde{\psi} = (\Gamma_1\Gamma_2\tilde{\psi})_1 + (\Gamma_1\Gamma_2\tilde{\psi})_2$, with 
\[
\big\|\langle \log y\rangle^{-1}\nabla_{t,y}^{\beta}(\Gamma_1\Gamma_2\tilde{\psi})_1\big\|_{L_y^2(y\ll t)}\lesssim \eps\langle t\rangle^{1+(1+[\frac{2|\beta|}{N_1}])10\nu},\,0\leq |\beta|\leq \frac{N_1}{2}+C,
\]
as well as 
\[
\big\|\nabla_{t,y}\nabla_{t,y}^{\beta}(\Gamma_1\Gamma_2\tilde{\psi})_2\big\|_{L_y^2(y\ll t)}\lesssim \eps \langle t\rangle^{\frac{1}{2}-\delta_1},\,|\beta|+2\leq N_1.
\]

\end{lem}

\begin{proof}
In fact, using \eqref{eq:technical3}, we get 
\begin{align*}
\Gamma_1\tilde{\psi}& = \Gamma_1\big(\tilde{\phi}_2(y)\int_0^y\tilde{\phi}_1(\tilde{y})\big[\tilde{\psi}_{tt}(t, \tilde{y}) + P_cG(t, \tilde{y})\big]\,d\tilde{y}\big)\\
&-\Gamma_1\big(\tilde{\phi}_1(y)\int_0^y\tilde{\phi}_2(\tilde{y})\big[\tilde{\psi}_{tt}(t, \tilde{y}) + P_cG(t, \tilde{y})\big]\,d\tilde{y}\big)\\
&+\Gamma_1\big(a(t)\tilde{\phi}_1(y)\big).
\end{align*}
Here we have 
$$a'(t) = \partial_t\tilde{\psi}(t,0) = t^{-1}\Gamma_2\tilde{\psi}(t,0),\, a''(t) =   \partial_t^2\tilde{\psi}(t,0) = t^{-2}(\Gamma_2^2\tilde{\psi} - \Gamma_2\tilde{\psi})(t,0),$$
and so using the bound \eqref{eq:technical01} (proved independently below), we get 
\[
\big\|\nabla_{t,y}\nabla_{t,y}^{\beta}\big(ya'(t)\tilde{\phi}_1\big)\big\|_{L^2_{dy}(y\ll t)}\lesssim \eps\langle t\rangle^{(1+[\frac{2|\beta|}{N_1}])100\nu},\,|\beta|+2\leq N_1,
\]
while we have 
\[
\big\|\langle \log y\rangle^{-1}ta(t)\tilde{\phi}_1'(y)\big\|_{L_y^2(y\ll t)}\lesssim \eps\langle t\rangle^{\frac{1}{2}-\delta_1}.
\]
Further, write 
\begin{align*}
& \Gamma_1\big(\tilde{\phi}_2(y)\int_0^y\tilde{\phi}_1(\tilde{y})\big[\tilde{\psi}_{tt}(t, \tilde{y}) + P_cG(t, \tilde{y})\big]\,d\tilde{y}\big)\\
&-\Gamma_1\big(\tilde{\phi}_1(y)\int_0^y\tilde{\phi}_2(\tilde{y})\big[\tilde{\psi}_{tt}(t, \tilde{y}) + P_cG(t, \tilde{y})\big]\,d\tilde{y}\big)\\
&=: I + II,
\end{align*}
where 
\begin{align*}
I &= t\tilde{\phi}_2'(y)\int_0^y\tilde{\phi}_1(\tilde{y})\big[\tilde{\psi}_{tt}(t, \tilde{y}) + P_cG(t, \tilde{y})\big]\,d\tilde{y}\\& -  t\tilde{\phi}_1'(y)\int_0^y\tilde{\phi}_2(\tilde{y})\big[\tilde{\psi}_{tt}(t, \tilde{y}) + P_cG(t, \tilde{y})\big]\,d\tilde{y},
\end{align*}
\begin{align*}
II&= y\tilde{\phi}_2(y)\int_0^y\tilde{\phi}_1(\tilde{y})\big[\tilde{\psi}_{ttt}(t, \tilde{y}) + (P_cG)_t(t, \tilde{y})\big]\,d\tilde{y}\\& -  y\tilde{\phi}_1(y)\int_0^y\tilde{\phi}_2(\tilde{y})\big[\tilde{\psi}_{ttt}(t, \tilde{y}) + (P_cG)_t(t, \tilde{y})\big]\,d\tilde{y}.
\end{align*}
In view of Lemma \ref{lemma:appB}, we have
\[
\big\|\nabla_{t,y}^{\beta}\tilde{\psi}_{ttt}\big\|_{L_y^2(y\ll t)}\lesssim \eps\langle t\rangle^{(1+[\frac{2|\beta|}{N_1}])100\nu-2},\,|\beta|+2\leq N_1.
\]
It then follows that 
\[
\big\|\nabla_{t,y}\nabla_{t,y}^{\beta}I\big\|_{L^2_{dy}(y\ll t)}\lesssim \eps\langle t\rangle^{(1+[\frac{2|\beta|}{N_1}])100\nu},\,|\beta|+2\leq N_1
\]
For the term $II$ above, observe that 
\begin{align*}
II_y&= (y\tilde{\phi}_2)'(y)\int_0^y\tilde{\phi}_1(\tilde{y})\big[\tilde{\psi}_{ttt}(t, \tilde{y}) + (P_cG)_t(t, \tilde{y})\big]\,d\tilde{y}\\& -  (y\tilde{\phi}_1)'(y)\int_0^y\tilde{\phi}_2(\tilde{y})\big[\tilde{\psi}_{ttt}(t, \tilde{y}) + (P_cG)_t(t, \tilde{y})\big]\,d\tilde{y}
\end{align*}
which can be estimated just like $I_t$. Finally, we have 
\begin{align*}
II_t&= (y\tilde{\phi}_2)(y)\int_0^y\tilde{\phi}_1(\tilde{y})\big[\tilde{\psi}_{tttt}(t, \tilde{y}) + (P_cG)_{tt}(t, \tilde{y})\big]\,d\tilde{y}\\& -  (y\tilde{\phi}_1)(y)\int_0^y\tilde{\phi}_2(\tilde{y})\big[\tilde{\psi}_{tttt}(t, \tilde{y}) + (P_cG)_{tt}(t, \tilde{y})\big]\,d\tilde{y}
\end{align*}
Then use the equation to write $\tilde{\psi}_{tttt} = \tilde{\psi}_{ttyy} + V\tilde{\psi}_{tt} + l.o.t.$. 
Performing an integration by parts, this allows us to write 
\begin{align*}
II_t&= (y\tilde{\phi}_2)(y)\int_0^y-\tilde{\phi}_1'(\tilde{y})\tilde{\psi}_{tt\tilde{y}}(t, \tilde{y}) + \tilde{\phi}_1(\tilde{y})\big[V(\tilde{y})\tilde{\psi}_{tt} + (P_cG)_{tt}(t, \tilde{y})\big]\,d\tilde{y}\\& -  (y\tilde{\phi}_1)(y)\int_0^y-\tilde{\phi}_2'(\tilde{y})\tilde{\psi}_{tt\tilde{y}}(t, \tilde{y}) + \tilde{\phi}_2(\tilde{y})\big[V(\tilde{y})\tilde{\psi}_{tt} + (P_cG)_{tt}(t, \tilde{y})\big]\,d\tilde{y}
\end{align*}
Using 
\[
\big\|\nabla_{t,y}^{\beta}\tilde{\psi}_{tty}\big\|_{L_y^2(y\ll t)}\lesssim \eps\langle t\rangle^{(1+[\frac{2|\beta|}{N_1}])100\nu-2},\,|\beta|+2\leq N_1,
\]
(see Lemma \ref{lemma:appB}) as well as the identity 
\[
\tilde{\psi}_{tt} = (t^2-y^2)^{-1}\big[\Gamma_2^2\tilde{\psi} - 2y\partial_y\Gamma_2\tilde{\psi} + 2y\partial_y\tilde{\psi} - \Gamma_2\tilde{\psi}+y^2\square\tpsi\big]
\]
one gets 
\begin{align*}
\big\|\nabla_{t,y}\nabla_{t,y}^{\beta}II\big\|_{L_y^2(y\ll t)}\lesssim \eps\langle t\rangle^{(1+[\frac{2|\beta|}{N_1}])100\nu},\,|\beta|+2\leq N_1. 
\end{align*}

Next, consider $\Gamma_1^2\tilde{\psi}$. Recall the identity
\[
(\Gamma_1^2 - \Gamma_2^2)\tilde{\psi} = (t^2-y^2)\square\tpsi.
\]
This yields for $|\beta|+1\leq N_1$:
\begin{eqnarray*}
\big\|\nabla_{t,y}^{\beta}\big((\Gamma_1^2 - \Gamma_2^2)\tilde{\psi}\big)\big\|_{L_y^2(y\ll t)}&\lesssim& t^2\|\nabla_{t,y}^{\beta}(V(\cdot)\tpsi)\|_{L^2_y}+l.o.t\\
&\lesssim&  \eps\langle t\rangle^{\frac{3}{2}-\delta_1},
\end{eqnarray*}
where we used in particular bootstrap assumption \eqref{eq:disp2boot}. Together with the bootstrap assumption \eqref{eq:ener2boot}, we conclude that we can split 
\[
\Gamma_1^2\tilde{\psi} = (\Gamma_1^2\tilde{\psi})_1 + (\Gamma_1^2\tilde{\psi})_2
\]
with the desired properties. \\
The proof of the last assertion of the lemma follows from (3) in the preceding proof. 
\end{proof}

We now continue with the proof of \eqref{eq:ener2}, our main tools being Proposition~\ref{eq:keyest1}, Proposition~\ref{eq:keyest2}. Write the equation for $\tilde{\psi}$ as before in the form 
\[
-\partial_t^{2}\tilde{\psi} + \partial_y^2\tilde{\psi}+ \frac{1}{2}\frac{3+\frac{y^2}{2}}{(1+y^2)^2}\tilde{\psi} = P_cG,
\]
where 
\[
G = (1+y^2)^{\frac{1}{4}}F(\phi , \nabla\phi ,\nabla^2\phi ).
\]
We decompose $G$ into its weighted part $G_1$ (terms with weights at least $\langle y\rangle^{-2}$), as well as the pure cubic part $G_2$, 
\[
G = G_1 + G_2.
\]
Use the bound
\begin{align*}
\big|\nabla_{t,y}^{\beta}\Gamma_2^\kappa G_1\big|&\lesssim \sum_{\substack{\kappa_1+\kappa_2\leq \kappa\\|\beta_1|+|\beta_2|\leq |\beta|+2}}\big|\frac{\nabla_{t,y}^{\beta_1}\Gamma_2^{\kappa_1}\phi \nabla_{t,y}^{\beta_2}\Gamma_2^{\kappa_2}\phi }{(1+y^2)^{\frac{3}{4}}}\big|.
\end{align*}
According to \eqref{eq:weightedenergy2}, we need to bound the right-hand side in $\big\|\cdot\big\|_{L_t^1(L^2_{dy}\cap L^1_{\langle y\rangle^{\eps}dy})}$. Start with the case of less than top-level derivatives, $|\beta|+\kappa\leq N_1-1$. When $\kappa = 1$, in view of bootstrap assumptions \eqref{eq:disp2boot} and \eqref{eq:localenerboot}, the above expression is bounded by 
\begin{eqnarray}\label{eq:technical03}
&&\big\|\cdot\big\|_{L_t^1(L^2_{dy}\cap L^1_{\langle y\rangle^{\eps_*}dy})([0,T])}\\
\nonumber&\lesssim& \big\|\frac{\langle\nabla_{t,y}\rangle^{\frac{N_1}{2}}\phi}{\langle y\rangle^{\frac{1}{2}-}}\big\|_{L_t^2 L_y^\infty}\big\|\langle y\log y\rangle^{-1}\langle\nabla_{t,y}\rangle^{N_1}\langle\Gamma_2\rangle\tilde{\phi}\big\|_{L^2_{y,t}([0,T])}\\
\nonumber&& +\sum_{\beta_2<\frac{N_1}{2}}\big\|\langle y\log y\rangle^{-1}\langle\nabla_{t,y}\rangle^{N_1}\tilde{\phi}\big\|_{L^2_{y,t}([0,T])}\big\|\langle y\log y\rangle^{-1}\nabla_{t,y}^{\beta_2}\langle\Gamma_2\rangle\tilde{\phi}\big\|_{L^2_{y,t}([0,T])}\\
\nonumber&\lesssim& \eps^2\langle
T\rangle^{(1+[\frac{2|\beta|}{N_1}])10\nu} + \eps^2\langle
T\rangle^{11\nu}\mathbf{1}_{\{|\beta|>\frac{N_1}{2}\}}\\
\nonumber&\lesssim& \eps^2\langle T\rangle^{(1+[\frac{2|\beta|}{N_1}])10\nu},
\end{eqnarray}
as required. \\

The case $\kappa = 2$ is estimated, in view of bootstrap assumptions \eqref{eq:disp2boot} and \eqref{eq:localenerboot}, as follows 
\begin{eqnarray}\label{eq:technical04}
&&\big\|\cdot\big\|_{L_t^1(L^2_{dy}\cap L^1_{\langle y\rangle^{\eps_*}dy})([0,T])}\\
\nonumber&\lesssim& \big\|\frac{\langle\nabla_{t,y}\rangle^{\frac{N_1}{2}}\phi}{\langle y\rangle^{\frac{1}{2}-}}\big\|_{L_t^2 L_y^\infty}\big\|\langle y\log y\rangle^{-1}\langle\nabla_{t,y}\rangle^{N_1-1}\langle\Gamma_2\rangle^2\tilde{\phi}\big\|_{L^2_{y,t}([0,T])}\\
\nonumber&& +\sum_{\beta_2<\frac{N_1}{2}-1}\big\|\langle y\log y\rangle^{-1}\langle\nabla_{t,y}\rangle^{N_1-1}\tilde{\phi}\big\|_{L^2_{y,t}([0,T])}\big\|\langle y\log y\rangle^{-1}\nabla_{t,y}^{\beta_2}\langle\Gamma_2\rangle^2\tilde{\phi}\big\|_{L^2_{y,t}([0,T])}\\
\nonumber&&+\sum_{|\beta_1|+|\beta_2|\leq N_1-1}\prod_{j=1,2}\big\|\langle y \log y\rangle^{-1}\nabla_{t,y}^{\beta_j}\langle\Gamma_2\rangle\tilde{\phi}\big\|_{L^2_{y,t}([0,T])}\\
\nonumber&\lesssim& \eps^2\langle
T\rangle^{(1+[\frac{2|\beta|}{N_1}])100\nu} + \eps^2\langle
T\rangle^{101\nu}\mathbf{1}_{\{|\beta|>\frac{N_1}{2}\}} + \eps^2\langle T\rangle^{(2+[\frac{2|\beta|}{N_1}])10\nu}\\
\nonumber&\lesssim& \eps^2\langle T\rangle^{(1+[\frac{2|\beta|}{N_1}])100\nu},
\end{eqnarray} 
as required. \\

The case of top level derivatives $|\beta|+\kappa = N_1$ is treated as
in Section \ref{sect:ener1toporder} via integration by parts and induction on the number of $y$-derivatives, and omitted. 
\\

This leads us to the problem of bounding the contribution of the pure
cubic terms $G_2$. By using the inherent gradient structure
\eqref{eq:keygradientstructure2}, as well as the estimates \eqref{eq:weightedenergy2}, \eqref{eq:weightedenergy3} and \eqref{eq:weightedenergy4}, 
we reduce to bounding the schematic expressions 
\begin{align*}
&\big\|\nabla_{t,y}^{\beta}\langle\Gamma_2\rangle^{\kappa}\big(\phi_{t,y}^2\tilde{\phi}_t\big)\big\|_{L_t^1L^2_{dy}},\,\big\|\langle y\rangle^{-1}\langle\Gamma_2\rangle^{\kappa}\nabla_{t,y}^{\beta}\big(\phi_{t,y}^2\tilde{\phi}_t\big)\big\|_{L_t^1(L^2_{dy}\cap L^1_{\langle y\rangle^{\eps_*}dy})},\\&\big\|\langle y\rangle^{\frac{1}{2}}\langle\Gamma_2\rangle^{\kappa}\nabla_{t,y}^{\beta}\big(\phi_t^2(\phi_{yy}-\phi_{tt})\big)\big\|_{L_t^1(L^2_{dy}\cap L^1_{\langle y\rangle^{\eps_*}dy})}. 
\end{align*}
We shall only consider the case of non-top order derivatives, i.\ e.\ $|\beta|+\kappa<N_1$, since the remaining case is again handled via the energy identity and the integration by parts trick to reduce to the case of lower order derivatives. We treat the above terms separately: 
\\

{\it{(1): the bound for $\big\|\nabla_{t,y}^{\beta}\langle\Gamma_2\rangle^{\kappa}\big(\phi_{t,y}^2\tilde{\phi}_t\big)\big\|_{L_t^1L^2_{dy}([0,T])}$.}} Start with the case $\kappa = 1$, $|\beta|< \frac{N_1}{2}$. We have
\begin{equation}\label{eq:technical4}\begin{split}
\big\|\nabla_{t,y}^{\beta}\langle\Gamma_2\rangle^{\kappa}\big(\phi_{t,y}^2\tilde{\phi}_t\big)\big\|_{L_t^1L^2_{dy}([0,T])}&\lesssim \sum_{\sum\beta_j = \beta}\big\|\nabla_{t,y}^{\beta_1}\langle\Gamma_2\rangle\phi_{t,y}\nabla_{t,y}^{\beta_2}\phi_{t,y}\nabla_{t,y}^{\beta_3}\tilde{\phi}_t\big\|_{L_t^1L^2_{dy}([0,T])}\\
&+\sum_{\sum\beta_j = \beta}\big\|\nabla_{t,y}^{\beta_1}\phi_{t,y}\nabla_{t,y}^{\beta_2}\phi_{t,y}\nabla_{t,y}^{\beta_3}\langle\Gamma_2\rangle\tilde{\phi}_t\big\|_{L_t^1L^2_{dy}([0,T])}\\
&\lesssim \sum_{\sum\beta_j \leq \beta}\big\|\nabla_{t,y}^{\beta_1}\langle\Gamma_2\rangle\tilde{\phi}_{t,y}\nabla_{t,y}^{\beta_2}\phi_{t,y}\nabla_{t,y}^{\beta_3}\phi_t\big\|_{L_t^1L^2_{dy}([0,T])}\\
&+\sum_{\sum\beta_j \leq \beta}\big\|\langle y\rangle^{-1}\nabla_{t,y}^{\beta_1}\langle\Gamma_2\rangle\tilde{\phi}\nabla_{t,y}^{\beta_2}\phi_{t,y}\nabla_{t,y}^{\beta_3}\phi_t\big\|_{L_t^1L^2_{dy}([0,T])}.
\end{split}\end{equation}
We estimate the first term in the right-hand side of \eqref{eq:technical4}. Using the bootstrap assumptions \eqref{eq:disp1boot} and \eqref{eq:ener2boot}, we have
\begin{align*}
&\sum_{\sum\beta_j \leq \beta}\big\|\nabla_{t,y}^{\beta_1}\langle\Gamma_2\rangle\tilde{\phi}_{t,y}\nabla_{t,y}^{\beta_2}\phi_{t,y}\nabla_{t,y}^{\beta_3}\phi_t\big\|_{L_t^1L^2_{dy}([0,T])}\\&\lesssim \sum_{\sum\beta_j \leq \beta}\big\|\langle t\rangle^{-10\nu}\nabla_{t,y}^{\beta_1}\langle\Gamma_2\rangle\tilde{\phi}_{t,y}\big\|_{L_t^\infty L_y^2([0,T])}\big\|\langle t\rangle^{10\nu}\nabla_{t,y}^{\beta_2}\phi_{t,y}\nabla_{t,y}^{\beta_3}\phi_t\big\|_{L_t^1 L_y^\infty([0,T])}\\
&\lesssim \frac{\eps^3}{\nu} \langle T\rangle^{10\nu}\lesssim \eps^2 \langle T\rangle^{10\nu}.
\end{align*}
To estimate the second term in \eqref{eq:technical4}, we use the local energy bound \eqref{eq:localenerboot}. Write
\begin{equation}\label{eq:technical00}\begin{split}
\big|\nabla_{t,y}^{\beta_2}\phi_{t,y}\big|& = \big|\nabla_{t,y}^{\beta_2}\big(\langle y\rangle^{-\frac{3}{2}}\tilde{\phi})\big| + \big|\nabla_{t,y}^{\beta_2}\big(\langle y\rangle^{-\frac{1}{2}}\tilde{\phi}_y\big)\big|\\
&\leq \sum_{|\tilde{\beta}_2|\leq |\beta_2|}\big[\big|\langle y\rangle^{-\frac{3}{2}}\nabla_{t,y}^{\tilde{\beta}_2}\tilde{\phi}\big| + \big|\langle y\rangle^{-\frac{1}{2}}\nabla_{t,y}^{\tilde{\beta}_2}\tilde{\phi}_y\big|\big],
\end{split}\end{equation}
and so
\begin{align*}
&\big\|\langle y\rangle^{-1}\nabla_{t,y}^{\beta_1}\langle\Gamma_2\rangle\tilde{\phi}\nabla_{t,y}^{\beta_2}\phi_{t,y}\nabla_{t,y}^{\beta_3}\phi_t\big\|_{L_t^1L^2_{dy}([0,T])}\\
&\lesssim \big\|\langle y\log y\rangle^{-1}\nabla_{t,y}^{\beta_1}\langle\Gamma_2\rangle\tilde{\phi}\big\|_{L_{t,y}^2([0,T])}\big(\sum_{|\tilde{\beta}_2|\leq |\beta_2|+1}\big\|\langle t\rangle^{-\nu}\langle\nabla_y\rangle\nabla_{t,y}^{\tilde{\beta}_2}\tilde{\phi}\big\|_{L_t^\infty L^\infty_{y}([0,T])}\big)\\
&\hspace{7cm}\cdot\big\|\langle t\rangle^{\nu}\langle\log y\rangle \langle y\rangle^{-\frac{1}{2}}\nabla_{t,y}^{\beta_3}\phi_t\big\|_{L_{t}^2L_y^\infty}\\
&\lesssim \eps^3\langle T\rangle^{10\nu},
\end{align*}
where we used the bootstrap assumption \eqref{eq:localenerboot} for the first term, the bootstrap assumption \eqref{eq:disp2boot} and interpolation for the last term, and the embedding $H^1(\R)\subset L^{\infty}(\R)$ and the bootstrap assumption \eqref{eq:ener1boot} for the middle term.\\

We continue with the case $\kappa = 1, |\beta|\geq\frac{N_1}{2}$. Again using \eqref{eq:technical4}, there may now be terms where only one of the three factors may be bounded in  $L_{t,y}^\infty$. Start with the first term, and assume $|\beta_2|>\frac{N_1}{2}$ (as we may by symmetry and since else we can argue as in the previous bounds). Then distinguish between the following two situations: 
\\

{\it{(a): $y\ll t$.}} Here the trick is to use the identities
\[
\tilde{\phi}_t = \frac{ t\Gamma_2\tilde{\phi} - y\Gamma_1\tilde{\phi}}{t^2 - y^2},\,\tilde{\phi}_y = \frac{ t\Gamma_1\tilde{\phi} - y\Gamma_2\tilde{\phi}}{t^2 - y^2}
\]
which imply 
\begin{equation}\label{eq:estimphiyllt}
\big|\nabla_{t,y}^{\beta_2}\phi_{t,y}\big|\lesssim \langle y\rangle^{-\frac{3}{2}}\big|\tilde{\phi}\big| + t^{-1}\langle y\rangle^{-\frac{1}{2}}\sum_{\substack{|\tilde{\beta}_2|< |\beta_2|\\\Gamma = \Gamma_{1,2}}}\big|\nabla_{t,y}^{\tilde{\beta}_2}\Gamma\tilde{\phi}\big|.
\end{equation}
To estimate the term $\big|\nabla_{t,y}^{\tilde{\beta}_2}\Gamma\tilde{\phi}\big|$, we observe $\Gamma_1\tilde{\phi}(t, 0) = 0$, whence using Lemma~\ref{lem:3} we get 
\[
\big|\nabla_{t,y}^{\tilde{\beta}_2}\Gamma_1\tilde{\phi}(t, y)\big|\lesssim \eps\langle y\rangle^{\frac{1}{2}}\langle t\rangle^{\frac{1}{2}-\delta_1}.
\]
We also have (see Lemma \ref{lem:4})
\[
\big|\nabla_{t,y}^{\beta}\Gamma_2^{\kappa}\tilde{\psi}\big|(t,y)\lesssim \eps\langle t\rangle^{(1+[\frac{2|\beta|}{N_1}])10^{\kappa}\nu}\langle y\rangle^{\frac{1}{2}},\,|\beta|+\kappa\leq N_1.
\]
The previous observations imply that 
\[
\big|\nabla_{t,y}^{\beta_2}\phi_{t,y}\big|\lesssim \langle y\rangle^{-\frac{3}{2}}\big|\tilde{\phi}\big| + \eps t^{-\frac{1}{2}-\delta_1}\lesssim \eps t^{-\frac{1}{2}-\delta_1},\,y\ll t,
\]
with a similar bound applying to $\nabla_{t,y}^{\beta_3}\phi_{t}$. But then we easily get
\begin{align*}
&\big\|\nabla_{t,y}^{\beta_1}\langle\Gamma_2\rangle\tilde{\phi}_{t,y}\nabla_{t,y}^{\beta_2}\phi_{t,y}\nabla_{t,y}^{\beta_3}\phi_t\big\|_{L_t^1L^2_{dy}(y\ll t)}\\
&\lesssim \big\|\langle t\rangle^{-10\nu}\nabla_{t,y}^{\beta_1}\langle\Gamma_2\rangle\tilde{\phi}_{t,y}\big\|_{L_t^\infty L_y^2}\big\|\langle t\rangle^{10\nu}\nabla_{t,y}^{\beta_2}\phi_{t,y}\big\|_{L_t^{2}L_{y}^\infty(y\ll t)}\big\|\nabla_{t,y}^{\beta_3}\phi_t\big\|_{L_t^{2}L_{y}^\infty(y\ll t)}\\
&\lesssim \frac{\eps^3}{\delta_1-10\nu}\lesssim \eps^2.
\end{align*}
The remaining term in \eqref{eq:technical4} is treated similarly. 
\\

{\it{(b): $y\gtrsim t$.}} Here we may of course assume $y\sim t$, since the case $y\gg t$ is handled just like {\it(a)}. Note that from \eqref{eq:technical00}, we get using also the Sobolev embedding $H^1_{ydy}(\R)\hookrightarrow L^\infty(\R)$ and the bootstrap assumption \eqref{eq:ener1boot}
\begin{equation}\label{eq:neartheconeestimate}
\big|\nabla_{t,y}^{\beta_2}\phi_{t,y}\big|\lesssim \eps t^{\nu-\frac{1}{2}},\,y\sim t,
\end{equation}
and so the a priori bounds imply 
\begin{align*}
&\big\|\nabla_{t,y}^{\beta_1}\langle\Gamma_2\rangle\tilde{\phi}_{t,y}\nabla_{t,y}^{\beta_2}\phi_{t,y}\nabla_{t,y}^{\beta_3}\phi_t\big\|_{L_t^1L^2_{dy}(y\sim t),\,t\in [0,T]}\\
&\lesssim \big\|\nabla_{t,y}^{\beta_1}\langle\Gamma_2\rangle\tilde{\phi}_{t,y}\big\|_{L_t^\infty L_y^2[0,T]}\big\|\nabla_{t,y}^{\beta_2}\phi_{t,y}\nabla_{t,y}^{\beta_3}\phi_t\big\|_{L_t^1 L_y^\infty([0,T]}\lesssim \frac{\eps^3}{\nu}\langle T\rangle^{11\nu}\leq \eps^2\langle T\rangle^{20\nu}
\end{align*}
which is the required bound. 
\\

For the second term in \eqref{eq:technical4}, again assuming that $|\beta_2|\geq\frac{N_1}{2}$, we get using \eqref{eq:neartheconeestimate} and the bootstrap assumptions \eqref{eq:disp1boot} and \eqref{eq:localenerboot}(and restricting to $y\sim t$)
\begin{align*}
&\big\|\langle y\rangle^{-1}\nabla_{t,y}^{\beta_1}\langle\Gamma_2\rangle\tilde{\phi}\nabla_{t,y}^{\beta_2}\phi_{t,y}\nabla_{t,y}^{\beta_3}\phi_t\big\|_{L_t^1L^2_{dy}([0,T])}\\
&\lesssim \big\|\langle\log y\rangle^{-1}\langle y\rangle^{-1}\nabla_{t,y}^{\beta_1}\langle\Gamma_2\rangle\tilde{\phi}\big\|_{L_{t,y}^2([0,T])}\big\|\langle \log y\rangle \nabla_{t,y}^{\beta_2}\phi_{t,y}\big\|_{L_{t,y}^\infty}\big\|\nabla_{t,y}^{\beta_3}\phi_t\big\|_{L_{t}^2L_y^\infty([0,T])}\\
&\lesssim \eps^3  \langle \log T\rangle^{\frac{1}{2}}T^{11\nu},
\end{align*}
which is much better than the bound $\eps\langle T\rangle^{20\nu}$ we need. 
\\

This completes the case $\kappa = 1$ for {\it(1)}. For the case $\kappa = 2$, one proceeds analogously, but now also encounters terms of the form 
\[
\nabla_{t,y}^{\beta_1}\langle\Gamma_2\rangle\phi_{t,y}\nabla_{t,y}^{\beta_2}\langle\Gamma_2\rangle\phi_{t,y}\nabla_{t,y}^{\beta_3}\tilde{\phi}_t,
\]
In the region $y\ll t$ or $y\gg t$, we can proceed for it like in (a) above, applied to the factor $\nabla_{t,y}^{\beta_3}\tilde{\phi}_t$. In the region $y\sim t$, 
one uses 
\[
\big\|\nabla_{t,y}^{\beta_2}\langle\Gamma_2\rangle\phi_{t,y}\nabla_{t,y}^{\beta_3}\phi_t\big\|_{L_t^1 L_y^\infty([0,T])}\lesssim \eps^2\langle T\rangle^{21\nu};
\]
We omit the simple details. 
\\

{\it{(2): the bound for $\big\|\langle y\rangle^{-1}\langle\Gamma_2\rangle^{\kappa}\nabla_{t,y}^{\beta}\big(\phi_{t,y}^2\tilde{\phi}_t\big)\big\|_{L_t^1(L^2_{dy}\cap L^1_{\langle y\rangle^{\eps_*}dy})([0,T])}$.}} The $L_t^1L_y^2$-norm corresponds exactly to the second term in \eqref{eq:technical4} (if $\kappa = 1$, and analogous with $\kappa = 2$), and is easier than the  $L^1$-type bound. Thus consider now the (modified) $L^1_{dy}$-norm. From \eqref{eq:estimphiyllt} and a straightforward modification, we get 
\begin{align*}
&\big|\nabla_{t,y}^{\beta_2}\langle\Gamma_2\rangle\phi_{t,y}\big|\\&\lesssim \langle y\rangle^{-\frac{3}{2}}\big|\langle\Gamma_2\rangle\tilde{\phi}\big| + (\max\{t,y\})^{-1}\langle y\rangle^{-\frac{1}{2}}\sum_{\substack{|\tilde{\beta}_2|< |\beta_2|\\\Gamma = \Gamma_{1,2}}}\big|\nabla_{t,y}^{\tilde{\beta}_2}\Gamma\langle\Gamma_2\rangle\tilde{\phi}\big|,\,y\ll t\,\text{or}\,y\gg t,
\end{align*}
while from \eqref{eq:technical00} we get 
\begin{align*}
\big|\nabla_{t,y}^{\beta_2}\langle\Gamma_2\rangle\phi_{t,y}\big|
&\leq \sum_{|\tilde{\beta}_2|\leq |\beta_2|}\big[\big|\langle y\rangle^{-\frac{3}{2}}\nabla_{t,y}^{\tilde{\beta}_2}\langle\Gamma_2\rangle\tilde{\phi}\big| + \big|\langle y\rangle^{-\frac{1}{2}}\nabla_{t,y}^{\tilde{\beta}_2}\langle\Gamma_2\rangle\tilde{\phi}_y\big|\big]
\end{align*}
which is useful in the region $y\sim t$. Using Lemma \ref{lem:subtle} and the bootstrap assumption \eqref{eq:ener2boot}, we infer 
\begin{align*}
\big|\langle \log y\rangle\langle y\rangle^{\frac{1}{2}+\eps_*}\nabla_{t,y}^{\beta_2}\langle\Gamma_2\rangle\phi_{t,y}\big|
&\leq \eps \langle t\rangle^{\eps_*+(1+[\frac{2|\beta|}{N_1}])10\nu+}
\end{align*}

If we now write (as usual $\kappa\in \{1,2\}$)
\begin{align*}
&\langle y\rangle^{-1}\langle\Gamma_2\rangle^{\kappa}\nabla_{t,y}^{\beta}\big(\phi_{t,y}^2\tilde{\phi}_t\big)\\& = \sum_{\substack{\sum\kappa_j = \kappa,\,\kappa_2\leq \min\{1,\kappa_1\}\\\sum \beta_j = \beta}}\langle y\rangle^{-1}(\nabla_{t,y}^{\beta_1}\langle\Gamma_2\rangle^{\kappa_1}\phi_{t,y})(\nabla_{t,y}^{\beta_2}\langle\Gamma_2\rangle^{\kappa_2}\phi_{t,y})(\nabla_{t,y}^{\beta_3}\langle\Gamma_2\rangle^{\kappa_3}\tilde{\phi}_t)
\end{align*}
then if $\kappa_3 = 1$, $\kappa_1 = 1$, we get 
\begin{align*}
&\big\|\cdot\big\|_{L_t^1 L^1_{\langle y\rangle^{\eps_*}dy}([0,T])}\\&\lesssim \big\|\langle y\rangle^{\eps_*+\frac{1}{2}}\langle\log y\rangle\nabla_{t,y}^{\beta_1}\langle\Gamma_2\rangle\phi_{t,y}\big\|_{L_{t,y}^\infty([0,T])}\big\|\frac{\nabla_{t,y}^{\beta_2}\phi_{t,y}}{\langle\log y\rangle\langle y\rangle^{\frac{1}{2}}}\big\|_{ L_{t,y}^2([0,T])}\big\|\frac{\nabla_{t,y}^{\beta_3}\langle\Gamma_2\rangle\tilde{\phi}_{t}}{\langle y\rangle \langle \log y\rangle}\big\|_{L_{t,y}^2([0,T])}\\
&\lesssim \frac{\eps^3}{\nu}\langle T\rangle^{\eps_*+42\nu}\lesssim \eps^2\langle T\rangle^{100\nu}
\end{align*}
which is as desired; we have used the preceding considerations to bound the first factor. On the other hand, when $\kappa_3 = 2$, we obtain the bound 
\begin{align*}
&\big\|\cdot\big\|_{L_t^1 L^1_{\langle y\rangle^{\eps_*}dy}([0,T])}\\&\lesssim \big\|\langle y\rangle^{\eps_*-\frac{1}{2}}\langle\log y\rangle\nabla_{t,y}^{\beta_1}\phi_{t,y}\nabla_{t,y}^{\beta_2}\tilde{\phi}_{t,y}\big\|_{L_{t,y}^2([0,T])}\big\|\frac{\nabla_{t,y}^{\beta_3}\langle\Gamma_2\rangle^{2}\tilde{\phi}_{t}}{\langle y\rangle \langle \log y\rangle}\big\|_{L_{t,y}^2([0,T])}\\
&\lesssim \eps^3\langle T\rangle^{(1+[\frac{2|\beta|}{N_1}])100\nu}. 
\end{align*}
The remaining combinations are handled similarly and this completes the estimate {\it(2)}. 
\\

{\it{(3): the bound for $\big\|\langle y\rangle^{\frac{1}{2}}\langle\Gamma_2\rangle^{\kappa}\nabla_{t,y}^{\beta}\big(\phi_t^2(\phi_{yy}-\phi_{tt})\big)\big\|_{L_t^1(L^2_{dy}\cap L^1_{\langle y\rangle^{\eps_*}dy})}$.}} Here we use the equation for $\phi$. This produces a term just like in {\it{(2)}}, as well as a further linear term of the form 
\[
\langle y\rangle^{-2}\langle\Gamma_2\rangle^{\kappa}\nabla_{t,y}^{\beta}\big(\phi_t^2\tilde{\phi}\big).
\]
This term is handled like in {\it{(2)}} if we note that 
\begin{align*}
\big\|\langle y\rangle^{-1}\tilde{\phi}\big\|_{L^2_{dy}}&\lesssim \big\|\langle y\rangle^{-1}\tilde{\phi}(0)\big\|_{L^2_{dy}} + \big\||\langle y\rangle^{-1}[\tilde{\phi}(y)-\tilde{\phi}(0)]\big\|_{L^2_{dy}}\\
&\lesssim \big\|\langle y\rangle^{-1}\tilde{\phi}(0)\big\|_{L^2_{dy}} + \big\|\tilde{\phi}_y\big\|_{L_{dy}^2}\lesssim \eps\langle t\rangle^{\nu}.
\end{align*}
This then allows us to reduce the above expression to the following crude schematic form 
\[
\big\|\langle y\rangle^{\frac{1}{2}}\langle\Gamma_2\rangle^{\kappa}\nabla_{t,y}^{\beta}\big(\phi_t^2[\langle y\rangle^{-2}\phi^2 + \phi^3]\big)\big\|_{L_t^1(L^2_{dy}\cap L^1_{\langle y\rangle^{\eps_*}dy})}
\]
which is straightforward to estimate by $\lesssim \eps^4$. 
\\

This concludes the proof of \eqref{eq:ener2}.


\section{Local energy decay}\label{sec:localenergydecay}


The goal of this section is to prove the local energy decay \eqref{eq:localener} for which we use Proposition~\ref{eq:keyest3}. This follows essentially along the same lines as the proof of the estimate \eqref{eq:ener2}, except in the case of top level derivatives, which have to be treated differently. 
\\

{\it(1): derivatives below top degree: $|\beta|+\gamma\leq N_1$ (referring to \eqref{eq:localener}).} We follow the same pattern as in the preceding proof, except that now the 'bad norm' $L^1_{\langle y\rangle^{\eps}dy}$ is replaced by $L^2_{\langle y\rangle^{1+}dy}$. Using the equation for $\tilde{\psi}$ as in the preceding proof and splitting the source into 
\[
G = G_1 + G_2,
\]
we see that in order to control the contribution from $G_1$, we have to bound
$$\sum_{\substack{\kappa_1+\kappa_2\leq \kappa\\|\beta_1|+|\beta_2|\leq |\beta|+2}}\big\|\frac{\nabla_{t,y}^{\beta_1}\Gamma_2^{\kappa_1}\phi \nabla_{t,y}^{\beta_2}\Gamma_2^{\kappa_2}\phi }{(1+y^2)^{\frac{3}{4}}}\big\|_{L_t^1 L^2_{\langle y\rangle^{1+}dy}}.$$
In fact, note that in \eqref{eq:technical03} we obtain $L^1_{dy}$-control by sacrificing one factor $\langle y\rangle^{-\frac{1}{2}}$, and so the $L^2_{\langle y\rangle^{1+}dy}$-norm of the above expressions is bounded exactly by  \eqref{eq:technical03}, \eqref{eq:technical04} (corresponding to $\kappa = 1,2$). 
The same comment applies to the non-gradient terms constituting $G_2$, which can hence be estimated just like in {\it(1)} - {\it{(3)}} of the proof of \eqref{eq:ener2} above. 
\\

{\it(2): derivatives of top degree: $|\beta|+\gamma=N_1+1$ (referring to \eqref{eq:localener}).} The idea is to again use an inductive argument to reduce to the case of lower order derivatives. This time a simple integration by parts argument seems to no longer work, and we instead use an approximate parametrix to express the top order derivative terms. Specifically, assume $\beta+\gamma = N_1$, and consider the expression $\partial_t^\beta\Gamma_2^{\gamma}\tilde{\psi}$. This satisfies the following equation
\begin{eqnarray}\label{eq:technical6}
&&-\partial_t^{2}(\partial_t^\beta\Gamma_2^{\gamma}\tilde{\psi}) + \partial_y^2(\partial_t^\beta\Gamma_2^{\gamma}\tilde{\psi})+ \frac{1}{2}\frac{3+\frac{y^2}{2}}{(1+y^2)^2}(\partial_t^\beta\Gamma_2^{\gamma}\tilde{\psi})\\
\nonumber &=& \partial_t^\beta\Gamma_2^{\gamma}(P_cG) + \pr_t^\beta[\square,\Ga_2^\kappa]\tpsi + \sum_{\tilde{\gamma}<\gamma}V_{\tilde{\gamma}}\partial_t^\beta\Gamma_2^{\tilde{\gamma}}\tilde{\psi}
\end{eqnarray}
where the potentials $V_{\tilde{\gamma}}$ are of the schematic form 
\[
V_{\tilde{\gamma}} = (y\partial_y)^{\gamma - \tilde{\gamma}}\big(\frac{1}{2}\frac{3+\frac{y^2}{2}}{(1+y^2)^2}\big).
\]
Our goal is to derive an a priori bound for 
\begin{equation}\label{eq:toporderlocalen}
\big\|\langle\log y\rangle^{-1}\langle y\rangle^{-1}\nabla_{t,y}\partial_t^\beta\Gamma_2^{\gamma}\tilde{\psi}\big\|_{L_{t,y}^2([0,T])}.
\end{equation}
To this end, we shall express $\partial_t^\beta\Gamma_2^{\gamma}\tilde{\psi}$ via an approximate representation formula (a parametrix) based on the method of characteristics (as we are essentially in $1+1$-dimensions), 
{\it{taking the smaller top order terms in $\partial_t^\beta\Gamma_2^{\gamma}(P_cG)$ into account}}. To start with, write\[
\partial_t^\beta\Gamma_2^{\gamma}(P_cG) = \partial_t^\beta\Gamma_2^{\gamma}(G) - \partial_t^\beta\Gamma_2^{\gamma}(P_dG),
\]
where the error term $\partial_t^\beta\Gamma_2^{\gamma}(P_dG)$ is effectively a lower order term. Then collecting all the top order derivative terms contained in 
\[
\partial_t^\beta\Gamma_2^{\gamma}(G),
\]
we re-cast the equation \eqref{eq:technical6} in the form (we normalize the first coefficient to be equal to $1$, thereby introducing the factor $\kappa(t,y)$ on the right)
\begin{equation}\label{eq:technical007}
-\partial_t^2(\partial_t^\beta\Gamma_2^{\gamma})\tilde{\psi} + g_1(\phi, \nabla\phi)\partial_y^2(\partial_t^\beta\Gamma_2^{\gamma}\tilde{\psi}) + g_2(\phi, \nabla\phi)\partial_{ty}(\partial_t^\beta\Gamma_2^{\gamma}\tilde{\psi}) = \kappa(t, y)H,
\end{equation}
with 
\begin{eqnarray*}
H & = & -\frac{1}{2}\frac{3+\frac{y^2}{2}}{(1+y^2)^2}(\partial_t^\beta\Gamma_2^{\gamma}\tilde{\psi}) - \partial_t^\beta\Gamma_2^{\gamma}(P_dG) + \widetilde{\big(\partial_t^\beta\Gamma_2^{\gamma}(G)\big)} + \pr_t^\beta[\square,\Ga_2^\kappa]\tpsi  \\
&&+ \sum_{\tilde{\gamma}<\gamma}V_{\tilde{\gamma}}\partial_t^\beta\Gamma_2^{\tilde{\gamma}}\tilde{\psi}
\end{eqnarray*}
where  $ \widetilde{\big(\partial_t^\beta\Gamma_2^{\gamma}(G)\big)}$
denotes all non-top order terms, while the top order terms (i.\ e.\ when $\partial_t^\beta\Gamma_2^{\gamma}$ falls on a second derivative term in $G$) have been moved to the left. 
Note in particular that 
\begin{align*}
&g_1(\phi, \nabla\phi) = 1+O(\frac{\phi}{1+y^2} + [\nabla_{t,y}\phi]^2),\,g_2(\phi, \nabla\phi) = O(\frac{\phi}{1+y^2} + [\nabla_{t,y}\phi]^2),\\&\kappa(t,y) = 1+O(\frac{\phi}{1+y^2} + [\nabla_{t,y}\phi]^2)
 \end{align*}
Then we approximately factorize the left hand side of \eqref{eq:technical007} as follows:
\begin{align*}
&-\partial_t^2\tilde{\psi} + g_1(\phi, \nabla\phi)\partial_y^2\tilde{\psi} + g_2(\phi, \nabla\phi)\partial_{ty}\tilde{\psi}\\& = (-\partial_t - h_1(\phi, \nabla\phi)\partial_y)(\partial_t - h_2(\phi, \nabla\phi)\partial_y)\tilde{\psi}
 - h_1(\phi, \nabla\phi)\pr_y(h_2(\phi, \nabla\phi))\partial_y\tilde{\psi}
\\&- \partial_t(h_2(\phi, \nabla\phi))\partial_y\tilde{\psi}
\end{align*}
 where the functions $h_{1,2}$ are chosen to satisfy 
 \[
 -h_1 + h_2 = g_2(\phi, \nabla\phi),\,h_1 h_2 = g_1(\phi, \nabla\phi)
 \]
 whence 
 \[
 h_{1,2} = 1 + O(\frac{\phi}{1+y^2} + [\nabla_{t,y}\phi]^2).
 \]
 Hence we obtain from \eqref{eq:technical007} the relation 
 \begin{equation}\label{eq:technical8}\begin{split}
 &(-\partial_t - h_1(\phi, \nabla\phi)\partial_y)(\partial_t - h_2(\phi, \nabla\phi)\partial_y)\tilde{\psi}\\& = 
  h_1(\phi, \nabla\phi)(\partial_y h_2(\phi, \nabla\phi))\partial_y\tilde{\psi}+(\partial_t h_2(\phi, \nabla\phi))\partial_y\tilde{\psi} + H =:H_1.
 \end{split}\end{equation}
This is the equation we solve approximately via the method of characteristics. Precisely, introduce the functions $\lambda_{1,2}(s;t,y)$ via the ODEs
 \begin{equation}\label{eq:lambda1}
 \partial_s\lambda_1(s;t,y) = h_1(\phi, \nabla\phi)(s, \lambda_1(s;t,y)),\,\lambda_1(t;t,y) = y,
 \end{equation}
  \begin{equation}\label{eq:lambda2}
 \partial_s\lambda_2(s;t,y) = -h_2(\phi, \nabla\phi)(s, \lambda_2(s;t,y)),\,\lambda_2(t;t,y) = y.
 \end{equation}
Note that from our a priori bounds, we get the crude asymptotic
\[
\lambda_{1,2}(s;t,y) = y\mp(t-s) + O(\eps(t-s)^{\frac{1}{2}-\delta_1})
\]
Then we introduce the following approximate parametrix for the problem associated with \eqref{eq:technical8}: 
\begin{lem}\label{lem:aproxparametrix}
Let $f$, $g$ and $\tH$ three given scalar functions. Let $S[f,g,\tH]$ be defined by
\begin{eqnarray*}
S[f,g,\tH](t,y) &=& \frac{1}{2}\big[f(\lambda_1(0;t,y)) + f(\lambda_2(0;t,y)\big]\\ 
\nonumber&&+ \int_{\lambda_1(0;t,y)}^{\lambda_2(0;t,y)}\frac{g(\tilde{y})}{(h_1+h_2)(\phi, \nabla\phi)(0,\tilde{y})}\,d\tilde{y}\\
\nonumber&&+\int_0^t \int_{\lambda_1(s;t,y)}^{\lambda_2(s;t,y)}\frac{\tH(s, \tilde{y})}{(h_1+h_2)(s,\tilde{y})}\,d\tilde{y}ds.
\end{eqnarray*}
Then, we have
\begin{eqnarray*}
S[f,g,\tH](0,y) &=& f(y),\\
\partial_tS[f,g,\tH](0,y) &=& \big((h_2-h_1)(\phi, \nabla\phi)\big)(0,y)f'(y) + g(y),
\end{eqnarray*}
and  
\begin{eqnarray*}
(-\partial_t - h_1(\phi, \nabla\phi)\partial_y)(\partial_t - h_2(\phi, \nabla\phi)\partial_y)S[f,g,\tH](t, y)= \tH+E[f,g,\tH](t,y),
\end{eqnarray*}
where the error term $E[f,g,\tH](t,y)$ is given by
\begin{eqnarray*}
&& E[f,g,\tH](t,y)\\ 
&=& \big(\partial_t  h_2(\phi, \nabla\phi) - h_1(\phi, \nabla\phi)\partial_yh_2(\phi, \nabla\phi)+h_2(\phi, \nabla\phi)\partial_yh_1(\phi, \nabla\phi)\big)(t,y)\\
&&\times\partial_y\big(f(\lambda_1(0;t,y))\big)\\
&&+\frac{ \big(\partial_t  h_2(\phi, \nabla\phi) - h_1(\phi, \nabla\phi)\partial_yh_2(\phi, \nabla\phi)+h_2(\phi, \nabla\phi)\partial_yh_1(\phi, \nabla\phi)\big)(t,y)}{(h_1+h_2)(\phi,\nabla\phi)(t, \lambda_1(0;t,y))}\\
&&\times\partial_y\lambda_1(0;t,y)g(\lambda_1(0;t,y))\\
&&+\big(\partial_t  h_2(\phi, \nabla\phi) - h_1(\phi, \nabla\phi)\partial_yh_2(\phi, \nabla\phi)+h_2(\phi, \nabla\phi)\partial_yh_1(\phi, \nabla\phi)\big)(t,y)\\
&&\times\int_0^t \partial_y\lambda_1(s;t,y)\frac{\tH(s, \lambda_1(s;t,y))}{(h_1+h_2)(s,\lambda_1(s;t,y))}\,ds.
\end{eqnarray*}
\end{lem}

\begin{proof}
First, we trivially have 
\[
S[f,g,\tH](0,y) = f(y),
\]
as well as 
\[
\partial_tS[f,g,\tH](0,y) = \frac{1}{2}(\partial_t\lambda_1+\partial_t\lambda_2)(0;0,y)f'(y) + g(y),
\]
where we have exploited the fact that 
\[
\big(\partial_t + h_1(\phi, \nabla\phi)(t,y)\partial_y\big)\lambda_1(s;t,y) = 0,\,\big(\partial_t - h_2(\phi, \nabla\phi)(t,y)\partial_y\big)\lambda_2(s;t,y) = 0.
\]
Together with the fact that
\begin{equation}\label{eq:paraerror1}
\frac{1}{2}(\partial_t\lambda_1+\partial_t\lambda_2)(0;0,y)f'(y) = \big((h_2-h_1)(\phi, \nabla\phi)\big)(0,y)f'(y),
\end{equation}
we deduce
$$\partial_tu(0,y) = \big((h_2-h_1)(\phi, \nabla\phi)\big)(0,y)f'(y) + g(y).$$
Finally, the statement is done by direct check on the definition of $S[f,g,\tH]$. This concludes the proof of the lemma. 
\end{proof}

Next, we estimate $S[f,g,\tH]$ and $E[f,g,\tH]$.
\begin{lem}\label{lem:5} 
Assume that $\langle y\rangle(f', g)\in L^\infty_y$, and the decomposition 
$$\tH=\tH^{(1)}+\langle y\rangle^{-2}\tH^{(2)},$$
with
$$\sup_{t\in [0,T]}\langle t\rangle^{-2\cdot 10^{\kappa}\nu+1}\|\tH^{(1)}(t, \cdot)\|_{L_y^2}+\sup_{t\in [0,T]}\langle t\rangle^{-2\cdot 10^{\kappa}\nu}\big\|\langle y\log y\rangle^{-1}\tH^{(2)}\big\|_{L_{t,y}^2([0,t])}<+\infty.$$
Then, we have the following estimate for $S[f,g,\tH]$
\begin{eqnarray*}
&&\sup_{t\in [0,T]}\langle t\rangle^{-2\cdot
10^{\kappa}\nu}\big\|\langle y\log y\rangle^{-1}\nabla_{t,y} S[f,g,\tH]\big\|_{L_{t,y}^2([0,t])}\\
&\lesssim& \|\langle y\rangle(f',g)\|_{L^\infty_y}+\frac{1}{\nu}\sup_{t\in [0,T]}\langle t\rangle^{-2\cdot 10^{\kappa}\nu+1}\|\tH^{(1)}(t, \cdot)\|_{L_y^2}\\
&&+\sup_{t\in [0,T]}\langle t\rangle^{-2\cdot 10^{\kappa}\nu}\big\|\langle y\log y\rangle^{-1}\tH^{(2)}\big\|_{L_{t,y}^2([0,t])}.
\end{eqnarray*}
Furthermore, $E[f,g,\tH]$ satisfies the following decomposition
$$E[f,g,\tH]=E^{(1)}[f,g,\tH]+\langle y\rangle^{-2}E^{(2)}[f,g,\tH],$$
where $E^{(1)}[f,g,\tH]$ and $E^{(2)}[f,g,\tH]$ satisfy
\begin{eqnarray*}
&&\sup_{t\in [0,T]}\langle t\rangle^{-2\cdot 10^{\kappa}\nu+1}\|E^{(1)}[f,g,\tH](t, \cdot)\|_{L_y^2}\\
&&+\sup_{t\in [0,T]}\langle t\rangle^{-2\cdot 10^{\kappa}\nu}\big\|\langle y\log y\rangle^{-1}E^{(2)}[f,g,\tH]\big\|_{L_{t,y}^2([0,t])}\\
&\lesssim& \eps\|\langle y\rangle(f',g)\|_{L^\infty_y}+\sqrt{\eps}\sup_{t\in [0,T]}\langle t\rangle^{-2\cdot 10^{\kappa}\nu+1}\|\tH^{(1)}(t, \cdot)\|_{L_y^2}\\
&&+\sqrt{\eps}\sup_{t\in [0,T]}\langle t\rangle^{-2\cdot 10^{\kappa}\nu}\big\|\langle y\log y\rangle^{-1}\tH^{(2)}\big\|_{L_{t,y}^2([0,t])}.
\end{eqnarray*}
\end{lem}

\begin{proof} 
We start by proving the first bound of the lemma. Compute 
\begin{align*}
\nabla_{t,y}S[f,g,\tH](t,y)& = \frac{1}{2}\sum_{j=1,2}\nabla_{t,y}\lambda_j(0;t,y)f'(\lambda_j(0;t,y))\\&+\sum_{j=1,2}(-1)^j\nabla_{t,y}\lambda_j(0;t,y)\frac{g(\lambda_j(0;t,y))}{(h_1+h_2)(\phi, \nabla\phi)(0,\lambda_j(0;t,y))}\\
&+\sum_{j=1,2}(-1)^j\int_0^t\nabla_{t,y}\lambda_j(s;t,y)\frac{\tH(s,\lambda_j(s;t,y))}{(h_1+h_2)(\phi, \nabla\phi)(s,\lambda_j(s;t,y))}\,ds\\
& =:A+B+C.
\end{align*}
In order to estimate these terms, we need  pointwise bounds on $\nabla_{t,y}\lambda_j(s;t,y)$. By definition, we have the equation 
\[
\frac{\partial_{ts}^2\lambda_j(s;t,y)}{\partial_t\lambda_j(s;t,y)} = \pm\partial_y[h(\phi,\nabla\phi)](s, \lambda_j(s;t,y)).
\]
Also, we recall the schematic relation 
\[
\partial_y[h(\phi,\nabla\phi)] = O\big(\partial_y(\frac{\phi}{1+y^2}) + \partial_y([\nabla_{t,y}\phi]^2)\big).
\]
We need to check the absolute integrability of this expression with respect to $s$. First, it is readily verified (since $\partial_s\lambda_j\sim \pm1$) that 
\[
\int_0^T\big|\partial_y(\frac{\phi}{1+y^2})\big|(s, \lambda_j(s;t,y))\,ds \lesssim \eps.
\]
The expression $\partial_y[\nabla_{t,y}\phi]^2$ is a bit more delicate to control, since it fails logarithmically to be time integrable. In fact, we get 
\[
\big|\int_s^t\partial_y[\nabla_{t,y}\phi]^2(s_1, \lambda_j(s_1;t,y))\,ds_1\big|\lesssim \eps^2\log(\frac{\langle t\rangle}{\langle s\rangle}),
\]
and so we obtain the bound 
\begin{equation}\label{eq:lambda_jbound}
\big(\frac{\langle t\rangle}{\langle s\rangle}\big)^{-C\eps^2}\lesssim \big|\nabla_{t,y}\lambda_j(s;t,y)\big|\lesssim \big(\frac{\langle t\rangle}{\langle s\rangle}\big)^{C\eps^2}.
\end{equation}
Then using the bound
\[
\big|f'(\lambda_j(0;t,y))\big|+\big|g(\lambda_j(0;t,y))\big|\lesssim \sum_{\pm}\langle y\pm t+O(t^{\frac{1}{2}-\delta_1})\rangle^{-1}\|\langle y\rangle(f',g)\|_{L^\infty},
\]
it is immediately verified that 
\[
\big\|\langle\log y\rangle^{-1}\langle y\rangle^{-1}A\|_{L_{t,y}^2([0,T])}+\big\|\langle\log y\rangle^{-1}\langle y\rangle^{-1}B\|_{L_{t,y}^2([0,T])}\lesssim \|\langle y\rangle(f',g)\|_{L^\infty}.
\]
For the term $C$, first decompose $C$ as
$$C=C^{(1)}+C^{(2)}$$
according to the decomposition $\tH=\tH^{(1)}+\langle y\rangle^{-2}\tH^{(2)}$. We first estimate $C^{(1)}$. Write $\Lambda_j(s;t,y) = \lambda_j(s;t,y)$ if $s\leq t$ and 
\[
\Lambda_j(s;t,y) = y\mp(t-s),\,j = 1,2,\,s>t.
\]
Then we get (for $t\leq T$)
\[
|C^{(1)}|\lesssim \sum_{j=1,2}\int_0^T|\nabla_{t,y}\Lambda_j(s;t,y)|\big|\frac{\tH^{(1)}(s,\Lambda_j(s;t,y))}{(h_1+h_2)(\phi, \nabla\phi)(s,\Lambda_j(s;t,y))}\big|\,ds
\]
and by a simple change of variables argument and Minkowski's inequality, one obtains 
\begin{eqnarray*}
\big\|\langle \log y\rangle^{-1}\langle y\rangle^{-1}C^{(1)}\big\|_{L_{t,y}^2([0,T])}&\lesssim& \int_0^T \big(\frac{\langle T\rangle}{\langle s\rangle}\big)^{2C\eps^2}\big\|\tH^{(1)}(s, \cdot)\big\|_{L_y^2}\,dy\\
&\lesssim& \frac{\langle T\rangle^{2\cdot 10^{\kappa}\nu}}{\nu}\sup_{t\in [0,T]}\langle t\rangle^{-2\cdot 10^{\kappa}\nu+1}\|\tH^{(1)}(t, \cdot)\|_{L_y^2}.
\end{eqnarray*}
Next, we estimate $C^{(2)}$. Using the Cauchy-Schwarz inequality, we get
\begin{align*}
|C^{(2)}|(t, y)\lesssim \big(\int_0^T \big(\frac{\langle T\rangle}{\langle s\rangle}\big)^{2C\eps^2}\langle \log\Lambda_j(s;t,y)\rangle^{-2} \langle\Lambda_j(s;t,y)\rangle^{-2}(\tH^{(2)}(s, \Lambda_j(s;t,y))^2\,ds\big)^{\frac{1}{2}}
\end{align*}
provided $t\in [0,T]$. Using Fubini and a simple change of variables, we conclude
\begin{eqnarray*}
&&\big\|\langle y\log y\rangle^{-1}C^{(2)}\big\|_{L_{t,y}^2([0,T])}^2\\
&\lesssim& \int_y \langle y\log y\rangle^{-2}\big(\int_0^T\int_0^T \big(\frac{\langle T\rangle}{\langle s\rangle}\big)^{2C\eps^2}\langle \log\Lambda_j(s;t,y)\rangle^{-2} \langle\Lambda_j(s;t,y)\rangle^{-2}(\tH^{(2)}(s, \Lambda_j(s;t,y))^2\,ds dt\big)dy\\
&\lesssim &   \langle T\rangle^{2\cdot 10^{\kappa}\nu}\big(\int_y \langle y\log y\rangle^{-2}dy \big)\sup_{t\in [0,T]}\langle t\rangle^{-2\cdot 10^{\kappa}\nu}\big\|\langle y\log y\rangle^{-1}\tH^{(2)}\big\|_{L_{t,y}^2([0,t])}^2\\
&\lesssim & \langle T\rangle^{2\cdot 10^{\kappa}\nu}\sup_{t\in [0,T]}\langle t\rangle^{-2\cdot 10^{\kappa}\nu}\|\tH^{(2)}(t, \cdot)\|_{L_{t,y}^2([0,t])}.
\end{eqnarray*}
as desired. This establishes the first bound of the lemma.\\ 

Next, we consider the error term $E[f,g,\tH]$. As we did for $\nabla_{t,y}S[f,g,\tH]$, we decompose $E[f,g,\tH]$ in view of its definition as
$$E[f,g,\tH]=A+B+C^{(1)}+C^{(2)}$$
where $A$, $B$, $C^{(1)}$ and $C^{(2)}$ correspond respectively to the contribution of $f$, $g$, $\tH^{(1)}$ and $\tH^{(2)}$. For $A$ and $B$, we use the bound 
\[
\big|f'(\lambda_j(0;t,y))\big|+\big|g(\lambda_j(0;t,y))\big|\lesssim \sum_{\pm}\langle y\pm t+O(t^{\frac{1}{2}-\delta_1})\rangle^{-1}\|\langle y\rangle(f',g)\|_{L^\infty}.
\]
Then we infer 
\[
|A|+|B|\lesssim \big(\big|\nabla_{t,y}(\frac{\phi}{\langle y\rangle^{2}})\big| + \big|\nabla_{t,y}(\phi_{t,y}^2)\big|\big) \sum_{\pm}\langle y\pm t+O(t^{\frac{1}{2}-\delta_1})\rangle^{-1}\|\langle y\rangle(f',g)\|_{L^\infty},
\]
which together with the bootstrap assumption \eqref{eq:disp1boot} for $\phi_{t,y}$ yields
\begin{eqnarray*}
&&\big\||A|+|B|\big\|_{L_{t,y}^2([0,T])}\\
&\lesssim& \eps\|(\langle y\rangle^{-2}+\langle t\rangle^{-1})\sum_{\pm}\langle y\pm t+O(t^{\frac{1}{2}-\delta_1})\rangle^{-1}\|_{L_{t,y}^2([0,T])}\|\langle y\rangle(f',g)\|_{L^\infty}.\\
&\lesssim& \eps\|\langle y\rangle(f',g)\|_{L^\infty}.
\end{eqnarray*}

Next, we consider the contributions of $C^{(1)}$ and $C^{(2)}$. We have 
\begin{align*}
\big|C^{(1)}\big|(t, y)\lesssim \big(\big|\nabla_{t,y}(\frac{\phi}{\langle y\rangle^{2}})\big| + \big|\nabla_{t,y}(\phi_{t,y}^2)\big|\big)\sum_{j=1,2}\int_0^t\big(\frac{\langle t\rangle}{\langle s\rangle}\big)^{C\eps^2}|\tH^{(1)}(s, \Lambda_j(s;t,y))|\,ds
\end{align*}
and
\begin{align*}
\big|C^{(2)}\big|(t, y)\lesssim
\big(\big|\nabla_{t,y}(\frac{\phi}{\langle y\rangle^{2}})\big| +
\big|\nabla_{t,y}(\phi_{t,y}^2)\big|\big)\sum_{j=1,2}\int_0^t\big(\frac{\langle
t\rangle}{\langle s\rangle}\big)^{C\eps^2}|\langle \Lambda_j(s;t,y)\rangle^{-2}\tH^{(2)}(s, \Lambda_j(s;t,y))|\,ds.
\end{align*}
{\it{(a): Contribution of $C^{(1)}$.}}
First, consider the contribution of $\big|\nabla_{t,y}(\phi_{t,y}^2)\big|$. Estimating this factor by $\lesssim \eps^2\langle t\rangle^{-1}$ and using a straightforward change of variables (using \eqref{eq:lambda_jbound}), we obtain 
\begin{align*}
&\big\|\big|\nabla_{t,y}(\phi_{t,y}^2)\big|(t, \cdot)\sum_{j=1,2}\int_0^t\big(\frac{\langle t\rangle}{\langle s\rangle}\big)^{C\eps^2}|\tH^{(1)}(s, \lambda_j(s;t,y))|\,ds\big\|_{L^2_{dy}}\\
&\lesssim \eps^2\langle t\rangle^{-1}\int_0^t\big(\frac{\langle t\rangle}{\langle s\rangle}\big)^{2C\eps^2}\big\|\tH^{(1)}(s, \cdot)\big\|_{L^2_{dy}}\,ds\\
&\lesssim \eps\langle t\rangle^{2\cdot 10^{\kappa}\nu-1}\sup_{t\in [0,T]}\langle t\rangle^{-2\cdot 10^{\kappa}\nu+1}\|\tH^{(1)}(t, \cdot)\|_{L_y^2}
\end{align*}
which is as desired. For the contribution of $\nabla_{t,y}(\frac{\phi}{\langle y\rangle^{2}})$, we estimate 
\begin{align*}
&\big\|\langle\log y\rangle^{-1}\langle y\rangle^{-1}\phi(t, y)\int_0^t\big(\frac{\langle t\rangle}{\langle s\rangle}\big)^{C\eps^2}|\tH^{(1)}(s, \lambda_j(s;t,y))|\,ds\big\|_{L_{t,y}^2([0,T])}\\
&\lesssim \big\|\langle\log y\rangle^{-1}\langle y\rangle^{-1}\phi(t, y)\int_0^T\big(\frac{\langle T\rangle}{\langle s\rangle}\big)^{C\eps^2}|H_1(s, \Lambda_j(s;t,y))|\,ds\big\|_{L_{t,y}^2([0,T])}\\
&\lesssim \big\|\frac{\phi}{\langle y\rangle^{\frac{1}{2}}}\big\|_{L_t^2L_y^\infty}\int_0^T\big(\frac{\langle T\rangle}{\langle s\rangle}\big)^{2C\eps^2}\big\|H_1(s, \cdot)\big\|_{L_y^2}\,ds\\
&\lesssim \sqrt{\eps}\langle T\rangle^{2\cdot 10^{\kappa}\nu}\sup_{t\in [0,T]}\langle t\rangle^{-2\cdot 10^{\kappa}\nu+1}\|\tH^{(1)}(t, \cdot)\|_{L_y^2},
\end{align*}
again as required. 
\\

{\it{(b): Contribution of $C^{(2)}$. }} For the contribution of $\big|\nabla_{t,y}(\phi_{t,y}^2)\big|$, we get 
\begin{align*}
&\big\| \big|\nabla_{t,y}(\phi_{t,y}^2)\big|\int_0^t\big(\frac{\langle
t\rangle}{\langle s\rangle}\big)^{C\eps^2}|\langle \lambda_j(s;t,y)\rangle^{-2}\tH^{(2)}(s, \lambda_j(s;t,y))|\,ds\big\|_{L^2_{dy}}\\
&\lesssim
\big\|\big|\nabla_{t,y}(\phi_{t,y}^2)\big|\big(\int_0^t\big(\frac{\langle
t\rangle}{\langle s\rangle}\big)^{2C\eps^2}\big|\frac{\tH^{(2)}(s,
\lambda_j)|}{\langle\log (\lambda_j)\rangle\langle\lambda_j\rangle}\big|^2\,ds\big)^{\frac{1}{2}}\big\|_{L^2_{dy}}\\
&\lesssim \eps^2\langle t\rangle^{-1}\left(\int_0^t\big(\frac{\langle
t\rangle}{\langle s\rangle}\big)^{3C\eps^2}\big\|\langle y\log
y\rangle^{-1}\tH^{(2)}(s,y)\big\|^2_{L_y^2}ds\right)^{\frac{1}{2}},
\end{align*}
where we used Cauchy-Schwartz and a change of variable in $y$. Integrating by parts in $s$ so that the $s$ derivative falls on $\langle s\rangle^{-3C\eps^2}$, we deduce
\begin{align*}
&\big\| \big|\nabla_{t,y}(\phi_{t,y}^2)\big|\int_0^t\big(\frac{\langle
t\rangle}{\langle s\rangle}\big)^{C\eps^2}|\langle \lambda_j(s;t,y)\rangle^{-2}\tH^{(2)}(s, \lambda_j(s;t,y))|\,ds\big\|_{L^2_{dy}}\\
&\lesssim \eps\langle t\rangle^{2\cdot 10^{\kappa}\nu-1}\sup_{t\in [0,T]}\langle t\rangle^{-2\cdot 10^{\kappa}\nu}\big\|\langle y\log y\rangle^{-1}\tH^{(2)}\big\|_{L_{t,y}^2([0,t])}.
\end{align*}

Finally, for the contribution of $\nabla_{t,y}(\frac{\phi}{\langle y\rangle^{2}})$, we estimate
\begin{align*}
&\big\|\langle\log y\rangle^{-1}\langle y\rangle^{-1}\phi(t,
y)\int_0^t\big(\frac{\langle t\rangle}{\langle
s\rangle}\big)^{C\eps^2}|\langle
\lambda_j\rangle^{-2}\tH^{(2)}(s,\lambda_j)|\,ds\big\|_{L_{t,y}^2([0,T])}\\
&\lesssim \big\|\langle\log y\rangle^{-1}\langle y\rangle^{-1}\phi(t, y)\big(\int_0^T\big(\frac{\langle T\rangle}{\langle s\rangle}\big)^{2C\eps^2}|\frac{\tH^{(2)}(s, \lambda_j(s;t,y))}{\langle\log (\lambda_j(s;t,y))\rangle\langle\lambda_j(s;t,y)\rangle}|^2\,ds\big)^{\frac{1}{2}}\big\|_{L_{t,y}^2([0,T])}\\
&\lesssim \eps \left(\int_0^T\int_0^T\big(\frac{\langle T\rangle}{\langle s\rangle}\big)^{2C\eps^2}|\frac{\tH^{(2)}(s, \lambda_j(s;t,y))}{\langle\log (\lambda_j(s;t,y))\rangle\langle\lambda_j(s;t,y)\rangle}|^2\,dsdt\right)^{\frac{1}{2}}\\
&\lesssim \eps \left(\int_0^T\big(\frac{\langle T\rangle}{\langle s\rangle}\big)^{3C\eps^2}\big\|\langle y\log y\rangle^{-1}\tH^{(2)}(s,.)\big\|^2_{L_y^2}\,ds\right)^{\frac{1}{2}},
\end{align*}
where we used Cauchy-Schwartz, Fubini, and a change of variable in $t$. Integrating by parts in $s$ so that the $s$ derivative falls on $\langle s\rangle^{-3C\eps^2}$, we deduce
\begin{align*}
&\big\|\langle\log y\rangle^{-1}\langle y\rangle^{-1}\phi(t,
y)\int_0^t\big(\frac{\langle t\rangle}{\langle
s\rangle}\big)^{C\eps^2}|\langle
\lambda_j\rangle^{-2}\tH^{(2)}(s,\lambda_j)|\,ds\big\|_{L_{t,y}^2([0,T])}\\
&\lesssim \sqrt{\eps}\langle T\rangle^{2\cdot 10^{\kappa}\nu}\sup_{t\in [0,T]}\langle t\rangle^{-2\cdot 10^{\kappa}\nu}\big\|\langle y\log y\rangle^{-1}\tH^{(2)}\big\|_{L_{t,y}^2([0,t])}
\end{align*}
which is again as desired. This completes the proof of the lemma. 
\end{proof}

We are now in position to derive the desired bound for \eqref{eq:toporderlocalen}. Let 
\[ f_1(y)=\pr_t^\beta\Ga_2^\gamma\tpsi(0,y),\,\,
g_1(y)=\pr_t\pr_t^\beta\Ga_2^\gamma\tpsi(0,y),\,  \beta+\gamma=N_1,\]
\[f_j(y)=0,\, j\geq 2,\, g_j(y)=-\big((h_2-h_1)(\phi,
\nabla\phi)\big)(0,y)f_{j-1}'(y),\, j\geq 2,\]
$H_1$ is defined by \eqref{eq:technical8}, and
$$H_j(t,y)=-E[f_{j-1},g_{j-1},H_{j-1}](t,y),\,\, j\geq 2.$$
Note first that $f_1$ and $g_1$ satisfy in view of the assumptions on the initial data of $\tpsi$
$$\|\langle y\rangle(f_1', g_1)\|_{L^\infty}\lesssim \delta_0.$$
Also, $H_1$ is defined by \eqref{eq:technical8} satisfies 
$$H_1=H_1^{(1)}+\langle y\rangle^{-2} H^{(2)}_1,$$
where $H_1^{(1)}$ and $H_1^{(2)}$, in view of the bootstrap
assumptions on $\phi$ and the proof of \eqref{eq:localener} for the
case of non top order derivatives in part (1) of this section, verify
\[\sup_{t\in [0,T]}\langle t\rangle^{-2\cdot
10^{\kappa}\nu+1}\|H_1^{(1)}(t, \cdot)\|_{L_y^2}+\sup_{t\in
[0,T]}\langle t\rangle^{-2\cdot 10^{\kappa}\nu}\big\|\langle y\log
y\rangle^{-1}H_1^{(2)}\big\|_{L_{t,y}^2([0,t])}\lesssim \eps^2.\]
This is clear except for the second term amid the five terms
constituting $H$, and for this it will be an easy consequence of the
estimates below used to prove \eqref{eq:unstablebound}. Next, we
deduce in view of Lemma \ref{lem:5} that for $j\geq 1$, that
\[\|\langle y\rangle (f_j',g_j)\|_{L^\infty_y}\lesssim
\delta_0\eps^{\frac{j-1}{2}},\, j\geq 1.\]
Furthermore, we have a decomposition 
$$H_j=H_j^{(1)}+\langle y\rangle^{-2} H^{(2)}_j,\, j\geq 1,$$
where $H_j^{(1)}$ and $H_j^{(2)}$ verify
 $$\sup_{t\in [0,T]}\langle t\rangle^{-2\cdot 10^{\kappa}\nu+1}\|H_j^{(1)}(t, \cdot)\|_{L_y^2}+\sup_{t\in [0,T]}\langle t\rangle^{-2\cdot 10^{\kappa}\nu}\big\|\langle y\log y\rangle^{-1}H_j^{(2)}\big\|_{L_{t,y}^2([0,t])}\lesssim \eps^{2+\frac{j-1}{2}}.$$
Finally, we have the following estimate for $S[f_j,g_j, H_j]$
\begin{eqnarray*}
\sup_{t\in [0,T]}\langle t\rangle^{-2\cdot
10^{\kappa}\nu}\big\|\langle y\log y\rangle^{-1}\nabla_{t,y}S[f_j,g_j,H_j]\big\|_{L_{t,y}^2([0,t])}\lesssim \big(\delta_0+\frac{\eps^2}{\nu}\big) \eps^{\frac{j-1}{2}},\, j\geq 1.
\end{eqnarray*}
We deduce that the sum
\[u_\infty(t,y)=\sum_{j\geq 1}S[f_j,g_j,H_j]\]
converges and satisfies
\begin{eqnarray*}
\sup_{t\in [0,T]}\langle t\rangle^{-2\cdot
10^{\kappa}\nu}\big\|\langle y\log y\rangle^{-1}\nabla_{t,y}u_\infty\big\|_{L_{t,y}^2([0,t])}\lesssim \eps^{\frac{3}{2}}.
\end{eqnarray*}
Furthermore, in view of Lemma \ref{lem:aproxparametrix}, we have
\[u_\infty(0,y)=\pr_t^\beta\Ga_2^\gamma\tpsi(0,y),\,\,
\pr_tu_\infty(0,y)=\pr_t\pr_t^\beta\Ga_2^\gamma\tpsi,(0,y),\,\beta+\gamma=N_1\]
and 
$$(-\partial_t - h_1(\phi, \nabla\phi)\partial_y)(\partial_t - h_2(\phi, \nabla\phi)\partial_y)u_\infty(t,y)=H_1(t,y).$$
By uniqueness, we deduce
\[u_\infty=\pr_t^\beta\Ga_2^\gamma\tpsi,\, \beta+\gamma=N_1\]
and hence
\begin{eqnarray*}
\sup_{t\in [0,T]}\langle t\rangle^{-2\cdot
10^{\kappa}\nu}\big\|\langle y\log
y\rangle^{-1}\nabla_{t,y}\pr_t^\beta\Ga_2^\gamma\tpsi\big\|_{L_{t,y}^2([0,t])}\lesssim \eps^{\frac{3}{2}},\,  \beta+\gamma=N_1.
\end{eqnarray*}
This is the desired bound for the top order derivatives, which concludes the proof of \eqref{eq:localener}.


\section{Pointwise decay estimates}\label{sec:dispersive}

The goal of this section is to prove the decay estimates
\eqref{eq:disp1} and \eqref{eq:disp2}. Our key tool is the radiative
decay estimates from Proposition~\ref{eq:keyest1}. As usual, our point of departure is the schematic equation for $\tilde{\psi}$
\[
-\partial_t^{2}\tilde{\psi} + \partial_y^2\tilde{\psi}+ \frac{1}{2}\frac{3+\frac{y^2}{2}}{(1+y^2)^2}\tilde{\psi} = P_cG,
\]
\[
G = (1+y^2)^{\frac{1}{4}}F(\phi , \nabla\phi ,\nabla^2\phi ).
\]
Note that for the free evolution part of the equation, the estimates
\eqref{eq:disp1} and \eqref{eq:disp2} follow immediately from
Proposition~\ref{eq:keyest1} for any $\delta_1\in [0,\frac12]$. For
the inhomogeneous parts of the evolution we use Duhamel's principle
applied to Proposition~\ref{eq:keyest1} and conclude that we need to
bound
\begin{equation}\label{eq:duhamelformula} 
\int_0^T \jb{T - s}^{-\sigma} \left\| \jb{y}^\sigma
\nabla_{t,y}^\beta G \right\|_{L^1_y} ds \end{equation}
for $|\beta| \leq \frac{N_1}{2} + C$. We consider
\eqref{eq:duhamelformula} with $\sigma = \frac12$ for
\eqref{eq:disp1}, and $\sigma = \frac12 + \delta_1$ for
\eqref{eq:disp2}, using that $\jb{y}^{-1} \leq \jb{y}^{-\frac12 -
\delta_1}$. 

Now, take some $\gamma\in (0,\frac12]$ 
\emph{which is sufficiently small but can be chosen independently of 
$\nu$}, it suffices to show that for some $\delta_1 \leq \gamma$ the
following bound 
\begin{equation} \label{eq:nonlinearbounddisp}
\big\|\jb{T-t}^{-\sigma} \langle
y\rangle^{\frac{1}{2}+\gamma}\nabla_{t,y}^{\beta}G\big\|_{L_t^1
([0,T], L_y^1)}
\lesssim \epsilon^{\frac32} \jb{T}^{-\sigma} ,\quad |\beta|\leq
\frac{N_1}{2}+C
\end{equation}
holds for $\sigma \in \{\frac12, \frac12 + \delta_1\}$. 

To establish \eqref{eq:nonlinearbounddisp} we can write schematically,
using \eqref{eq:disp1boot},
\begin{equation}\label{eq:technical26}\begin{split}
\big|\langle y\rangle^{\frac{1}{2}+\gamma}\nabla_{t,y}^{\beta}G\big|& \lesssim \langle y\rangle^{\frac{1}{2}+\gamma}\sum_{\substack{|\beta_1|+|\beta_2|\leq |\beta|+2\\|\beta_1|\leq \frac{N_1}{2},\, \beta_{2}\geq 1}}\frac{|\nabla_{t,y}^{\beta_1}\phi||\nabla_{t,y}^{\beta_2}\tilde{\phi}|}{\langle y\rangle^{2}}\\
&+\langle y\rangle^{\frac{1}{2}+\gamma}C(\phi, \nabla\phi, \nabla^2\phi)
\end{split}\end{equation}
where $C(\phi, \nabla\phi, \nabla^2\phi)$ denotes the pure cubic nonlinear terms. 

To estimate the quadratic terms on the right of
\eqref{eq:technical26}, we put $\beta_1$ derivatives in $L^\infty$ and 
$\beta_2$ derivatives in $L^2$
\begin{equation}\label{eq:technical27}\begin{split}
&\big\|\langle
y\rangle^{\frac{1}{2}+\gamma}\sum_{\substack{|\beta_1|+|\beta_2|\leq
|\beta|+2\\|\beta_1|\leq \frac{N_1}{2},\,|\beta_2|\geq
1}}\frac{|\nabla_{t,y}^{\beta_1}\phi||\nabla_{t,y}^{\beta_2}\tilde{\phi}|}{\langle
y\rangle^{2}}\big\|_{L_y^1}\\
&\leq \varepsilon \jb{t}^{-1/2} \big\|
\sum_{1 \leq |\beta_2|\leq
|\beta|+2}\frac{|\nabla_{t,y}^{\beta_2}\tilde{\phi}|}{\langle
y\rangle^{\frac32 - \gamma}}\big\|_{L_y^1}\\
& \lesssim \varepsilon \jb{t}^{-1/2} \big\|
\sum_{1 \leq |\beta_2|\leq
|\beta|+2}\frac{|\nabla_{t,y}^{\beta_2}\tilde{\phi}|}{\langle
y\rangle^{1 - \gamma -}}\big\|_{L_y^2}
\end{split}\end{equation}
We split the $L^2$ integration into the regions $\{ y \ll t\}$ and
$\{y \gtrsim t\}$. For the latter, we observe that 
\begin{equation}
\big\|
\sum_{1 \leq |\beta_2|\leq
|\beta|+2}\frac{|\nabla_{t,y}^{\beta_2}\tilde{\phi}|}{\langle
y\rangle^{1 - \gamma -}}\big\|_{L_y^2(y \gtrsim t)} \lesssim 
\varepsilon \jb{t}^{-1 + \gamma + \nu +} 
\end{equation}
by \eqref{eq:ener1boot}. For the former, we use that when $y\ll t$
\begin{align*}
\big|\nabla_{t,y}^{\beta_2}\tilde{\phi}\big|\lesssim \sum_{|\tilde{\beta}_2|<|\beta_2|}t^{-1}\big|\nabla_{t,y}^{\tilde{\beta}_2}\Gamma\tilde{\phi}\big|
\end{align*}
with $\Gamma$ comprising both $\Gamma_{1,2}$. Then splitting $\Gamma_1\tilde{\phi} = (\Gamma_1\tilde{\phi})_1 + (\Gamma_1\tilde{\phi})_2$ as in Lemma~\ref{lem:subtle}, we have 
\begin{align*}
\big\|\sum_{|\tilde{\beta}_2|<|\beta_2|}t^{-1}\big|\nabla_{t,y}^{\tilde{\beta}_2}(\Gamma_1\tilde{\phi})_1\big|\big\|_{L^2(y\ll t)}\lesssim \eps \langle t\rangle^{-\frac{1}{2}-\delta_1}.
\end{align*}
Furthermore, we get (using also bootstrap assumption \eqref{eq:ener2boot})
\begin{align*}
\big\|\sum_{|\tilde{\beta}_2|<|\beta_2|}\big|\nabla_{t,y}^{\tilde{\beta}_2}(\Gamma_1\tilde{\phi})_2\big|\big\|_{L^2(y\ll t)} + \big\|\sum_{|\tilde{\beta}_2|<|\beta_2|}\big|\nabla_{t,y}^{\tilde{\beta}_2}\Gamma_2\tilde{\phi}\big|\big\|_{L^2(y\ll t)}\lesssim \eps \langle t\rangle^{100\nu}.
\end{align*}
This yields
\begin{equation}
\big\|
\sum_{1 \leq |\beta_2|\leq
|\beta|+2}\frac{|\nabla_{t,y}^{\beta_2}\tilde{\phi}|}{\langle
y\rangle^{1 - \gamma -}}\big\|_{L_y^2(y \ll t)} \lesssim 
\varepsilon \left( \jb{t}^{-\frac12 - \delta_1} + \jb{t}^{-1 + 100\nu}
\right).
\end{equation}
Thus we arrive at the decay bound
\[ \big\|\langle
y\rangle^{\frac{1}{2}+\gamma}\sum_{\substack{|\beta_1|+|\beta_2|\leq
|\beta|+2\\|\beta_1|\leq \frac{N_1}{2},\,|\beta_2|\geq
1}}\frac{|\nabla_{t,y}^{\beta_1}\phi||\nabla_{t,y}^{\beta_2}\tilde{\phi}|}{\langle
y\rangle^{2}}\big\|_{L_y^1} \lesssim \varepsilon^2 \jb{t}^{-1 -
\delta_1}.\]
Using that 
\begin{equation}\label{eq:technicalw01}
\begin{split} 
& \int_0^T \jb{T-t}^{-\sigma} \jb{t}^{-1 - \delta_1} dt \\
&\leq  \jb{T/2}^{-\sigma} \int_0^{T/2} \jb{t}^{-1 -\delta_1} dt +
\jb{T/2}^{-1 - \delta_1} \int_{T/2}^T \jb{T - t}^{-\sigma} dt \\
& \lesssim \jb{T}^{-\sigma} + \jb{T/2}^{-1 - \delta_1} \jb{T/2}^{1 -
\sigma} \lesssim \jb{T}^{-\sigma}.
\end{split}
\end{equation}
we obtain the desired control on the weighted nonlinear terms. 

It then suffices to consider the pure cubic terms, which we write schematically in the form 
\begin{equation}\label{eq:cubicweighted}
\langle y\rangle^{1+\gamma}\big((\partial_t\phi)^2\partial_y^2\phi - 2\partial_y\phi\partial_t\phi\partial_{ty}^2\phi + (\partial_y\phi)^2\partial^2_t\phi\big).
\end{equation}
This time, we shall have to take advantage of the full inherent
null-structure, i.e.\ cancellations between the various terms. We
start by absorbing weights by the factors, i.\ e.\ by replacing $\phi$ by $\tilde{\phi}$. Note that schematically 
\begin{equation}\label{eq:technical7}
(\partial_t\phi)^2\partial_y^2\phi\sim \langle y\rangle^{-\frac{3}{2}}(\partial_t\tilde{\phi})^2\partial_y^2\tilde{\phi} + \langle y\rangle^{-\frac{5}{2}}(\partial_t\tilde{\phi})^2\partial_y\tilde{\phi} + \langle y\rangle^{-\frac{7}{2}}(\partial_t\tilde{\phi})^2\tilde{\phi}.
\end{equation}
We claim that the contribution of the second and third term are
straightforward to handle, since they have favorable weights. In fact,
using \eqref{eq:disp1boot} we have 
\begin{align*}
\left\| \jb{y}^{1+\gamma} \left[\jb{y}^{-\frac52} (\partial\tilde{\phi})^3 +
\jb{y}^{-\frac72} (\partial\tilde{\phi})^2 \tilde\phi\right]
\right\|_{L^1_y} \leq \varepsilon \jb{t}^{-\frac12} \big\| \frac{\partial
\tilde{\phi}}{\jb{y}^{\frac12 - \gamma/2}} \big\|_{L^2_y}^2.
\end{align*}
Splitting again into $y\ll t$ and $y\gtrsim t$ regions, we see that
arguing as before we have 
\[ \big\| \frac{\partial
\tilde{\phi}}{\jb{y}^{\frac34 - \gamma/2}} \big\|_{L^2_y}^2 \lesssim
\varepsilon^2 \left( \jb{t}^{-1 - 2\delta} + \jb{t}^{-1 + \gamma
+ \nu}\right) \]
and so their contributions to \eqref{eq:nonlinearbounddisp} can be easily controlled. 

The remaining terms in \eqref{eq:cubicweighted} are treated similarly, and so we now reduce to estimating the following expression 
\begin{equation}\label{eq:cubicweighted1}
\big\|\langle
y\rangle^{-\frac{1}{2}+\gamma}\big((\partial_t\tilde{\phi})^2\partial_y^2\tilde{\phi}
-
2\partial_y\tilde{\phi}\partial_t\tilde{\phi}\partial_{ty}^2\tilde{\phi}
+
(\partial_y\tilde{\phi})^2\partial^2_t\tilde{\phi}\big)\big\|_{L_{y}^1}.
\end{equation}
In fact, if one uses the equation for $\tilde{\phi}$ to switch
$\tilde{\phi}_{tt}$ with $\tilde{\phi}_{yy}$ and thereby generating error terms at most as bad as the last term in \eqref{eq:technical7} (whose contribution we already bounded), it suffices to consider 
\begin{equation}\label{eq:cubicweighted2}
\big\|\langle y\rangle^{-\frac{1}{2}+\gamma}\big((\partial_t\tilde{\phi})^2\partial_t^2\tilde{\phi} - 2\partial_y\tilde{\phi}\partial_t\tilde{\phi}\partial_{ty}^2\tilde{\phi} + (\partial_y\tilde{\phi})^2\partial^2_y\tilde{\phi}\big)\big\|_{L_{y}^1}.
\end{equation}
At this point we remark that re-visiting the estimate
\eqref{eq:technicalw01} used to control \eqref{eq:nonlinearbounddisp},
it suffices to show that the above quantity can be controlled with
decay rate $\jb{t}^{-1-}$. To this end we write 
\begin{equation}\label{eq:nullform1}\begin{split}
&(\partial_t\tilde{\phi})^2\partial_t^2\tilde{\phi} - 2\partial_y\tilde{\phi}\partial_t\tilde{\phi}\partial_{ty}^2\tilde{\phi} + (\partial_y\tilde{\phi})^2\partial^2_y\tilde{\phi}\\
&=\tilde{\phi}_t\frac{\Gamma_2\tilde{\phi}\Gamma_2\tilde{\phi}_t - \Gamma_1\tilde{\phi}\Gamma_1\tilde{\phi}_t}{t^2-y^2} - \tilde{\phi}_y\frac{\Gamma_2\tilde{\phi}\Gamma_2\tilde{\phi}_y - \Gamma_1\tilde{\phi}\Gamma_1\tilde{\phi}_y}{t^2-y^2}.
\end{split}\end{equation}
Then we treat a number of different regions, beginning with 
\\

{\it{(I): interior of the light cone, $y\ll t$.}} We exploit that the preceding expression is in effect a 'nested double null-structure'. Indeed we can write 
\begin{eqnarray}\label{eq:almostthere}
\nonumber&&\tilde{\phi}_t\frac{\Gamma_2\tilde{\phi}\Gamma_2\tilde{\phi}_t - \Gamma_1\tilde{\phi}\Gamma_1\tilde{\phi}_t}{t^2-y^2} - \tilde{\phi}_y\frac{\Gamma_2\tilde{\phi}\Gamma_2\tilde{\phi}_y - \Gamma_1\tilde{\phi}\Gamma_1\tilde{\phi}_y}{t^2-y^2}\\
\nonumber&=&\Gamma_2\tilde{\phi}\frac{\tilde{\phi}_t(\Gamma_2\tilde{\phi})_t - \tilde{\phi}_y(\Gamma_2\tilde{\phi})_y}{t^2 - y^2} - \Gamma_1\tilde{\phi}\frac{\tilde{\phi}_t(\Gamma_1\tilde{\phi})_t - \tilde{\phi}_y(\Gamma_1\tilde{\phi})_y}{t^2 - y^2}\\
\nonumber&&-\Gamma_2\tilde{\phi}\frac{\tilde{\phi}_t\tilde{\phi}_t - \tilde{\phi}_y\tilde{\phi}_y}{t^2 - y^2} \\
\nonumber&=&\Gamma_2\tilde{\phi}\frac{\Gamma_2\tilde{\phi}(\Gamma_2^2\tilde{\phi}) - \Gamma_1\tilde{\phi}\Gamma_1\Gamma_2\tilde{\phi}}{[t^2 - y^2]^2} - \Gamma_1\tilde{\phi}\frac{\Gamma_2\tilde{\phi}\Gamma_2\Gamma_1\tilde{\phi} - \Gamma_1\tilde{\phi}\Gamma_1^2\tilde{\phi}}{[t^2 - y^2]^2}\\
&&-\Gamma_2\tilde{\phi}\frac{\Gamma_2\tilde{\phi}\Gamma_2\tilde{\phi} - \Gamma_1\tilde{\phi}\Gamma_1\tilde{\phi}}{[t^2 - y^2]^2}.
\end{eqnarray}
Consider the worst term, which is $\frac{(\Gamma_1\tilde{\phi})^2\Gamma_1^2\tilde{\phi}}{[t^2-y^2]^2}$. Our task is to estimate 
\[
\big\|\chi_{y\ll t}\langle y\rangle^{-\frac{1}{2}+\gamma}\nabla_{t,y}^{\beta}\frac{(\Gamma_1\tilde{\phi})^2\Gamma_1^2\tilde{\phi}}{[t^2-y^2]^2}\big\|_{L_{y}^1},\,|\beta|\leq \frac{N_1}{2}+C.
\]
The most delicate occurs when $|\beta| = 0$, which we deal with here, the other case being similar but simpler. Using Lemma~\ref{lem:subtle}, we have to estimate the expressions 
\[
\big\|\chi_{y\ll t}\langle y\rangle^{-\frac{1}{2}+\gamma}\frac{\Gamma_1\tilde{\phi}(\Gamma_1\tilde{\phi})_i(\Gamma_1^2\tilde{\phi})_j}{[t^2-y^2]^2}\big\|_{L_{y}^1},\,i,j\in \{1,2\}.
\]
Observe that we have by that same lemma
\begin{align*}
\chi_{y\ll t}\big|\Gamma_1\tilde{\phi}\big|& \leq\chi_{y\ll t}\big|(\Gamma_1\tilde{\phi})_1\big| +\chi_{y\ll t}\big|(\Gamma_1\tilde{\phi})_2\big|\\
&\leq\chi_{y\ll t}\big|(\Gamma_1\tilde{\phi})_1\big| + \chi_{y\ll t}\langle y\rangle^{\frac{1}{2}}\big(\int_0^y|\partial_y(\Gamma_1\tilde{\phi})_2|^2\,d\tilde{y}\big)^{\frac{1}{2}}\\
&\lesssim \eps\langle t\rangle^{\frac{1}{2}+100\nu}.
\end{align*}
Then when $j = 2$, we get 
\begin{align*}
\big\|\chi_{y\ll t}\langle y\rangle^{-\frac{1}{2}+\gamma}\frac{\Gamma_1\tilde{\phi}(\Gamma_1\tilde{\phi})_i(\Gamma_1^2\tilde{\phi})_2}{[t^2-y^2]^2}\big\|_{L_{y}^1}&
\lesssim \jb{t}^{-3+\gamma} \|\Gamma_1\tilde{\phi}\|_{L_y^\infty}\|(\Gamma_1\tilde{\phi})_i\|_{L_y^\infty}\|\frac{(\Gamma_1^2\tilde{\phi})_2}{\langle y\rangle}\|_{L_y^2}\\
&\lesssim \eps^3 \jb{t}^{-2+300\nu+\gamma}.
\end{align*}

On the other hand, if $j = 1$, then we obtain 
\begin{align*}
\big\|\chi_{y\ll t}\langle y\rangle^{-\frac{1}{2}+\gamma}\frac{\Gamma_1\tilde{\phi}(\Gamma_1\tilde{\phi})_i(\Gamma_1^2\tilde{\phi})_1}{[t^2-y^2]^2}\big\|_{L_{y}^1}&
\lesssim \jb{t}^{-4+\gamma}\|\Gamma_1\tilde{\phi}\|_{L_y^\infty}\|(\Gamma_1\tilde{\phi})_i\|_{L_y^\infty}\|(\Gamma_1^2\tilde{\phi})_1\|_{L_y^2}\\
&\lesssim \eps^3\jb{t}^{-\frac{3}{2}+\gamma+200\nu-\delta_1}. 
\end{align*}
The remaining terms above are more of the same. 
\\

{\it{(II): the region near the light cone; $y\sim t$}}. 
We split this into two terms, one restricted to the region very close
to the light cone, i.\ e.\ $|y-t|<\langle t\rangle^{-\delta_2}$, the other away from the light cone $|y-t|\geq\langle t\rangle^{-\delta_2}$. Here $\delta_2\gg\gamma>0$ is a small constant to be determined. We start with the latter case 
\\

{\it{(IIa): The estimate away from the light cone, $|y-t|\geq\langle
t\rangle^{-\delta_2}$.}} We further distinguish between a small
frequency and a large frequency case; this we accomplish using the 
\emph{standard Littlewood-Paley projectors} in the spatial variable 
$y$, which we denote by $P_{<s}$ and $P_{\geq s}$, and which are defined via a smooth cut-off
function using the standard Fourier transform (and should be be
confused with $P_c$ and $P_d$ defined relative to the \emph{distorted}
Fourier transform). 

Specifically, write 
\begin{equation}\label{eq:technical9}\begin{split}
&\tilde{\phi}_t\frac{\Gamma_2\tilde{\phi}\Gamma_2\tilde{\phi}_t - \Gamma_1\tilde{\phi}\Gamma_1\tilde{\phi}_t}{t^2-y^2} - \tilde{\phi}_y\frac{\Gamma_2\tilde{\phi}\Gamma_2\tilde{\phi}_y - \Gamma_1\tilde{\phi}\Gamma_1\tilde{\phi}_y}{t^2-y^2}\\
&=\tilde{\phi}_t\frac{\Gamma_2\tilde{\phi}P_{<t^{-\delta_3}}(\Gamma_2\tilde{\phi}_t) - \Gamma_1\tilde{\phi}P_{<t^{-\delta_3}}(\Gamma_1\tilde{\phi}_t)}{t^2-y^2}\\& - \tilde{\phi}_y\frac{\Gamma_2\tilde{\phi}P_{<t^{-\delta_3}}(\Gamma_2\tilde{\phi}_y) - \Gamma_1\tilde{\phi}P_{<t^{-\delta_3}}(\Gamma_1\tilde{\phi}_y)}{t^2-y^2}\\
&+\tilde{\phi}_t\frac{\Gamma_2\tilde{\phi}P_{\geq t^{-\delta_3}}(\Gamma_2\tilde{\phi}_t) - \Gamma_1\tilde{\phi}P_{\geq t^{-\delta_3}}(\Gamma_1\tilde{\phi}_t)}{t^2-y^2}\\& - \tilde{\phi}_y\frac{\Gamma_2\tilde{\phi}P_{\geq t^{-\delta_3}}(\Gamma_2\tilde{\phi}_y) - \Gamma_1\tilde{\phi}P_{\geq t^{-\delta_3}}(\Gamma_1\tilde{\phi}_y)}{t^2-y^2}\\
\end{split}\end{equation}
where $\delta_3\gg \delta_2$. We have 
\begin{align*}
&\big\|\chi_{t\sim y, |t-y|\geq t^{-\delta_2}}\langle y\rangle^{-\frac{1}{2}+\gamma}\tilde{\phi}_{t,y}\frac{\Gamma\tilde{\phi}P_{<t^{-\delta_3}}(\Gamma\tilde{\phi}_{t,y})}{t^2-y^2}\big\|_{L_{y}^1}\\
&\lesssim \jb{t}^{-1+\gamma}\log t\|\tilde{\phi}_{t,y}\|_{L_y^\infty}\|\chi_{y\sim t}\langle y\rangle^{-\frac{1}{2}}\Gamma\tilde{\phi}\|_{L_y^\infty}\|\chi_{y\sim t}P_{<t^{-\delta_3}}(\Gamma\tilde{\phi}_{t,y})\|_{L_y^\infty}
\end{align*}
where we have used the factor $(t-y)^{-1}$ to control the $L_y^1$-integral. Also, $\Gamma$ stands for either $\Gamma_1$ or $\Gamma_2$. On account of Remark~\ref{rem:1}, we have 
\[
\|\chi_{y\sim t}\langle y\rangle^{-\frac{1}{2}}\Gamma\tilde{\phi}\|_{L_y^\infty}\lesssim \eps \langle t\rangle^{100\nu}.
\]
On the other hand, from Bernstein's inequality, we get 
\[
\|\chi_{y\sim t}P_{<t^{-\delta_3}}(\Gamma\tilde{\phi}_{t,y})\|_{L_y^\infty}\lesssim \eps \langle t\rangle^{-\frac{\delta_3}{2}+100\nu}
\]
It follows that 
\begin{align*}
&\big\|\chi_{t\sim y, |t-y|\geq t^{-\delta_2}}\langle y\rangle^{-\frac{1}{2}+\gamma}\tilde{\phi}_{t,y}\frac{\Gamma\tilde{\phi}P_{<t^{-\delta_3}}(\Gamma\tilde{\phi}_{t,y})}{t^2-y^2}\big\|_{L_{y}^1}\\
&\lesssim \eps^3\jb{t}^{-1+\gamma+201\nu - \frac{\delta_3}{2}}\jb{\log
t}.
\end{align*}
This reduces things to the large frequency case, i.\ e.\ the last two expressions in \eqref{eq:technical9}. Here the idea is to again invoke the 'double null-structure' as in the right-hand side of \eqref{eq:almostthere}. This causes one technical complication as we need to commute frequency localizers and vector fields. Note that 
\[
[\Gamma_2, P_{\geq t^{-\delta_3}}]
\]
acts boundedly in the $L^2_{dy}$-sense. Also, we have 
\[
\big\|[\Gamma_1, P_{\geq t^{-\delta_3}}]\tilde{\phi}\big\|_{L_y^2}\lesssim t^{\delta_3}\big\|\tilde{\phi}_t\|_{L_y^2}.
\]
It follows that in order to bound the last two terms in
\eqref{eq:technical9}, we need to bound the following expressions: 
\begin{equation}\label{eq:technical10}\begin{split}
&\big\|\chi_{t\sim y, |t-y|\geq t^{-\delta_2}}\langle y\rangle^{-\frac{1}{2}+\gamma}\Gamma\tilde{\phi}\frac{\Gamma\tilde{\phi}P_{\geq t^{-\delta_3}}\Gamma^2\tilde{\phi}}{[t^2-y^2]^2}\big\|_{L_{y}^1}\\
&\big\|\chi_{t\sim y, |t-y|\geq t^{-\delta_2}}\langle y\rangle^{-\frac{1}{2}+\gamma}\Gamma\tilde{\phi}\frac{\Gamma\tilde{\phi}t^{\delta_3}(\Gamma\tilde{\phi})_t}{[t^2-y^2]^2}\big\|_{L_{y}^1},\\
\end{split}\end{equation}
where $\Gamma$ represents either $\Gamma_1$ or $\Gamma_2$. For the first expression, one writes formally
\[
P_{\geq t^{-\delta_3}}\Gamma^2\tilde{\phi}\lesssim t^{\delta_3}\partial_y\Gamma^2\tilde{\phi}.
\]
Keeping in mind the physical localization due to the cutoff $\chi_{t\sim y, |t-y|\geq t^{-\delta_2}}$ as well as Remark~\ref{rem:1}, and the bound 
\[
\|\chi_{y\sim t}\Gamma\tilde{\phi}\|_{L_y^\infty}\lesssim \eps\langle t\rangle^{\frac{1}{2}+100\nu}, 
\]
we bound the first term in \eqref{eq:technical10} by\begin{align*}
&\big\|\chi_{t\sim y, |t-y|\geq t^{-\delta_2}}\langle y\rangle^{-\frac{1}{2}+\gamma}\Gamma\tilde{\phi}\frac{\Gamma\tilde{\phi}P_{\geq t^{-\delta_3}}\Gamma^2\tilde{\phi}}{[t^2-y^2]^2}\big\|_{L_{y}^1}\\
&\lesssim \jb{t}^{-\frac{5}{2}+\gamma+\delta_2+\delta_3}\jb{\log t}\|\Gamma\tilde{\phi}\|_{L_y^\infty}^2\|\chi_{y\sim t}\partial_y\Gamma^2\tilde{\phi}\|_{L_y^\infty}\\
&\lesssim
\eps^3\jb{t}^{-\frac{3}{2}+\gamma+\delta_2+\delta_3+300\nu}\jb{\log
t}.
\end{align*}
The second term in \eqref{eq:technical10} is handled similarly. 
\\

{\it{(IIb): The estimate near the light cone, $|y-t|<\langle t\rangle^{-\delta_2}$.}} Here we work again with the 'intermediate null-fom expansion' as in the first line of  \eqref{eq:technical9}. Noting that schematically
\[
(\Gamma_1 - \Gamma_2)\tilde{\phi}\sim (t-y)\tilde{\phi}_{t,y},
\]
we get 
\begin{equation}\label{eq:technical11}\begin{split}
&\tilde{\phi}_t\frac{\Gamma_2\tilde{\phi}\Gamma_2\tilde{\phi}_t - \Gamma_1\tilde{\phi}\Gamma_1\tilde{\phi}_t}{t^2-y^2} - \tilde{\phi}_y\frac{\Gamma_2\tilde{\phi}\Gamma_2\tilde{\phi}_y - \Gamma_1\tilde{\phi}\Gamma_1\tilde{\phi}_y}{t^2-y^2}\\
&\sim \tilde{\phi}_{t,y}\frac{\tilde{\phi}_{t,y}\Gamma\tilde{\phi}_{t,y}}{t+y} + \tilde{\phi}_{t,y}\frac{\Gamma\tilde{\phi}\nabla_{t,y}^2\tilde{\phi}}{t+y}.
\end{split}\end{equation}
We then easily get the bound 
\begin{align*}
\big\|\chi_{|y-t|<t^{-\delta_2}}\langle
y\rangle^{-\frac{1}{2}+\gamma}\eqref{eq:technical11}\big\|_{L_{y}^1}\lesssim
\eps^3\jb{t}^{-1+\gamma-\delta_2+102\nu},
\end{align*}
which is admissible since we may arrange $\max\{\gamma, \nu\}\ll \delta_2$. 
\\

{\it{(III): exterior of the light cone, $y\gg t$}}. This is handled analogously to {\it{(I)}}. 
\\

This finally completes the proof of estimates \eqref{eq:disp1} and \eqref{eq:disp2}.


\section{Control over the unstable mode}\label{sec:existencea}


To complete the proof of Proposition~\ref{prop:Core}, we need to prove the existence of $a$ such that the estimates \eqref{eq:unstablebound} are satisfied. 

\begin{lem}\label{lemma:existencea}
Let $\underline{\phi}$ be any extension to $t\in[0,+\infty)$ of $\phi$ which satisfies the bootstrap assumptions  \eqref{eq:ener1boot}-\eqref{eq:localenerboot} on $t\in [0,+\infty)$. Let $b\in\R$ given by
$$b=\left(a+\langle \tphi_1, g_d\rangle+\frac{\langle \tphi_2, g_d\rangle}{k_d}-\frac{1}{k_d}\int_0^{+\infty}\langle (1+y^2)^{\frac{1}{4}}F(\underline{\phi}, \nabla\underline{\phi}, \nabla^2\underline{\phi})(s), g_d\rangle e^{-k_ds}ds\right)e^{k_dT}.$$
Then, there exists $a\in [-\eps^{\frac{3}{2}},\, \eps^{\frac{3}{2}}]$ such that
\begin{eqnarray*}
|b| \lesssim \eps^{\frac{3}{2}}\langle T\rangle^{-2}.
\end{eqnarray*}
\end{lem}

\begin{proof}
Note in view of the assumptions on the initial data, the bootstrap assumptions for $\phi$ and the exponential decay of $g_d$ that 
\begin{eqnarray*}
&&\left|\langle \tphi_1, g_d\rangle+\frac{\langle \tphi_2, g_d\rangle}{k_d}-\frac{1}{k_d}\int_0^{+\infty}\langle (1+y^2)^{\frac{1}{4}}F(\underline{\phi}, \nabla\underline{\phi}, \nabla^2\underline{\phi})(s), g_d\rangle e^{-k_ds}ds\right|\\
&\lesssim& \delta_0+\eps^2.
\end{eqnarray*}
We infer
\begin{equation}\label{eq:cavamacher}
\left|be^{-k_dT}-a\right|\lesssim \delta_0+\eps^2.
\end{equation}
Let us now consider the subsets $\mathcal{I}_{\pm}$ of $[-\eps^{\frac{3}{2}},\, \eps^{\frac{3}{2}}]$ defined by
\begin{eqnarray*}
\mathcal{I}_+ &=& \{a\in [-\eps^{\frac{3}{2}},\, \eps^{\frac{3}{2}}]\,/\,b>2 \eps^{\frac{3}{2}}\langle T\rangle^{-2}\},\\
\mathcal{I}_- &=& \{a\in [-\eps^{\frac{3}{2}},\, \eps^{\frac{3}{2}}]\,/\,b<-2 \eps^{\frac{3}{2}}\langle T\rangle^{-2}\}.
\end{eqnarray*}
In view of \eqref{eq:cavamacher} and the fact that we may always assume that $T$ satisfies 
$$e^{k_dT}> 4\langle T\rangle^2,$$ 
we immediately see that $\pm \eps^{\frac{3}{2}}\in \mathcal{I}_{\pm}$. Furthermore, by the continuity of the flow, $\mathcal{I}_{\pm}$ are clearly open. Thus, $\mathcal{I}_{\pm}$ are two open, nonempty and disjoint subsets of $[-\eps^{\frac{3}{2}},\, \eps^{\frac{3}{2}}]$. Hence, there exists $a\in [-\eps^{\frac{3}{2}},\, \eps^{\frac{3}{2}}]$ such that 
$$a\in [-\eps^{\frac{3}{2}},\, \eps^{\frac{3}{2}}]\setminus (\mathcal{I}_+\cup \mathcal{I}_-).$$ 
This concludes the proof of the lemma.
\end{proof}

For $a$ given by Lemma \ref{lemma:existencea}, we now prove \eqref{eq:unstablebound}. In view of the formula for $h$ of Lemma \ref{lemma:formluah} and the definition of $b$, we have
\begin{eqnarray*}
h(t) & =& b+\frac{1}{2k_d}\left(\int_t^{+\infty}\langle (1+y^2)^{\frac{1}{4}}F(\underline{\phi}, \nabla\underline{\phi}, \nabla^2\underline{\phi})(s), g_d\rangle  e^{-k_ds}ds\right)e^{k_dt}\nonumber\\
&+&\frac{1}{2}\left(a+\langle \tphi_1, g_d\rangle-\frac{\langle \tphi_2, g_d\rangle}{k_d}+\frac{1}{k_d}\int_0^t\langle (1+y^2)^{\frac{1}{4}}F(\phi, \nabla\phi, \nabla^2\phi)(s), g_d\rangle 
e^{k_ds}ds\right)e^{-k_dt}.\nonumber
\end{eqnarray*}
Let $\underline{h}$ given by
$$\underline{h}(t)=h(t)-b-\frac{1}{2}\left(a+\langle \tphi_1, g_d\rangle-\frac{\langle \tphi_2, g_d\rangle}{k_d}\right)e^{-k_dt}.$$ 
Then, $\underline{h}$ can be also written as
\begin{eqnarray}\label{eq:newformulah}
\underline{h}(t) &=& \frac{1}{2k_d}\left(\int_t^{+\infty}\langle (1+y^2)^{\frac{1}{4}}F(\underline{\phi}, \nabla\underline{\phi}, \nabla^2\underline{\phi})(s), g_d\rangle e^{-k_ds}ds\right)e^{k_dt}\\
&&+\frac{1}{2k_d}\left(\int_0^t\langle (1+y^2)^{\frac{1}{4}}F(\phi, \nabla\phi, \nabla^2\phi)(s), g_d\rangle 
e^{k_ds}ds\right)e^{-k_dt}.\nonumber
\end{eqnarray}

\begin{rem}
The point of introducing an extension $\underline{\phi}$ of $\phi$ to $t\in [0, +\infty)$ is to avoid boundary terms at $t=T$ when we will integrate by parts below in the formula \eqref{eq:newformulah} for $\underline{h}$.
\end{rem}

In view of Lemma \ref{lemma:existencea} and the assumptions on the initial data, it suffices to prove \eqref{eq:unstablebound} with $h$ replaced with $\underline{h}$. Using that 
\[
\big|F(\phi, \nabla\phi, \nabla^2\phi)\big|\lesssim \langle y\rangle^{-2}|\langle\nabla_{t,y}\rangle^2\phi|^2 + |\nabla_{t,y}\phi|^2|\nabla_{t,y}^2\phi|
\]
one immediately infers from \eqref{eq:newformulah} that 
\[
|\partial_t^{\beta}\underline{h}(t)|\lesssim \eps^2\langle
t\rangle^{-1-2\delta_1},\,\beta+1\leq N_1.
\]

For the weighted derivatives, we first have 
\begin{align*}
t\underline{h}'(t) &= \frac{t}{2}e^{k_d t}\int_t^{+\infty} e^{-k_d s}\langle (1+y^2)^{\frac{1}{4}}F(\phi, \nabla\phi, \nabla^2\phi)(s, \cdot), g_d\rangle \,ds\\
&-\frac{t}{2}e^{-k_d t}\int_0^t e^{k_d s}\langle (1+y^2)^{\frac{1}{4}}F(\phi, \nabla\phi, \nabla^2\phi)(s, \cdot), g_d\rangle \,ds\\
&=\frac{t}{2k_d}e^{k_d t}\int_t^{+\infty} (-\partial_s)(e^{-k_d s})\langle (1+y^2)^{\frac{1}{4}}F(\phi, \nabla\phi, \nabla^2\phi)(s, \cdot), g_d\rangle \,ds\\
&-\frac{t}{2k_d}e^{-k_d t}\int_0^t \partial_s(e^{k_d s})\langle (1+y^2)^{\frac{1}{4}}F(\phi, \nabla\phi, \nabla^2\phi)(s, \cdot), g_d\rangle \,ds\\
&=\frac{t}{2k_d}e^{k_d t}\int_t^{+\infty} e^{-k_d s} \partial_s\langle (1+y^2)^{\frac{1}{4}}F(\phi, \nabla\phi, \nabla^2\phi)(s, \cdot), g_d\rangle \,ds\\
&+\frac{t}{2k_d}e^{-k_d t}\int_0^t e^{k_d s} \partial_s\langle (1+y^2)^{\frac{1}{4}}F(\phi, \nabla\phi, \nabla^2\phi)(s, \cdot), g_d\rangle \,ds.
\end{align*}
Continuing in this vein, we get 
\begin{align*}
(t\partial_t)^2\underline{h} + t\underline{h}'(t) &= \frac{t^2}{2}e^{k_d t}\int_t^{+\infty} e^{-k_d s} \partial_s\langle (1+y^2)^{\frac{1}{4}}F(\phi, \nabla\phi, \nabla^2\phi)(s, \cdot), g_d\rangle \,ds\\
&-\frac{t^2}{2}e^{-k_d t}\int_0^t e^{k_d s} \partial_s\langle (1+y^2)^{\frac{1}{4}}F(\phi, \nabla\phi, \nabla^2\phi)(s, \cdot), g_d\rangle \,ds\\
&=\frac{t^2}{2k_d}e^{k_d t}\int_t^{+\infty} (-\partial_s)(e^{-k_d s})\partial_s\langle (1+y^2)^{\frac{1}{4}}F(\phi, \nabla\phi, \nabla^2\phi)(s, \cdot), g_d\rangle \,ds\\
&-\frac{t^2}{2k_d}e^{-k_d t}\int_0^t \partial_s(e^{k_d s}) \partial_s\langle (1+y^2)^{\frac{1}{4}}F(\phi, \nabla\phi, \nabla^2\phi)(s, \cdot), g_d\rangle \,ds
\end{align*}
and performing the integration by parts, we obtain 
\begin{align*}
(t\partial_t)^2\underline{h} + t\underline{h}'(t) &= \frac{t^2}{2k_d}e^{k_d t}\int_t^{+\infty} e^{-k_d s}\partial_s^2\langle (1+y^2)^{\frac{1}{4}}F(\phi, \nabla\phi, \nabla^2\phi)(s, \cdot), g_d\rangle \,ds\\
&+\frac{t^2}{2k_d}e^{-k_d t}\int_0^t e^{k_d s} \partial_s^2\langle (1+y^2)^{\frac{1}{4}}F(\phi, \nabla\phi, \nabla^2\phi)(s, \cdot), g_d\rangle \,ds.
\end{align*}
Then  note that 
\begin{align*}
&\frac{t^2}{2k_d}e^{k_d t}\int_t^{+\infty} e^{-k_d s}\partial_s^2\langle (1+y^2)^{\frac{1}{4}}F(\phi, \nabla\phi, \nabla^2\phi)(s, \cdot), g_d\rangle \,ds\\
&=\frac{1}{2k_d}e^{k_d t}\int_t^{+\infty} (\frac{t}{s})^2e^{-k_d s}s^2\partial_s^2\langle (1+y^2)^{\frac{1}{4}}F(\phi, \nabla\phi, \nabla^2\phi)(s, \cdot), g_d\rangle \,ds,
\end{align*}
\begin{align*}
&\frac{t^2}{2k_d}e^{-k_d t}\int_0^t e^{k_d s} \partial_s^2\langle (1+y^2)^{\frac{1}{4}}F(\phi, \nabla\phi, \nabla^2\phi)(s, \cdot), g_d\rangle \,ds\\
&=\frac{1}{2k_d}e^{-k_d t}\int_0^t (\frac{t}{s})^2e^{k_d s} s^2\partial_s^2\langle (1+y^2)^{\frac{1}{4}}F(\phi, \nabla\phi, \nabla^2\phi)(s, \cdot), g_d\rangle \,ds.
\end{align*}
Further, we have the identity 
\[
s^2\partial_s^2 \langle G, g_d\rangle  = \langle \Gamma_2^2G - \Gamma_2G -2y\partial_y\Gamma_2G + y^2\partial_y^2G+2y\pr_yG, g_d\rangle.
\]
The bounds \eqref{eq:unstablebound} are now a straightforward consequence of the structure of 
\[
F(\phi, \nabla\phi, \nabla^2\phi)
\]
and the bounds \eqref{eq:ener1boot} -  \eqref{eq:unstableboundboot}
for $\tilde{\phi}$. This concludes the proof of 
\eqref{eq:unstablebound}, and hence of Proposition \ref{prop:Core}.


\section{Proof of Theorem \ref*{thm:Main}}\label{sec:proofmaintheorem}


We are now in a position to conclude the proof of Theorem \ref{thm:Main}. In view of the choice of the initial data and by the continuity of the flow, note that the bootstrap assumptions are satisfied for some small $T>0$ and for any 
$$a\in [-\eps^{\frac{3}{2}},\, \eps^{\frac{3}{2}}].$$
Then, as a consequence of Proposition \ref{prop:Core}, we have that for any $T>0$, there exists $\eps>0$ small enough  and $a^{(T)}\in [-\eps^{\frac{3}{2}},\, \eps^{\frac{3}{2}}]$ such that the following estimates are satisfied on $t\in [0,T)$
\begin{equation}\label{eq:disp1th}
\|\nabla_{t,y}^{\beta}\phi\|_{L^\infty_{dy}}\lesssim \langle t\rangle^{-\frac{1}{2}},\,0\leq |\beta|\leq \frac{N_1}{2} + C,
\end{equation}
\begin{equation}\label{eq:disp2th}
\|\langle y\rangle^{-\frac{1}{2}}\nabla_{t,y}^{\beta}\phi\|_{L^\infty_{dy}}\lesssim \langle t\rangle^{-\frac{1}{2}-\delta_1},\,0\leq |\beta|\leq \frac{N_1}{2} + C.
\end{equation}
By compactness, we may extract a sequence $(t_n, a_n)$ such that  
$$(t_n)_{n\geq 0}\textrm{ is increasing }, t_n\to +\infty,\textrm{ and }a_n\to a\textrm{ as }n\to +\infty.$$
Then, let us call $\phi_n$ the solution corresponding to $a_n$ and $\phi$ the solution corresponding to $a$. Since the $\phi_n$ satisfy \eqref{eq:disp1th} and \eqref{eq:disp2th} on $[0,t_n)$ with the constants in $\lesssim$ being uniform in $n$, and since we have chosen $(t_n)_{n\geq 0}$ increasing with $t_n\to +\infty$, we deduce that $\phi$ is a global solution satisfying \eqref{eq:disp1th} and \eqref{eq:disp2th} on $[0,+\infty)$ and hence:
$$|\phi(t,.)|\lesssim \langle t\rangle^{-\frac{1}{2}}.$$
This concludes the existence part of the proof of Theorem \ref{thm:Main}.\\

Consider now the question of the Lipschitz continuity of $a$ with respect to the initial data. Let $\phi^{(1)}$ and $\phi^{(2)}$ two solutions corresponding respectively to parameters $a^{(1)}$ and $a^{(2)}$ and initial data $(\phi_1^{(1)}, \phi_2^{(1)})$ and $(\phi_1^{(2)}, \phi_2^{(2)})$ and let $h^{(1)}$ and $h^{(2)}$ the corresponding projections on $g_d$. Let us also denote
$$\Delta a=a^{(1)}-a^{(2)},\,\,(\Delta\phi_1, \Delta\phi_2)=(\phi_1^{(1)}-\phi_1^{(2)}, \phi_2^{(1)}-\phi_2^{(2)}),$$
and
$$\Delta \phi=\phi^{(1)}-\phi^{(2)},\,\, \Delta h=h^{(1)}-h^{(2)}.$$
$\phi^{(1)}$ and $\phi^{(2)}$ are obtained through the existence part of Theorem \ref{thm:Main} and are thus global and satisfy estimates \eqref{eq:ener1}-\eqref{eq:unstablebound} on $t\in [0,+\infty)$. Using these bounds together with the linear estimates of Propositions \ref{eq:keyest1}, \ref{eq:keyest2} and \ref{eq:keyest3}, we derive the following estimate for the difference $\Delta\phi$:
\begin{equation}\label{eq:eqofthedifference}
|\langle\nabla_{t,y}\rangle^2\Delta\phi(t,.)|\lesssim |\Delta a|+\|(\Delta\phi_1, \Delta\phi_2)\|_{X_0},\,\,t\in[0,+\infty).
\end{equation}
Furthermore, we have in view of Lemma \ref{lemma:formluah}
\begin{eqnarray*}
&&\Delta h(t)\\
&=& \frac{1}{2}\left(\Delta a+\langle \Delta\tphi_1, g_d\rangle+\frac{\langle \Delta\tphi_2, g_d\rangle}{k_d}-\frac{1}{k_d}\int_0^t\langle (1+y^2)^{\frac{1}{4}}\Delta F(\phi, \nabla\phi, \nabla^2\phi)(s), g_d\rangle e^{-k_ds}ds\right)e^{k_dt}\nonumber\\
&+&\frac{1}{2}\left(\Delta a+\langle \Delta\tphi_1, g_d\rangle-\frac{\langle \Delta\tphi_2, g_d\rangle}{k_d}+\frac{1}{k_d}\int_0^t\langle (1+y^2)^{\frac{1}{4}}\Delta F(\phi, \nabla\phi, \nabla^2\phi)(s), g_d\rangle 
e^{k_ds}ds\right)e^{-k_dt},\nonumber
\end{eqnarray*}
where 
$$\Delta F(\phi, \nabla\phi, \nabla^2\phi)=F(\phi^{(1)}, \nabla\phi^{(1)}, \nabla^2\phi^{(1)})-F(\phi^{(2)}, \nabla\phi^{(2)}, \nabla^2\phi^{(2)}).$$
Together with the estimates for $\phi^{(1)}$, $\phi^{(2)}$ and \eqref{eq:eqofthedifference}, we deduce
$$|\Delta a|\lesssim ( |\Delta a|+\|(\Delta\phi_1, \Delta\phi_2)\|_{X_0})e^{-k_d t}+\|(\Delta\phi_1, \Delta\phi_2)\|_{X_0}+( |\Delta a|+\|(\Delta\phi_1, \Delta\phi_2)\|_{X_0})^2.$$
We let $t\to +\infty$ which yields
$$|\Delta a|\lesssim \|(\Delta\phi_1, \Delta\phi_2)\|_{X_0}+( |\Delta a|+\|(\Delta\phi_1, \Delta\phi_2)\|_{X_0})^2$$
and hence
$$|\Delta a|\lesssim \|(\Delta\phi_1, \Delta\phi_2)\|_{X_0}$$
which implies the Lipschitz continuity of $a$ with respect to the initial data.\\ 

This concludes proof of Theorem \ref{thm:Main}.


\appendix



\section{Derivation of the equation of motion}\label{app:reduction}


As discussed in the beginning of Section \ref{sect:formulation}, we
consider the mapping depending on a scalar function $\phi(t,y)$
satisfying $\phi(t,y) = \phi(t,-y)$:
\[ (t,y,\omega) \mapsto \left( t,\jb{y} +
\frac{\phi(t,y)}{\jb{y}}, \sinh^{-1}y - \frac{y}{\jb{y}}
\phi(t,y),\omega\right)\]
and we ask that this mapping has vanishing mean curvature. We remind
our readers that we use the Japanese bracket notation $\jb{y} =
\sqrt{1 + y^2}$. Using that
the mean curvature is the first variation of the volume form, we can
derive the equation of motion by considering formally the
Euler-Lagrange equation associated to the volume density of the
pull-back metric. An elementary computation shows that for the mapping
above, the pull-back metric is
\begin{multline}
-(1 - \phi_t^2) \mathrm{d}t^2 + \left( 1 -\frac{2\phi}{\jb{y}^2} +
\frac{\phi^2}{\jb{y}^4} + \phi_y^2\right) \mathrm{d}y^2 \\
+ 2 \phi_t\phi_y \mathrm{d}t\mathrm{d}y + \left( 1 + y^2 + 2\phi +
\frac{\phi^2}{\jb{y}^2}\right)
\mathrm{d}\omega^2 
\end{multline}
whose associated volume element is 
\begin{equation}
\underbrace{\left( \jb{y} + \frac\phi{\jb{y}}\right)}_A \sqrt{ (1
- \phi_t^2)\vphantom{\frac\phi{\jb{y}^2}}\smash{\underbrace{(1
- \frac\phi{\jb{y}^2})^2}_{B^2}} + \phi_y^2}~ \mathrm{d}y~ \mathrm{d}t~
\mathrm{d}\omega~ .
\end{equation}
Using $L = A\sqrt{B^2 ( 1 - \phi_t^2) + \phi_y^2}$ as the Lagrangian
density, we obtain the Euler-Lagrange equations: 
\[ \frac{\delta L}{\delta\phi} = \frac{\partial}{\partial t}
\left(\frac{\delta L}{\delta \phi_t}\right) + \frac{\partial}{\partial y}
\left(\frac{\delta L}{\delta \phi_y}\right). \]
Let 
$$K = B^2(1-\phi_t^2) + \phi_y^2.$$ 
We have
\begin{align*}
A &= \jb{y} + \frac{\phi}{\jb{y}},
& A' &= \frac{1}{\jb{y}}, \\
B &= 1 - \frac{\phi}{\jb{y}^2}, &
B' &= - \frac{1}{\jb{y}^2},\\
\partial_tA &= A'\phi_t, & \partial_tB &= B' \phi_t, \\
\partial_yA &= A'\phi_y + \frac{y}{\jb{y}} -
\frac{y\phi}{\jb{y}^3} & \partial_yB &= B'\phi_y+
\frac{2y\phi}{\jb{y}^4}\\
&= A' \phi_y + y A' B, & &= B'\phi_y + 2yB'(B-1),
\end{align*}
and also
\begin{align*}
K' &= 2BB' - 2BB'\phi_t^2,\\
\partial_tK &= 2\phi_y \phi_{ty} + 2B\partial_tB -
2B\partial_t B\phi_t^2 -
2B^2\phi_t\phi_{tt}, \\
\partial_yK &= 2\phi_y\phi_{yy} + 2B\partial_yB -
2B\partial_yB \phi_t^2 - 2B^2\phi_t.
\phi_{ty}.
\end{align*}
The Euler-Lagrange equations become
\[
A'\sqrt{K} + A\frac{K'}{2\sqrt K} = -\frac{\partial}{\partial t}
\left[  \frac{AB^2 \phi_t}{\sqrt{K}} \right] +
\frac{\partial}{\partial y} \left[
\frac{A\phi_y}{\sqrt{K}}\right]
\]
which implies
\begin{align*}
ABB'\left[ 1 - (\phi_t)^2\right] K + A'K^2 &= K^{\frac32}
\frac{\partial}{\partial y} \left[
\frac{A\phi_y}{\sqrt{K}}\right] -
K^{\frac32}\frac{\partial}{\partial t}
\left[  \frac{AB^2 \phi_t}{\sqrt{K}} \right]\\
&= K\left[ \partial_yA \phi_y + A\phi_{yy}\right] -
\frac12 A\phi_y\partial_yK +
\frac12AB^2\phi_t\partial_tK\\
& \quad -K \left[ \partial_t A B^2 \phi_t +
2AB\partial_tB\phi_t + AB^2 \phi_{tt}\right]\\
&= K\left[ A'(\phi_y)^2 + yA'B\phi_y +
A\phi_{yy}\right] - A(\phi_y)^2
\phi_{yy}\\
&\quad - AB\partial_yB\phi_y +
AB\partial_yB\phi_y(\phi_t)^2 +
2 AB^2\phi_y\phi_t\phi_{ty} \\
&\quad + AB^3\partial_tB\phi_t - AB^3
\partial_tB(\phi_t)^3
-AB^4 (\phi_t)^2 \phi_{tt}\\ &\quad - KB\left[AB
\phi_{tt} + A'B (\phi_t)^2 + 2AB'(\phi_t)^2
\right].
\end{align*}
So we arrive at
\begin{align}
KABB' &+ A'K\left[ (\phi_y)^2 + B^2 -
B^2(\phi_t)^2\right] \nonumber\\
&= KA'(\phi_y)^2 + KyA'B\phi_y + KA
\phi_{yy} - KAB^2 \phi_{tt} -
KA'B^2(\phi_t)^2 \label{eqn:derA3}\\
&\quad -
KABB'(\phi_t)^2 - A(\phi_y)^2 \phi_{yy} +
  2AB^2\phi_y\phi_t\phi_{ty} -
AB^4(\phi_t)^2\phi_{tt} \nonumber\\
&\quad + ABB'\Big[ 2y(B-1)(\phi_y)(\phi_t)^2 -
(\phi_y)^2 - 2y(B-1)\phi_y \nonumber \\
&\qquad\qquad \qquad +
\underbrace{(\phi_y)^2(\phi_t)^2 +
B^2(\phi_t)^2 -
B^2(\phi_t)^4}_{K(\partial_t\phi)^2}\Big]\nonumber
\end{align}
and hence 
\begin{align}
KABB'  + KA'B^2 &= yKA'B\phi_y + KA\phi_{yy} -
KAB^2\phi_{tt} \nonumber \\
& \quad - A(\phi_y)^2 \phi_{yy} +
  2AB^2\phi_y\phi_t\phi_{ty} -
AB^4(\phi_t)^2\phi_{tt} \label{120614-2}\\
&\quad + ABB'\left[ 2y(B-1)(\phi_y)(\phi_t)^2 -
(\phi_y)^2 - 2y(B-1)\phi_y \right]. \nonumber
\end{align}
We deduce, after replacing $K$ by its definition
\begin{multline*}
\left[(\phi_y)^2 + B^2 -
B^2(\phi_t)^2\right]\left[ ABB' + A'B^2 -
yA'B\phi_y - A\phi_{yy} +
AB^2\phi_{tt}\right] \\
 = - A(\phi_y)^2 \phi_{yy} +
  2AB^2\phi_y\phi_t\phi_{ty} -
AB^4(\phi_t)^2\phi_{tt} \\
 + ABB'\left[ 2y(B-1)(\phi_y)(\phi_t)^2 -
(\phi_y)^2 - 2y(B-1)\phi_y \right]~.
\end{multline*}
This we can regroup, after collecting all terms depending on the
second derivatives, to get
\begin{multline*}
A B^2\left[ -\phi_{yy} + (\phi_t)^2 \phi_{yy} + (\phi_y)^2  \phi_{tt}
+ B^2\phi_{tt} - 2 \phi_y \phi_t \phi_{ty}\right] \\
= A'B(y\phi_y- B)\left[(\phi_y)^2 + B^2 -
B^2(\phi_t)^2\right]\\
 + ABB'\left[ \left(B^2 - \frac{2y\phi}{\jb{y}^2}\phi_y\right)\left((\phi_t)^2 -1\right)-
2(\phi_y)^2 \right]
\end{multline*}
from which we divide through by $\jb{y} B$ to obtain
\begin{align*}
\left( 1 - \frac{\phi^2}{\jb{y}^4}\right) &\big[ -\phi_{yy} +
(\phi_t)^2 \phi_{yy} + (\phi_y)^2  \phi_{tt}
+ B^2\phi_{tt} - 2 \phi_y \phi_t \phi_{ty}\big] \\
& = \frac{y\phi_y}{\jb{y}^2}\left[(\phi_y)^2 + B^2 (1-(\phi_t)^2)\right] - \frac{B}{\jb{y}^2} \left[ (\phi_y)^2 + B^2(1 -
(\phi_t)^2)\right]\\
&\quad - \left( \frac{1}{\jb{y}^2} + \frac{\phi}{\jb{y}^4}\right)
\left[ 
\frac{2y\phi}{\jb{y}^2}\phi_y\left(1- (\phi_t)^2 \right)- B^2(1 -
(\phi_t)^2) - 2(\phi_y)^2 \right],
\end{align*}
separating the quasilinear and semilinear contributions to the two
sides of the equality sign. 
The left hand side we see is precisely
\[ \phi_{tt} - \phi_{yy} + Q_2 + Q_3 + Q_4 \]
where $Q_*$ are defined in \eqref{eqs:mainQuasi}. The right hand side
exhibits some cancellations among the summands, and can be rewritten as
\begin{multline*}
\frac{y\phi_y}{\jb{y}^2}\left[(\phi_y)^2 + B^2 (1-(\phi_t)^2)\right] 
- \frac{1}{\jb{y}^2}\left[ \frac{2y\phi}{\jb{y}^2}\phi_y (1 -
  (\phi_t)^2) - (\phi_y)^2\right]\\ - \frac{\phi}{\jb{y}^4} \left[
\frac{2y\phi}{\jb{y}^2}\phi_y\left(1- (\phi_t)^2 \right)- 2(B^2-1) - 2
+ 2B^2(\phi_t)^2 - 3(\phi_y)^2 \right]~.
\end{multline*}
Reorganizing a little bit and picking out the terms, we see that the
above expression is equal to 
\[ \frac{y}{\jb{y}^2} \phi_y + \frac{2}{\jb{y}^4} \phi - S_2 - S_3 -
S_4 \]
where the semilinear terms $S_*$ are defined in \eqref{eqs:mainSemi}. 
With this we obtain \eqref{eq:Main}. 


\section{Technical lemmas}\label{app:technical}


\begin{lem}\label{lemma:appB}
We have
\[
\big\|\nabla_{t,y}^{\beta}\tilde{\psi}_{ttt}\big\|_{L_y^2(y\ll t)} + \big\|\nabla_{t,y}^{\beta}\tilde{\psi}_{tty}\big\|_{L_y^2(y\ll t)}\lesssim \eps\langle t\rangle^{(1+[\frac{2|\beta|}{N_1}])100\nu-2},\,|\beta|+2\leq N_1.
\]
One also gets the bound 
\[
\big\|\langle y\rangle^2\nabla_{t,y}^{\beta}\nabla_{t,y}^3\tilde{\psi}\big\|_{L_y^2(y\gg t)}\lesssim \eps\langle t\rangle^{(1+[\frac{2|\beta|}{N_1}])100\nu},\,|\beta|+2\leq N_1.
\]
\end{lem}

\begin{proof}
To prove this, write the equation for $\tilde{\psi}$ as usual in the form 
\[
-\partial_t^{2}\tilde{\psi} + \partial_y^2\tilde{\psi}+ \frac{1}{2}\frac{3+\frac{y^2}{2}}{(1+y^2)^2}\tilde{\psi} = P_cG.
\]
Compute 
\begin{align*}
\Gamma_2^2\tilde{\psi} &= t^2\tilde{\psi}_{tt} + y^2\tilde{\psi}_{yy} + 2ty\tilde{\psi}_{ty} + t\tilde{\psi}_t + y\tilde{\psi}_y\\
&=(t^2 + y^2)\tilde{\psi}_{tt} + y^2\big[P_cG - \frac{1}{2}\frac{3+\frac{y^2}{2}}{(1+y^2)^2}\tilde{\psi} \big]\\
&+ 2ty\tilde{\psi}_{ty} + \Ga_2\tilde{\psi}.
\end{align*}
By differentiating this equation, we obtain 
\begin{equation}\label{eq:technical13}\begin{split}
(t^2 + y^2)\tilde{\psi}_{ttt} + 2ty\tilde{\psi}_{tty}& = \big(\Gamma_2^2\tilde{\psi}\big)_t - y^2\big[P_cG - \frac{1}{2}\frac{3+\frac{y^2}{2}}{(1+y^2)^2}\tilde{\psi} \big]_t\\
&-2t\tilde{\psi}_{tt} - 2y\tilde{\psi}_{ty} - (\Ga_2\tilde{\psi})_t =:A,
\end{split}\end{equation}
\begin{equation}\label{eq:technical14}\begin{split}
(t^2 + y^2)\tilde{\psi}_{tty} + 2ty\tilde{\psi}_{tyy}& = \big(\Gamma_2^2\tilde{\psi}\big)_y - \big(y^2\big[P_cG - \frac{1}{2}\frac{3+\frac{y^2}{2}}{(1+y^2)^2}\tilde{\psi} \big]\big)_y\\
&-2y\tilde{\psi}_{tt} - 2t\tilde{\psi}_{ty} - (\Ga_2\tilde{\psi})_y =:B.
\end{split}\end{equation}
We can turn this into a linear system for the variables $\tilde{\psi}_{ttt}$, $\tilde{\psi}_{tty}$ by observing that 
\begin{align*}
 2ty\tilde{\psi}_{tyy} -  2ty\tilde{\psi}_{ttt} = 2ty\big[P_cG -  \frac{1}{2}\frac{3+\frac{y^2}{2}}{(1+y^2)^2}\tilde{\psi}\big]_t =:C.
\end{align*}
In order to prove the observation above, it now suffices to show that 
\[
\big\|\nabla_{t,y}^{\beta}A\big\|_{L_y^2(y\ll t)} + \big\|\nabla_{t,y}^{\beta}B\big\|_{L_y^2(y\ll t)} + \big\|\nabla_{t,y}^{\beta}C\big\|_{L_y^2(y\ll t)} \lesssim \eps\langle t\rangle^{(1+[\frac{2|\beta|}{N_1}])100\nu},\,|\beta|+2\leq N_1. 
\] 
Starting with $A$, the only delicate term is $y^2[P_cG]_t$, and here we may easily omit the $P_c$(as the weight $y^2$ gets absorbed by $g_d$ otherwise). 
Then the bound 
\[
\big\|y^2\nabla_{t,y}^{\beta}G_t\big\|_{L_y^2}\lesssim \eps\langle t\rangle^{(1+[\frac{2|\beta|}{N_1}])100\nu}
\]
is clear for all the weighted terms (with weight at least $\langle y\rangle^{-2}$). For the pure cubic terms, we reduce to 
\[
y^2\big(\phi_t^2\tilde{\psi}_{yy}\big)_t,\, y^2\big(\phi_y\phi_t\tilde{\psi}_{ty}\big)_t,\,y^2\big(\phi_y^2\tilde{\psi}_{tt}\big)_t
\]
as well as their derivatives. 
In each of these we gain one factor $t^{-1}$ by placing $(\phi_{t,y})^2$ into $L^\infty$, and an extra factor $t^{-1}$ by using  
\[
\big\|\tilde{\psi}_{tt}\big\|_{L^2_{y}(y\ll t)} + \big\|\tilde{\psi}_{ty}\big\|_{L^2_{y}(y\ll t)}\lesssim t^{-1}\big\|\nabla_{t,y}\langle\Gamma_2\rangle\tilde{\psi}\big\|_{L_y^2}
\]
Finally, the factor $\tilde{\psi}_{yy}$ may be replaced by $\tilde{\psi}_{tt}$ up to easily controllable errors, using the equation for $\tilde{\psi}$. 
\\
The same reasoning applies to the term $B$, except that now we also have the terms 
\[
y\tilde{\psi}_{tt},\,t\tilde{\psi}_{yt},
\]
which are controlled by 
\[
\big\|y\tilde{\psi}_{tt}\big\|_{L_y^2(y\ll t)} + \big\|t\tilde{\psi}_{yt}\big\|_{L_y^2(y\ll t)}\lesssim \big\|\nabla_{t,y}\langle\Gamma_2\rangle\tilde{\psi}\big\|_{L_y^2}.
\]
For term $C$, the new feature is the expression 
\[
2ty  \frac{1}{2}\frac{3+\frac{y^2}{2}}{(1+y^2)^2}\tilde{\psi}_t = 2y\frac{1}{2}\frac{3+\frac{y^2}{2}}{(1+y^2)^2}\big(\Gamma_2\tilde{\psi} - y\tilde{\psi}_y\big).
\]
Then, in view of Lemma~\ref{lem:4}, we obtain 
\begin{eqnarray*}
\big\|\nabla_{t,y}^{\beta}\big(y\frac{1}{2}\frac{3+\frac{y^2}{2}}{(1+y^2)^2}\big(\Gamma_2\tilde{\psi}\big)\big)\big\|_{L_y^2(y\ll t)}&\lesssim& \eps\langle t\rangle^{(1+[\frac{2|\beta|}{N_1}])10\nu}\|\langle y\rangle^{-\frac{1}{2}}\|_{L^2_y(y\ll t)}\\
&\lesssim& \eps\langle t\rangle^{(1+[\frac{2|\beta|}{N_1}])100\nu}.
\end{eqnarray*}
To obtain the second inequality of the lemma, the only new feature is the control of the weighted cubic terms above,
\[
y^2\big(\phi_t^2\tilde{\psi}_{yy}\big)_t,\, y^2\big(\phi_y\phi_t\tilde{\psi}_{ty}\big)_t,\,y^2\big(\phi_y^2\tilde{\psi}_{tt}\big)_t,
\]
in the region $y\gg t$. But we can schematically write 
\[
\big|y^2\big(\phi_t^2\tilde{\psi}_{yy}\big)_t\big|\lesssim \langle y\rangle^{-1}\big|\langle \nabla_{t,y}\rangle\Gamma\tilde{\phi}\langle\nabla_{t,y}\rangle\tilde{\phi}_t\langle\nabla_{t,y}\rangle\Gamma\tilde{\psi}_{t,y}\big|,\,y\gg t,
\]
where $\Gamma   = \Gamma_{1,2}$, and so we get 
\begin{align*}
\big\|y^2\big(\phi_t^2\tilde{\psi}_{yy}\big)_t\big\|_{L^2(y\gg t)}&\lesssim\big\|\langle y\rangle^{-1}\langle \nabla_{t,y}\rangle\Gamma\tilde{\phi}\big\|_{L^\infty(y\gg t)}\big\|\tilde{\phi}_t\big\|_{L^\infty}\big\|\langle\nabla_{t,y}\rangle\Gamma\tilde{\psi}_{t,y}\big\|_{L^2(y\gg t)}\\
&\lesssim \eps^3\langle t\rangle^{21\nu}
\end{align*}
The estimate for higher derivatives is similar. 
This concludes the proof of the lemma.
\end{proof}

\begin{lem}\label{lem:4} 
We have 
\[
\big|\nabla_{t,y}^{\beta}\Gamma_2^{\kappa}\tilde{\psi}\big|(t,y)\lesssim \eps\langle t\rangle^{(1+[\frac{2|\beta|}{N_1}])10^{\kappa}\nu}\langle y\rangle^{\frac{1}{2}},\,|\beta|+\kappa\leq N_1.
\]
\end{lem}

\begin{proof}
This follows immediately from the embedding $H^1(\R)\subset L^{\infty}(\R)$(without the factor $\langle y\rangle^{\frac{1}{2}}$), provided $|\beta|>0$. Hence assume now $|\beta| = 0$. Then the estimate follows from the fundamental theorem of calculus and Cauchy-Schwarz, provided we get a bound of the form 
\begin{equation}\label{eq:technical01}
\big|\Gamma_2^{\kappa}\tilde{\psi}(t, y_*)\big|\lesssim \eps\langle t\rangle^{10^{\kappa}\nu}
\end{equation}
for some $y_* = O(1)$. For this, consider the wave equation satisfied by $\Gamma_2^{\kappa}\tilde{\psi}$, which is 
\begin{align*}
-\partial_t^{2}\Gamma_2^{\kappa}\tilde{\psi} + \partial_y^2\Gamma_2^{\kappa}\tilde{\psi}+ \frac{1}{2}\frac{3+\frac{y^2}{2}}{(1+y^2)^2}\Gamma_2^{\kappa}\tilde{\psi} + l.o.t. = \Gamma_2^{\kappa}(P_cG),
\end{align*}
where we have (pointwise bound)
\[
|l.o.t.|\lesssim \sum_{0\leq\kappa_1<\kappa}\big|\langle y\rangle^{-2}\Gamma_2^{\kappa_1}\tilde{\psi}\big| + |P_c G|.
\]
By a simple calculation, we have 
\[
\big\|\Gamma_2^{\kappa}(P_cG)\big\|_{L^2_{dy}}\lesssim \eps\sum_{\tilde{\kappa}\leq \kappa}\big\|\langle y\rangle^{-2}\Gamma_2^{\tilde{\kappa}}\tilde{\phi} \big\|_{L^2_{dy}} + \eps\langle t\rangle^{10^{\kappa}\nu}.
\]
Also, in view of the bootstrap assumption \eqref{eq:ener2boot}, we have
$$\big\|\pr^2_{t,y}\Gamma_2^{\kappa}\tpsi\big\|_{L^2_{dy}}\lesssim  \eps\langle t\rangle^{10^{\kappa}\nu}.$$
Thus, using the previous bound and the wave equation satisfied by $\Ga_2^\kappa\tpsi$, we deduce
\begin{eqnarray*}
\big\|\langle y\rangle^{-2}\Gamma_2^{\kappa}\tilde{\psi}\big\|_{L^2_{dy}} &\lesssim &  \big\|\frac{1}{2}\frac{3+\frac{y^2}{2}}{(1+y^2)^2}\Gamma_2^{\kappa}\tilde{\psi} \big\|_{L^2_{dy}}\\
&\lesssim & \big\|\pr^2_{t,y}\Gamma_2^{\kappa}\tpsi\big\|_{L^2_{dy}}+\big\|\Gamma_2^{\kappa}(P_cG)\big\|_{L^2_{dy}}+l.o.t\\
&\lesssim & \eps\sum_{\tilde{\kappa}\leq \kappa}\big\|\langle y\rangle^{-2}\Gamma_2^{\tilde{\kappa}}\tilde{\phi} \big\|_{L^2_{dy}} +\eps\langle t\rangle^{10^{\kappa}\nu}.
\end{eqnarray*}
Using induction on $\kappa$, we obtain the bound 
\begin{equation}\label{eq:technical5}
\big\|\langle y\rangle^{-2}\Gamma_2^{\kappa}\tilde{\psi}\big\|_{L^2_{dy}}\lesssim  \eps\langle t\rangle^{10^{\kappa}\nu}.
\end{equation}
This implies the existence of a $y_*$ as in \eqref{eq:technical01}, proving the lemma.
\end{proof}


\end{document}